\documentclass[a4paper,11pt,reqno]{amsart}
\usepackage{amsmath,amsfonts,amsthm,amssymb,color}
\usepackage[T1]{fontenc}
\usepackage{pdfsync}
\usepackage{pstricks}
\usepackage{hyperref}
\usepackage{lmodern}
\usepackage[normalem]{ulem}
\usepackage{mathrsfs}

\numberwithin{equation}{section}

\newcommand{\ov}[1]{\overline{\raisebox{0pt}[0.9\height]{$#1$}}}

\newcommand{\bgg}{\mathbf{v}}

\newcommand{\vertiii}[1]{{\left\vert\kern-0.25ex\left\vert\kern-0.25ex\left\vert #1 
    \right\vert\kern-0.25ex\right\vert\kern-0.25ex\right\vert}}

\newcommand{\esp}{\ensuremath{\mathbb{E}}}

\newcommand{\R}{\mathbb R}
\newcommand{\N}{\mathbb N}

\newcommand{\Z}{\mathbb Z}

\newcommand{\bu}{\mathbf{u}}
\newcommand{\bv}{\mathbf{v}}
\newcommand{\bff}{\mathbf{F}}
\newcommand{\be}{\beta}
\newcommand{\bg}{\mathbf{g}}
\newcommand{\bp}{\mathbf{P}}
\newcommand{\bxi}{\mathbf{\xi}}
\newcommand{\1}{{\bf 1}}
\newcommand{\2}{{\bf 2}}

\newcommand{\cb}{\mathcal B}
\newcommand{\cac}{\mathcal C}
\newcommand{\cd}{\mathcal D}

\newcommand{\cf}{\mathcal F}
\newcommand{\ch}{\mathcal H}
\newcommand{\ci}{\mathcal I}
\newcommand{\cj}{\mathcal J}
\newcommand{\cl}{\mathcal L}

\newcommand{\cq}{\mathcal Q}
\newcommand{\cs}{\mathcal S}

\newcommand{\crr}{\mathcal R}

\newcommand{\struc}{\mathscr T}
\newcommand{\scal}{\mathfrak{s}}
\newcommand{\compac}{\mathfrak{K}}

\newcommand{\al}{\alpha}
\newcommand{\ga}{\gamma}
\newcommand{\gga}{\Gamma}

\newcommand{\ep}{\varepsilon}

\newcommand{\ka}{\kappa}
\newcommand{\la}{\lambda}

\newcommand{\si}{\sigma}

\newcommand{\vp}{\varphi}

\newcommand{\lcl}{\left\{}
\newcommand{\rcl}{\right\}}
\newcommand{\lln}{\left|}
\newcommand{\rrn}{\right|}

\newcommand{\uxi}{\pmb{\xi}}
\newcommand{\ueta}{\pmb{\eta}}

\newtheorem{theorem}{Theorem}[section]

\newtheorem{corollary}[theorem]{Corollary}

\newtheorem{definition}[theorem]{Definition}

\newtheorem{lemma}[theorem]{Lemma}

\newtheorem{proposition}[theorem]{Proposition}

\theoremstyle{remark}
\newtheorem{remark}[theorem]{Remark}
\theoremstyle{remark}

\newcommand{\bean}{\begin{eqnarray*}}
\newcommand{\eean}{\end{eqnarray*}}
\newcommand{\ben}{\begin{enumerate}}
\newcommand{\een}{\end{enumerate}}
\newcommand{\beq}{\begin{equation}}
\newcommand{\eeq}{\end{equation}}

\begin{document}

\date{\today}

\title{On a modelled rough heat equation}

\author{AUR{\'E}LIEN DEYA}

\keywords{Stochastic PDEs; Regularity structures theory; space-time Fractional Brownian motion; Rough paths theory.}

\subjclass[2000]{60H15, 60G15, 60G22}


\begin{abstract}
We use the formalism of Hairer's regularity structures theory \cite{hai-14} to study a heat equation with non-linear perturbation driven by a space-time fractional noise. Different regimes are observed, depending on the global pathwise roughness of the noise.

\smallskip

To this end, and following the procedure exhibited in \cite{hai-14}, the equation is first "lifted" into some abstract \emph{regularity structure} and therein solved through a fixed-point argument. Then we construct a consistent \emph{model} above the fractional noise, by relying on a smooth Fourier-type approximation of the process.
\end{abstract}

\maketitle

\

\

\section{Introduction and main result}

\subsection{Introduction}

Rough paths theory (\cite{friz-victoir-book,lyons-book}) is now robustly anchored in the probabilistic scenery and regarded as one of the most effective approaches to ordinary stochastic calculus, that is to systems of the form
\begin{equation}\label{rough-system}
dY_t=b(Y_t) \, dt+\si(Y_t) \cdot dX_t \ ,
\end{equation}
where $X$ stands for a finite-dimensional stochastic process. Among many other achievements, the theory offers a thorough and unprecedented treatment for a large class of Gaussian systems, beyond the classical semimartingale situation. The most insightful example of such an application is probably given by the fractional Brownian motion case $X=B^H$, which, for a Hurst index $H<1/2$, was essentially out of reach before rough paths came into the picture.

\smallskip

The extension of the rough paths method to SPDE settings is a much sparser field of investigation. If we focus on the standard heat equation model, the objective, in a very general form, would be to exhibit a natural pathwise interpretation of the equation: $u(0,.)=\psi$ and for all $t\geq 0$, $x\in \R^d$ (or $x$ in the $d$-dimensional torus), 
$$(\partial_t u)(t,x)=(\partial^2_x u)(t,x)+F_1(t,x,u(t,x),(\partial_x u)(t,x)) +F_2(t,x,u(t,x),(\partial_x u)(t,x)) \cdot \xi(t,x) \ , $$
equivalently
\begin{equation}\label{eq:intro}
u(t,x)=(G(t,.) \ast \psi)(x)+\big(G\ast \big[ F_1(.,u(.),(\partial_x u)(.))+F_2(.,u(.),(\partial_x u)(.)) \cdot \xi\big]\big)(t,x) \ ,
\end{equation}
when $\xi$ represents a stochastic space-time noise, $\psi$ a given initial condition, and $G$ refers to the heat kernel. Then as a particular spin-off, and motivated by the (one-parameter) rough-path results, we may expect to get new interpretations of the equation when considering a space-time fractional noise. 

\smallskip

Within the last few years, several rough-path-type methods have thus been developed in this direction for various types of $F_1,F_2$ and noise $\xi$. We can quote for instance \cite{caruana-friz-oberhauser,RHE,backward,friz-oberhauser} for extensive results in the particular situation where $\xi$ only depends on time, and more specifically when $\xi(t,x)=d\mathbf{w}_t$ for a given rough path $\mathbf{w}$. It is also in this context that a series of fundamental works by Hairer and a few co-authors has recently arisen \cite{hai-kpz,hai-rough,hai-approx,hai-burgers}, culminating with the so-called \emph{theory of regularity structures} \cite{hai-14} that will serve us as a guide throughout the paper. The aim of our study can indeed be summed up as follows: using the formalism of \cite{hai-14}, we wish to give a reasonable interpretation of (and simultaneously solve) the equation
\begin{equation}\label{intro:eq-frac}
(\partial_t u)(t,x)=(\partial^2_x u)(t,x)+F(x,u(t,x)) \cdot (\partial_t \partial_x X)(t,x) \quad , \quad u(0,.)=\psi \quad , \quad t\geq 0 \ , \ x\in \R \ ,
\end{equation}
where $F:\R \times \R \to \R$ is a smooth vector field and $X$ stands for a \emph{space-time fractional Brownian motion}. Let us be more specific about the latter designation: what we call here the \emph{space-time fractional Brownian motion}, or the \emph{fractional sheet}, with Hurst index $(H_1,H_2)\in (0,1)^2$ is the Gaussian process $\{X(t,x), \ t,x \in \R\}$ whose covariance function is given by the formula
\begin{equation}\label{brownian-sheet-cova}
\esp\big[ X(s,x) X(t,y) \big]=R_{H_1}(s,t) R_{H_2}(x,y) \ , \ \text{where} \ R_H(s,t)=\frac12 \{|s|^{2H}+|t|^{2H}-|t-s|^{2H}\} \ .
\end{equation}
When $H_1=H_2=\frac12$, $X$ is of course nothing but the classical Brownian sheet, so that the above noise $\xi=\partial_t\partial_x X$, understood in the sense of distributions, appears as a very natural fractional extension of the standard space-time white noise.

\smallskip

Before we go further regarding Equation (\ref{intro:eq-frac}), let us try to sketch out, at a very heuristic level, the main steps of the regularity structures procedure to handle (\ref{eq:intro}):

\smallskip

$(1)$ First, one associates the equation with a natural set of elements $Z(\xi)$, the so-called \emph{model}, constructed (when possible!) from the sole noise $\xi$, that is, independently of any potential "solution" $u$, and which extends to this multi-parametric setting the concept of a rough path. Indeed, just as its one-parameter counterpart, the model will prove to encode the whole stochastic dynamics of the equation, even though this phenomenon can only be revealed at the end of the procedure.

\smallskip

$(2)$ With a model in hand, one can "lift" the equation in a larger space $\struc$, the so-called \emph{regularity structure}, for which the successive operations involved in (\ref{eq:intro}), namely composition with $F$, multiplication with $\xi$ and convolution with $G$, all make perfectly sense. This is indeed the major difficulty when trying to give a pathwise interpretation of the equation: although the solution $u$ is expected to be a function, it is not clear at all how to give sense to the product
$$F(.,u(.),(\partial_x u)(.))\cdot \xi \ ,$$
even as a distribution. Lifting the equation in $\struc$ allows us to overcome this difficulty, at the expense of some mise en abyme of the equation. The procedure of turning an $\R$-valued process into a $\struc$-valued process can actually be compared with the controlled-path transformation of Gubinelli's theory (\cite{gubi,gubi-ramif}), where processes have to be artificially boosted with "derivatives" components of some abstract Taylor expansion. As in rough paths/controlled paths theory, a fundamental ingredient here lies in the extension of the standard Hölder topology to $\struc$-valued functions, which gives birth to the spaces of \emph{modelled distributions}.

\smallskip

$(3)$ Once endowed with a sufficiently regular solution $\bu$ for the problem in $\struc$, one can (fortunately) go back to the real world with the help of another central tool of the theory: the \emph{reconstruction operator} $\crr_{Z(\xi)}$, which associates $\bu$ with a real distribution $u=\crr_{Z(\xi)}\bu$ along a natural approximation procedure. The machinery turns out to be continuous with respect to the model $Z(\xi)$ picked at Step $(1)$, which in some way allows us to loop the procedure.

\smallskip

An important point is that this whole 3-step strategy can be made consistent with the rules of standard differential calculus: if $\xi$ happens to be a smooth process, then there exists a canonical model $Z(\xi)$ for which the resulting solution $u=\crr_{Z(\xi)}\bu$ coincides with the classical solution of the equation (understood in the Lebesgue sense). Combining this consistency result with the continuity properties of the procedure gives rise to very readable statements: given a smooth approximation $\xi^n$ of $\xi$, and provided the sequence of canonical models $Z(\xi^n)$ converges to an element $Z(\xi)$, then the sequence of classical solutions $u^n$ associated with $\xi^n$ converges to an element $u$. In turn, $u$ legitimately deserves to be called a \emph{solution} of the equation driven by $\xi$. 

\smallskip

Note that the above ordering $(1)$-$(2)$-$(3)$ is only very schematic and in fact, all the operations involved in this strategy tend to intermingle through highly technical considerations. Let us also insist on a few specificities (among others) of the approach:

\smallskip

$\bullet$ Distributions, and especially Besov distributions of negative order, are really at the core of the machinery, contrary to what one observes in the original rough paths pattern. In order to make these abstract spaces more easy to handle, the theory leans on the construction of sophisticated wavelets bases which very subtly account for the local behaviour of the processes under consideration. We will report an example of such a construction in the proof of Lemma \ref{GRR-gene}.

\smallskip

$\bullet$ The method is based on a multi-parametric formulation of the problem, in contrast with previous infinite-dimensional pathwise approaches of the equation (\cite{RHE,GLT,REE}). Otherwise stated, time and space variables are essentially considered at the same level and merge into a single variable $x\in \R^{d+1}$.

\smallskip

$\bullet$ The flexibility of the regularity structures formalism allows for possible combination with \emph{renormalization} procedures. Indeed, in certain situations where the above sequence $Z(\xi^n)$ of canonical models fails to converge, it may still be possible to renormalize it into some converging model $\hat{Z}(\xi^n)$, which in turn can be related to a specific renormalized equation. Such a scheme was for instance implemented in \cite{hai-kpz} for the KPZ equation, solving the long-standing issue of its well-posedness (see also \cite{hai-14} for two other examples). We will only see a small glimpse of these possibilities through the exhibition of an It{\^o}/Stratanovich-type correction of the equation.

\smallskip

With these (dense) considerations in mind, let us go back to our fractional equation (\ref{intro:eq-frac}). The smooth approximation $\xi^n$ of the noise that will serve us as a starting reference is derived from the following Fourier-type representation of the fractional sheet (see e.g. \cite{samo-taqqu}): for all $t,x\in \R$,
\begin{equation}\label{representation-sheet}
X(t,x)=c_{H_1,H_2}\int_{\R^2} \widehat{W}(d\xi,d\eta) \, \frac{e^{\imath t\xi}-1}{|\xi|^{H_1+\frac12}} \frac{e^{\imath x\eta}-1}{|\eta|^{H_2+\frac12}} \ , 
\end{equation}
for some constant $c_{H_1,H_2} >0$ and where $\widehat{W}$ is the Fourier transform of a space-time white noise in $\R^2$, defined on some filtered probability space $(\Omega,\mathcal{F},\bp)$. Then we introduce $\xi^n$ along the formula: $\xi^n=\partial_t \partial_x X^n$ with
\begin{equation}\label{approx-noise}
X^n(t,x)=c_{H_1,H_2}\int_{D_n} \widehat{W}(d\xi,d\eta) \, \frac{e^{\imath t\xi}-1}{|\xi|^{H_1+\frac12}} \frac{e^{\imath x\eta}-1}{|\eta|^{H_2+\frac12}} \ , 
\end{equation}
where we have set $D_n=[-2^{2n},2^{2n}]\times [-2^n,2^n]$. The reason of this choice mostly lies in the facilities to compute the moments of such a process, as we will see it through the considerations of Section \ref{sec:construction}. It is also worth mentioning that the very same approximation of $X$ has been used by Chouk and Tindel in \cite{chouk-tindel} to study subtle integration properties of the fractional sheet.

\smallskip

At this point, we are almost ready to state our main result. It only remains us to specify the class of vector fields $F$ covered by our analysis. In this context, the following "space-localization" condition can be compared with the assumptions that prevail in \cite{RHE}. It is essentially meant to counterbalance the non-compacity of the space domain under consideration, namely the whole space $\R$ (see also point \textbf{(3)} of Section \ref{subsec:further}).

\begin{definition}\label{def:vector-field}
Given a compact set $\compac\subset \R$, we say that a function $F:\R^2 \to \R$ belongs to the class $\cac_\compac^\infty(\R^2)$ if it admits bounded partial derivatives of any order and if $F(x,y)=0$ for every $(x,y)\in (\R\backslash \compac) \times \R$. 
\end{definition}

We denote by $\cl$ the (one-dimensional) heat operator, that is $\cl=\partial_t-\partial^2_x$. Also, for $\ga\in (0,1)$, we denote by $\cac^\ga(\R)$ the space of $\ga$-Hölder functions on $\R$, and we refer the reader to Section \ref{subsec:spaces} for the definition of the parabolic Hölder spaces $\cac^\ga_\scal(\compac)$, $\compac \subset \R^2$. With these notations in hand, the main findings of the paper can be summed up as follows.


\begin{theorem}\label{main-theo}
For every $H_1,H_2\in (0,1)$, consider the smooth approximation $X^n$ of the fractional sheet with Hurst index $(H_1,H_2)$ given by (\ref{approx-noise}).
Fix the non-linearity $F$ within the class $\cac_{\compac}^\infty(\R^2)$ defined above, for some compact set $\compac\subset \R$, and consider a sequence of bounded deterministic initial conditions $\Psi^n $ that converges in $L^\infty (\R)$ to some element $\Psi$. Then we have the following (non-exhaustive) regimes.

\smallskip

\noindent
$(i)$ Assume that $2H_1+H_2>2$. Then, almost surely, there exists a time $T>0$ such that, as $n$ tends to infinity, the sequence $(Y^n)$ of solutions to the equation
\begin{equation}\label{eq-base-young}
\left\{\begin{array}{ccl}
(\cl Y^n) (t,x) &= &F(x,Y^n(t,x)) \cdot \partial_t\partial_x X^n (t,x)\quad ,\\
Y^n(0,x) &= &\Psi^n(x) \ , 
\end{array}\right.
\end{equation}
converges in $L^\infty([0,T]\times \R)$ to a limit $Y$. Besides, for any $s\in (0,T)$, the latter convergence also holds in the space $\cac^\ga_\scal([s,T] \times \compac)$, for every compact set $\compac \subset \R$ and every $\ga \in (3-2H_1-H_2,1)$. 

\smallskip

\noindent
$(ii)$ Assume that $2\geq 2H_1+H_2 > \frac53$. Then, almost surely, there exists a time $T>0$ and a sequence of positive reals $(C_{H_1,H_2}^n)$
such that, as $n$ tends to infinity, 
\begin{equation}\label{estim-cstt}
C^n_{H_1,H_2}\sim 
\left\lbrace
\begin{array}{ll}
c^1_{H_1,H_2}\cdot  2^{2n(2-2H_1-H_2)}& \quad \text{if} \quad \frac53 < 2H_1+H_2 <2 \ ,\\
c^2_{H_1,H_2}\cdot  n & \quad \text{if} \quad 2H_1+H_2 =2 \ ,
\end{array}
\right.
\end{equation}
for some constants $c^1_{H_1,H_2},c^2_{H_1,H_2}$, and the sequence $Y^n$ of solutions to the (renormalized) equation
\begin{equation}\label{eq-base}
\left\{\begin{array}{ccl}
(\cl Y^n) (t,x) &= &F(x,Y^n(t,x)) \cdot \partial_t\partial_x X^n (t,x)-C^n_{H_1,H_2}\cdot F(x,Y^n(t,x))\cdot \partial_2F(x,Y^n(t,x))\ ,\\
Y^n(0,x) &= &\Psi^n(x) \ , 
\end{array}\right.
\end{equation}
converges in $L^\infty([0,T]\times \R)$ to a limit $Y$. Besides, for any $s\in (0,T)$, the latter convergence also holds in the space $\cac^\ga_\scal([s,T] \times \compac)$, for every compact set $\compac \subset \R$ and every $\ga \in (\frac23,-1+2H_1+H_2)$. Finally, when $H_1=\frac12$ and $H_2 > \frac23$, the limit process $Y$ almost surely coincides with the solution of the equation
\begin{equation}\label{ito-equation}
\partial_t Y=\partial^2_x Y+F(.,Y) \, \partial_t\partial_x X \ , \ Y_0=\psi \ ,
\end{equation}
understood in the classical Itô sense (see Section \ref{sec:identif} for further details).

\end{theorem}

Thus, as the global pathwise smoothness of the noise decreases (if we consider $(H_1,H_2)\approx (1,1)$ as the "starting" point), a change of regime is to occur at the frontier $2H_1+H_2=2$, with the emergence of some explosion phenomenon to be corrected with a specific drift term. As far as we know, such a behaviour has not been observed in the fractional-SPDEs literature up to now. 

\smallskip

It is worth noting it right now: the difficulty of Theorem \ref{main-theo} is concentrated in point $(ii)$, the proof of which will occupy most of the paper. On the opposite, the arguments toward point $(i)$ will be condensed in a few lines. Let us briefly elaborate on this organization.

\smallskip

In Section \ref{sec:mod-equation}, we review in detail, under the assumptions of point $(ii)$, the successive results toward the "lift" of the equation, following the lines of \cite{hai-14}. Thus, the procedure morally corresponds to the steps $(2)$ and $(3)$ of the above-described machinery: we recall how the equation can be transposed into an appropriate regularity structure (endowed with a suitable model) and therein solved with a basic fixed-point argument. The reasoning will in particular enable us to exhibit the central components of the model in this situation, that we have gathered within the concept of an \emph{$(\al,K)$-rough path}. Note that the theoretical considerations of Section \ref{sec:mod-equation} are independent of the choice of the approximation $\xi^n$ of the noise.

\smallskip

In Section \ref{sec:construction}, and in a somewhat retrospective manner, we will focus on the construction of the model, or equivalently of the $(\al,K)$-rough path, above the fractional noise. We obtain it as the limit of the (renormalized) canonical model associated with the smooth approximation $\xi^n=\partial_t\partial_x X^n$. The argument relies on a distributional variant of the classical Garsia-Rodemich-Rumsey Lemma (Lemma \ref{GRR-gene}), combined with suitable moments estimates. Together with the results of Section \ref{sec:mod-equation}, it will lead us to the proof of the convergence property in Theorem \ref{main-theo}, point $(ii)$ (as summed up in Section \ref{subsec:proof-point-ii}). We will then devote Section \ref{sec:identif} to the proof of the identification statement when $H_1=\frac12$, by relying on arguments borrowed from \cite{hairer-pardoux}.

\smallskip

Section \ref{sec:young} will consist in a survey of the proof of the (much easier) point $(i)$, that we call the \emph{Young case} in analogy with its one-parameter counterpart (morally, if $H_2\approx 1$, the condition indeed reduces to $H_1 >\frac12$). Finally, the appendix contains details regarding the proof of a technical result in Section \ref{sec:mod-equation}.

\

\noindent
\textbf{Acknowledgements.} I am deeply grateful to two anonymous referees: their remarks and stimulating questions have led to numerous significant improvements, both in the content and in the presentation of my results. I also thank David Nualart for his help during my bibliographic searches.

\subsection{Further work}\label{subsec:further}
We are aware that the results in this paper are only a first step toward a thorough understanding of the fractional heat model (\ref{intro:eq-frac}) through the machinery of regularity structures. Let us try to record a few natural questions (amongst many others) that arise from the statement of Theorem \ref{main-theo}, and the study of which we postpone to future works for the sake of conciseness.

\smallskip

\noindent
\textbf{(1)} The two situations described in Theorem \ref{main-theo} only cover the domain $2H_1+H_2 > \frac53$. Our guess is that the subsequent strategy could be extended (through highly technical constructions) up to the frontier $2H_1+H_2=1$. Actually, this extension would certainly require to consider successive slices of $[0,1]^2$ of the form $\frac{k+2}{k}\geq 2H_1+H_2 >\frac{k+3}{k+1}$ ($k\geq 1$), to be compared with the usual rough paths splitting $H\in (\frac{1}{k+1},\frac{1}{k}]$: a slice of higher order appeals to more a priori constructions, that is to a more sophisticated model. In parallel, the equation is of course likely to involve additional renormalization terms.

It is worth mentioning here that the particular case of a space-time white noise in (\ref{intro:eq-frac}) (equivalently, take $X$ a Brownian sheet, i.e., $H_1=H_2=\frac12$) has been recently treated by Hairer and Pardoux \cite{hairer-pardoux}, in a setting which only slightly differs from ours (the equation is therein studied on the torus, with approximation given by a mollifying procedure). As expected, the authors have to consider a much richer mode: with the above splitting of $[0,1]^2$ in mind ($\frac{k+2}{k}\geq 2H_1+H_2 >\frac{k+3}{k+1}$), the situation they focus on corresponds to a "fourth-slice" example, that is to $k=4$.

\smallskip

\noindent
\textbf{(2)} The results of Theorem \ref{main-theo} could certainly be extended to a $(d+1)$-parameter fractional Brownian motion (defined along a natural extension of (\ref{brownian-sheet-cova})), by replacing the condition $2H_1+H_2 >2$ (respectively $2\geq 2H_1+H_2 > \frac53$) with $2H_1+\sum_{i=2}^{d+1} H_i >d+1$ (respectively $d+1\geq 2H_1+\sum_{i=2}^{d+1}H_i > d+\frac23$). We have not checked the numerous technical details behind this assertion though.

\smallskip

\noindent
\textbf{(3)} The equation we study here is defined on a non-compact space domain, which, at first sight, prevents us from using numerous "topological" tools from \cite{hai-14} (observe for instance the compactness assumption in the general convergence result \cite[Theorem 10.7]{hai-14}). As we mentioned it earlier, this is the reason for our compactly-supported assumption on the vector field $F$ (see Definition \ref{def:vector-field}). In other words, and as we will see it in detail throughout the expansions of Section \ref{sec:mod-equation}, the fact that $F\in \cac^\infty_\compac(\R^2)$ will enable us to bring the study of the noise (and the model) back to some compact domain.

A significant improvement would then consist in getting rid of this compactly-supported assumption by allowing for more general $F$, which supposes to handle the (spatial) asymptotic behaviour of the processes throughout the machinery of regularity structures. Given the high technicality of the theory and its constructions, such an asymptotic follow-up happens to be a very hard task, although first results in this direction have recently been obtained by Hairer and Labbé \cite{hairer-labbe} for the parabolic Anderson model. The latter study, together with other works in progress by the same authors, give us hope that a similar strategy could be implemented for our non-linear fractional heat equation.  

\smallskip

\noindent
\textbf{(4)} Through the above formulation of the problem, we have therefore chosen to provide an interpretation of Equation (\ref{intro:eq-frac}) as the limit of ordinary differential calculus, in the spirit of rough-path-type results (see e.g. \cite[Definition 10.34]{friz-victoir-book} or the statements in \cite[Section 1.5]{hai-14}). In fact, as we evoked it earlier, this passing-to-the-limit approach will be seen as a consequence of a (much) more abstract reading of the equation in some "modelled" space. Thus, by anticipating the considerations of Section \ref{sec:mod-equation}, the regularity-structure solution $Y$ in Theorem \ref{main-theo} could equivalently be defined in a more intrinsic (but more abstract) way as the \emph{reconstruction of the unique (local) solution of the equation modelled along an appropriate $(\al,K)$-rough path $\uxi$} (see also Remark \ref{rk:maximal}).

In the case of a white-in-time noise, that is when $H_1 =\frac12$, it turns out that $Y$ also coincides with the classical Itô solution of the equation, as stated in Theorem \ref{main-theo}, point $(ii)$ (note that this property somehow extends the main identification result in \cite{hairer-pardoux} to a colored-in-space noise). Beyond this Itô situation, and based on the considerations of Section \ref{sec:identif}, we suspect that our solution $Y$ could be identified with some mild Skorohod-type solution of the equation, with convolutional integral understood via an appropriate divergence operator of Malliavin calculus (in the spirit of \cite{hu-nualart}). Nevertheless, the comparison procedure might be a complicated task here, simply because "stochastic" approaches to the non-linear model (\ref{intro:eq-frac}) have not been studied much in the literature for $H_1\neq\frac12$. Therefore, we have prefered to defer such a comparative analysis to a future study.

\smallskip

\noindent
\textbf{(5)} A well-identified drawback of this passing-to-the-limit approach is that, just as in \cite{hai-14} (or more generally throughout the rough-path literature), we cannot guarantee that the convergence results in Theorem \ref{main-theo} are independent of our choice of the approximation $\xi^n$ of $\xi=\partial_t\partial_x X$ (with the considerations of Section \ref{sec:mod-equation} in mind, the convergence actually relies on the whole underlying $(\al,K)$-rough path $\uxi$). Another natural approximation procedure, used for instance in \cite{hai-14}, is given by the convolution $\tilde{\xi}^n:=\rho_{2^n} \ast \partial_t\partial_x X$, with $\rho_{2^n}$ a (dyadic) sequence of smooth mollifiers. Our guess is that, using suitable Fourier transforms, such an approximation of the noise would lead to the very same "Skorohod-type" limit, with correction term of the same form and of the same order.

\subsection{Hairer-Besov spaces}\label{subsec:spaces} Let us conclude this first section by introducing the spaces of functions at the core of our study, along the ideas of \cite{hai-14}. From now on, and for the rest of the paper, we fix the parabolic scaling $\scal=(2,1)$ of $\R^2$, and set, for every multi-index $k=(k_1,k_2)\in \N^2$, $|k|_\scal:=2k_1+k_2$, while for $x=(x_1,x_2)\in \R^2$, we will use the scaled norm 
$$\|x\|_\scal:=(|x_1|+|x_2|^2)^{1/2} \ ,$$
with balls $\cb_\scal(x,r)$ centered at $x$ with radius $r$. 

\smallskip

\noindent
For every integer $\ell \geq 0$ and every set $\compac \subset \R^2$, we denote by $\cac^\ell(\compac)$ the space of functions $\vp:\R^2 \to \R$ with compact support in $\compac$, which are $k$-times differentiable for every multiindex $k=(k_1,k_2)$ such that $|k|_\scal \leq \ell$, with continuous bounded derivatives. Set 
$$\|\vp\|_{\cac^\ell(\compac)}:=\sum_{|k|_\scal \leq \ell} \|D^k \vp\|_{L^\infty(\R^2)} \quad \text{and} \quad \cac_c^\ell(\R^2):=\cup_{\compac \ \text{compact}} \cac^\ell(\compac).$$
Then, as usual, we define $\cd_\ell'(\R^2)$ as the set of linear forms on $\cac_c^\ell(\R^2)$ whose restriction to every $\cac^\ell(\compac)$ ($\compac$ compact) is continuous.

\smallskip

We first transpose the classical Hölder spaces in this setting:
\begin{definition}
For every $\al\in (0,1)$ and every set $\compac \subset \R^2$, we say that a function $\theta:\compac \to \R$ belongs to $\cac^\al(\compac)=\cac^\al_\scal(\compac)$ if the following quantity is finite:
\begin{equation}
\|\theta\|_{\al;\compac}:=\sup_{x\in \compac} \, |\theta (x)|+\sup_{\substack{x,y\in \compac\\ \|x-y\|_\scal \leq 1}} \frac{|\theta (x)-\theta (y)|}{\|x-y\|_\scal^\al} \ .
\end{equation}
Then we denote by $\cac^\al_c(\R^2)$ the space of functions $\theta: \R^2\to \R$ whose restriction to every compact set $\compac \subset \R^2$ belongs to $\cac^\al(\compac)$.
\end{definition}

Let us now turn to the definition of the Besov-type spaces of distributions involved in the theory of regularity structures. To this end, we need to recall the following notation for the "scaling" operator: for all $\delta >0$, $x,y\in \R^2$ and $\vp:\R^2 \to \R$, denote
\begin{equation}\label{scaling-operator}
(\cs^\delta_{\scal,x} \vp)(y):=\delta^{-3} \vp\big(\delta^{-2}(y_1-x_1),\delta^{-1}(y_2-x_2)\big) \ .
\end{equation}

\begin{definition}\label{def:besov-space}
For every $\al <0$, every $r_0\geq 1$ and every set $\compac\subset \R^2$, we say that a distribution $\xi\in \cd'(\R^2)$ belongs to $\cac^{\al,r_0}(\compac)$ if it belongs to $\cd'_{\ell}(\R^2)$ with $\ell=- \left\lfloor\al\right\rfloor$ and if the quantity
$$\|\xi\|_{\al,r_0;\compac}:=\sup_{x\in \compac} \sup_{\vp\in \cac^{\ell}(\cb_\scal(0,r_0))} \sup_{\delta\in (0,1]} \frac{|\langle \xi,\cs^\delta_{\scal,x} \vp \rangle |}{\delta^\al \|\vp\|_{\cac^\ell}}$$
is finite. In the sequel, we will essentially deal with $ \cac^{\al,1}(\compac)$, that we denote by $\cac^\al(\compac)$, and we write $\|\xi\|_{\al;\compac}$ for $\|\xi\|_{\al,1;\compac}$. Also, we denote by $\cac^\al_c(\R^2)$ the set of distributions $\xi\in \cd'(\R^2)$ such that $\xi\in \cac^\al(\compac)$ for every compact set $\compac$.
\end{definition}

The radius parameter $r_0$ actually has a limited impact in this definition, due to the following elementary property that we label for further use:

\begin{lemma}\label{lem:topo}
Fix $\al <0$. For every $r_0\geq 1$, there exists a constant $C_{r_0}$ such that, for any distribution $\xi\in \cd'_{\ell}(\R^2)$ (with $\ell=- \left\lfloor\al\right\rfloor$) and any set $\compac\subset \R^2$,
$$\|\xi\|_{\al;\compac} \leq \|\xi\|_{\al,r_0;\compac}\leq C_{r_0}\cdot \|\xi\|_{\al;\compac+\cb_\scal(0,r_0)} \ .$$
\end{lemma}

\begin{proof}
Consider a finite cover $(\cb_\scal(x_i,1))_{i\in I}$ (with $x_i \in \cb_\scal(0,r_0)$) of $\cb_\scal(0,r_0)$, together with a smooth partition of unity $(\rho_i)_{i\in I}$ associated with it (meaning that $\text{supp} \, \rho_i \subset \cb_\scal(x_i,1)$ and $\sum_i \rho_i=1$ on $\cb_\scal(0,r_0)$). For any $\vp \in \cac^\ell(\cb_\scal(0,r_0))$, set $\vp_i=\rho_i\cdot \vp$ and $\bar{\vp}_i(y)=\vp_i(y+x_i)$. Then $\bar{\vp}_i \in \cac^\ell(\cb_\scal(0,1))$ and
$$\langle \xi,\cs^\delta_{\scal,x}\vp\rangle =\sum_{i\in I} \langle \xi,\cs^\delta_{\scal,x}\vp_i\rangle=\sum_{i\in I} \langle \xi,\cs^\delta_{\scal,x+(\delta^2x_i^1,\delta x_i^2)}\bar{\vp}_i\rangle \ .$$
The result immediately follows.
\end{proof}

The subsequent analysis will also appeal to some "lifted" version of these spaces, with "fibers" locally given by $\cac^\al_c(\R^2)$:

\begin{definition}\label{def:besov-space-2}
For every $\al <0$ and every set $\compac\subset \R^2$, we say that a map $\zeta:\R^2 \to  \cd'(\R^2)$ belongs to $\pmb{\cac}^\al(\compac)$ if for every $x\in \R^2$, $\zeta_x$ belongs to $\cd'_{\ell}(\R^2)$ with $\ell=- \left\lfloor\al\right\rfloor$  and if the quantity
$$\|\zeta\|_{\al;\compac}:=\sup_{x\in \compac} \sup_{\vp\in \cac^{\ell}(\cb_\scal(0,1))} \sup_{\delta\in (0,1]} \frac{|\langle \zeta_x,\cs^\delta_{\scal,x} \vp \rangle |}{\delta^\al \|\vp\|_{\cac^\ell}}$$
is finite. We denote by $\pmb{\cac}^\al_c(\R^2)$ the set of maps $\zeta:\R^2 \to  \cd'(\R^2)$ such that $\zeta\in \pmb{\cac}^\al(\compac)$ for every compact set $\compac$.
\end{definition}

\begin{remark}
One can obviously identify $\cac^\al_c(\R^2)$ with the subspace of constant map in $\pmb{\cac}^\al_c(\R^2)$.
\end{remark}

\section{Modelled equation}\label{sec:mod-equation}

In this section, we propose to review in detail the successive steps of the "lift" procedure at the core of the regularity structures machinery, under the assumptions of Theorem \ref{main-theo}, point $(ii)$, that is when $2\geq 2 H_1+H_2 >\frac53$. Thus, the subsequent reasoning is essentially a particular case of the general procedure exhibited in \cite{hai-14}. To be more specific, the assumptions that prevail in our study are not exactly the same as in \cite{hai-14}, since we have chosen to work on the whole space $\R$ (rather than a bounded domain) and rely on a localization-in-space argument based on the vector field $F$.  We will also slightly revisit the localization-in-time argument by using smooth cut-off functions. Therefore, it is our duty to check that these technical changes do not affect the method and the results of \cite{hai-14}.


\smallskip

Beyond the consideration of these minor differences, we have seen the problem as an opportunity to provide a detailled application (with explicit parameters and expansions) of the regularity structures procedure, whose general formulation may look very dense and abstract to a non-initiated reader. In some way, the forthcoming process can actually be considered as a PDE counterpart of the usual rough paths "2-step" situation, that is when the Hölder coefficient of the driving process in (\ref{rough-system}) belongs to $(\frac13,\frac12]$ and one thus needs to involve a Lévy area in the procedure (see e.g. \cite{gubi}). In this regard, observe that if we "smoothen" the noise in space by letting $H_2$ tend to $1$, the condition for the time-regularity parameter $H_1$ turns to $H_1\in (\frac13,\frac12]$. With these preliminary comments in mind, we can now turn to the introduction of our framework.

\smallskip

A first observation regarding the lift procedure is that it only depends on the global pathwise smoothness of the noise $\xi=\partial_t\partial_x X$, measured in terms of some Besov-Hairer space $\cac^\al_c(\R^2)$. Therefore, we need to slightly anticipate the next section, and especially Corollary \ref{cor:converg}: in the situation we are interested in, we know that (almost surely) $\xi\in \cac_c^\al(\R^2)$ for any $\al\in (-\frac43,-3+2H_1+H_2)$. Accordingly, from now on and for the rest of the section, we fix 
\begin{equation}
\boxed{\al\in (-\frac43,-1)}\ .
\end{equation}

\subsection{Setting}\label{subsec:prelim}
Let us first introduce the space (or \emph{regularity structure}) in which the equation will be transposed for better reading. Namely, we define successively
$$A=\{ \al,2\al+2,\al+1,0,\al+2,1 \} \ ,$$
$$\struc_\al=\text{span}\{\Xi\} \quad , \quad \struc_{2\al+2}=\text{span}\{\Xi \ci(\Xi)\} \ ,$$
$$\struc_{\al+1}=\text{span}\{\Xi X_2\} \quad , \quad \struc_{\al+2}=\text{span}\{\ci(\Xi)\} \quad , \quad \struc_1=\text{span}\{X_2\} \ ,$$
and set $\struc=\oplus_{\la\in A} \struc_\la$. For the time being, all these symbols $\Xi$, $\Xi\ci(\Xi)$,... must be considered as abstract words that we turn into independent basis vectors of the space $\struc$. We also endow $\struc$ with a natural commutative product $\star$ by setting
$$\Xi\star \ci(\Xi)=\Xi \ci(\Xi) \quad , \quad \Xi \star X_2=\Xi X_2 \quad , \quad \1 \star \tau=\tau$$
for all basis vector $\tau$, and $\tau \star \tau'=0$ in every other situation. For any $\bu\in \struc$ and $\be\in A$, we will denote by $\struc_\be(\bu)$ the projection of $\bu$ onto $\struc_\be$.

\smallskip

Then, in order to introduce the set of elements (the \emph{model}) that will make a link between $\struc$ and more common spaces of distributions, a few technical preliminaries are in order. To start with, note that the machinery is to behave very differently for the heat kernel around the singularity $0$ and away from it. Along this idea, the strategy will heavily rely on the following splitting.

\begin{lemma}{(\cite[Lemma 5.5]{hai-14})}\label{lem:decompo-noyau}
There exists a smooth function $\rho_0:\R^2 \to \R$ with compact support such that we can decompose the heat kernel $G$ as a sum
\begin{equation}\label{decompo-g}
G=\rho_0\cdot  G+(1-\rho_0)\cdot G=:K+G^\sharp 
\end{equation}
with $K,G^\sharp$ satisfying the following properties:
\smallskip

\noindent
$\bullet$ $G^\sharp$ is smooth in $\R^2$ and supported in $\R^2 \backslash \cb_\scal(0,\frac12)$\ ,

\smallskip

\noindent
$\bullet$ $K$ is supported in $\cb_\scal(0,1)$ and it can be decomposed as a sum $K=\sum_{n\geq 0} K_n$ with $K_n(x)= 2^{-2n} (\cs^{2^{-n}}_{\scal,0} K_0)(x)$ 
for some smooth function $K_0$ with support in $\cb_\scal(0,1)$. In particular, for every distribution $\eta$ and every multiindex $\ell=(\ell_1,\ell_2)\in \N^2$, one has the identity
\begin{equation}\label{decompo-k}
[ D^\ell K_n\ast \eta](x)=2^{(|\ell|_\scal-2)n} \langle \eta,\cs_{\scal,x}^{2^{-n}}(D^\ell K_0)\rangle \ ,
\end{equation}
where we recall that $|\ell|_\scal=2\ell_1+\ell_2$.
\end{lemma}

Let us see to what extent convolution with the singular part $K$ allows to regularize our starting noise $\xi$.

\begin{lemma}\label{lem:convo-c-al}
For every $\xi\in \cac_c^\al(\R^2)$, it holds that $K\ast \xi \in \cac_c^{\al+2}(\R^2)$ and for every compact set $\compac \subset \R^2$,
$$\|K\ast \xi\|_{\al+2;\compac} \lesssim \|\xi\|_{\al;\text{rect}(\compac)} \ ,$$
where $\text{rect}(\compac)$ stands for the smallest rectangle $[x_1,x_2]\times [y_1,y_2]$ that contains $\compac$.
\end{lemma}

\begin{proof}
First, for $x\in \compac$, one has by (\ref{decompo-k}) 
$$|(K\ast \xi)(x)| \leq \sum_{n\geq 0} 2^{-2n} \big| \langle \xi,\cs_{\scal,x}^{2^{-n}}K_0 \rangle \big| \lesssim \|\xi\|_{\al;\compac}\cdot \sum_{n\geq 0} 2^{-n(\al+2)} \ ,$$
and hence $\sup_{x\in \compac} |(K\ast \xi)(x)|  \lesssim \|\xi\|_{\al;\compac}$. Then fix $x\neq y \in \compac$ such that $\|x-y\|_\scal \leq 1$ and let $i\geq 0$ be such that $2^{-(i+1)} \leq \|x-y\|_\scal \leq 2^{-i}$. Decompose the difference $(K\ast \xi)(x)-(K\ast \xi)(y)$ as
$$\sum_{0\leq n <i} \big[(K_n\ast \xi)(x)-(K_n\ast \xi)(y) \big]+\sum_{n\geq i} \big[(K_n\ast \xi)(x)-(K_n\ast \xi)(y) \big]=: I_{xy}+II_{xy} \ .$$
As above, we can rely on the bound $|(K_n \ast \xi)(x)| \lesssim \|\xi\|_{\al;\compac}\cdot 2^{-n(\al+2)}$ to derive that
$$|II_{xy}| \lesssim \|\xi\|_{\al;\compac}\cdot 2^{-i(\al+2)} \lesssim \|\xi\|_{\al;\compac}\cdot \|x-y\|_\scal^{\al+2} \ .$$
In order to deal with $I_{xy}$, we decompose each summand as
\begin{eqnarray*}
\lefteqn{(K_n \ast \xi)(x)-(K_n \ast \xi)(y)}\\
&=&\big[ (K_n \ast \xi)(x_1,x_2)- (K_n \ast \xi)(y_1,x_2)\big]+\big[ (K_n \ast \xi)(y_1,x_2)- (K_n \ast \xi)(y_1,y_2)\big]\\
&=& \int_0^1 dr \, \big\{(x_1-y_1) \cdot (D^{(1,0)}K_n \ast \xi)(y_1+r(x_1-y_1),x_2)\\
& &\hspace{3cm}+(x_2-y_2)\cdot (D^{(0,1)}K_n \ast \xi)(y_1,y_2+r(x_2-y_2)) \big\} \ .
\end{eqnarray*}
Therefore, by (\ref{decompo-k}), 
$$|(K_n \ast \xi)(x)-(K_n \ast \xi)(y) | \lesssim \|\xi\|_{\al;\text{rect}(\compac)}\cdot \big\{ |x_1-y_1| \cdot 2^{-n\al}+|x_2-y_2| \cdot 2^{-n(\al+1)} \big\} \ , $$
and since both $\al$ and $\al+1$ are negative, we can conclude that
$$|I_{xy}| \lesssim \|\xi\|_{\al;\text{rect}(\compac)}\cdot \big\{ |x_1-y_1| \cdot 2^{-i\al}+|x_2-y_2| \cdot 2^{-i(\al+1)} \big\} \lesssim \|\xi\|_{\al;\text{rect}(\compac)}  \cdot \|x-y\|_\scal^{\al+2} \ .$$
\end{proof}

With this property in hand, we can define the central element (on top of $\xi$ itself) at the core of the forthcoming model.

\begin{definition}
Let $K$ be defined as in Lemma \ref{lem:decompo-noyau} and fix $\xi\in \cac^{\al}_c(\R^2)$. We call a \emph{$K$-L{\'e}vy area above $\xi$} any map $\mathcal{A} : \R^2 \to \cd'(\R^2)$ satisfying the two following conditions.

\smallskip

\noindent
(i) \emph{$K$-Chen relation:} For all $x,y\in \R^2$, $\mathcal{A}_x-\mathcal{A}_y=[(K\ast \xi)(y)-(K\ast \xi)(x)] \cdot \xi$ \ ;

\smallskip

\noindent
(ii) \emph{Besov regularity:} $\mathcal{A}$ belongs to $\pmb{\cac}^{2\al+2}_c(\R^2)$ (see Definition \ref{def:besov-space-2})\ .

\smallskip

\begin{remark}\label{rk:levy-area-smooth}
In the benchmark situation where $\xi$ actually defines a function, there exists a "canonical" $K$-Lévy area above it given by the formula 
\begin{equation}\label{levy-area-smoo}
\mathcal{A}_x(z):=[(K\ast \xi)(z)-(K\ast \xi)(x)]\cdot \xi(z) \ .
\end{equation}
However, just as with the classical Lévy area of rough-path theory, there is of course no systematic way to extend (\ref{levy-area-smoo}) to any distribution $\xi$. Section \ref{sec:construction} will actually be devoted to the construction of such a process above the fractional noise involved in Theorem \ref{main-theo}. Note also that, like its one-parameter counterpart, a $K$-Lévy area is not unique: for instance, any constant $C$ gives rise to another $K$-Lévy area by setting $\hat{\mathcal{A}}_x(z):=\mathcal{A}_x(z)+C$.
\end{remark}

We call an \emph{$(\alpha,K)$-rough path} any pair $\uxi=(\xi,\bxi^{\mathbf{2}})$ where $\xi \in \cac^{\al}_c(\R^2)$ and $\bxi^{\mathbf{2}}$ is a $K$-Lévy area above $\xi$, and we denote, for every compact set $\compac \subset \R^2$,
\begin{equation}
\|\uxi\|_{\al;\compac}=\|\xi\|_{\al;\compac}+\|\xi^{ \mathbf{2}}\|_{2\al+2;\compac} \ .
\end{equation}
Also, if $\uxi=(\xi,\bxi^{\mathbf{2}})$ and $\pmb{\zeta}=(\zeta,\zeta^{\mathbf{2}})$ are two $(\al,K)$-rough paths, we denote, for every compact set $\compac \subset \R^2$,
\begin{equation}\label{norm-rp}
\|\uxi;\pmb{\zeta}\|_{\al;\compac}=\|\xi-\zeta\|_{\al;\compac}+\|\xi^{ \mathbf{2}}-\zeta^{ \mathbf{2}}\|_{2\al+2;\compac} \ .
\end{equation}
\end{definition}

\

Now, given an $(\al,K)$-rough path $\uxi$, define 
$$\gga^{\uxi}: \R^2 \times \R^2 \to \cl(T) \quad , \quad \Pi^{\uxi}: \R^2 \to \cl(T,\cs'(\R^2)) \ ,$$
along the following formulas: for all $x,y\in \R^2$,
\begin{equation}\label{def:gga-1}
\gga^{\uxi}_{xy}(\Xi)=\Xi \quad , \quad \gga^{\uxi}_{xy}(\1)=\1 \quad , \quad  \gga^{\uxi}_{xy}(\ci(\Xi))=\ci(\Xi)+\{(K\ast \xi)(x)-(K\ast \xi)(y)\} \, \1 \ ,
\end{equation}
\begin{equation}\label{def:gga-2}
\gga^{\uxi}_{xy}(X_2)=X_2+(x_2-y_2) \, \1 \ ,
\end{equation}
\begin{equation}\label{def:gga-3}
\gga^{\uxi}_{xy}(\Xi \ci(\Xi))=\gga^{\uxi}_{xy}(\Xi)\star \gga^{\uxi}_{xy}(\ci(\Xi)) \quad , \quad  \gga^{\uxi}_{xy}(\Xi X_2)=\gga^{\uxi}_{xy}(\Xi)\star \gga^{\uxi}_{xy}(X_2) \ , 
\end{equation}
and
\begin{equation}\label{def:pi-1}
\Pi^{\uxi}_x(\Xi)=\xi \ , \ \Pi^{\uxi}_x(\Xi \ci(\Xi))=\xi^{\mathbf{2}}_x \ , \ \Pi^{\uxi}_x(\1)=1 \ , \ \Pi^{\uxi}_x(\ci(\Xi))(y)=[K\ast \xi](y)-[K\ast \xi](x) \ ,
\end{equation}
\begin{equation}\label{def:pi-2}
\Pi^{\uxi}_x(X_2)(y)=y_2-x_2 \ , \ \Pi^{\uxi}_x(\Xi X_2)=\Pi^{\uxi}_x(X_2)\cdot \Pi^{\uxi}_x(\Xi) \ .
\end{equation}

By combining Lemma \ref{lem:convo-c-al} with the definition of an $(\al,K)$-rough path, the following property is readily checked.
\begin{proposition}\label{prop:model}
For every $(\al,K)$-rough path $\uxi$, the pair $Z(\uxi)=(\Pi^{\uxi},\gga^{\uxi})$ defined along Formulas (\ref{def:gga-1})-(\ref{def:pi-2}) is a model for the regularity structure $\struc$, in the sense of \cite[Definition 2.17]{hai-14}. In particular, it satisfies the relation: for all $x,y\in \R^2$
\begin{equation}\label{key-relation}
\Pi^{\uxi}_y=\Pi^{\uxi}_x \circ \gga^{\uxi}_{xy} \ .
\end{equation}
Also, for every compact set $\compac\subset \R^2$, and with the notation of \cite[Section 2.3]{hai-14}, it holds that
$$\|(\Pi^{\uxi},\gga^{\uxi})\|_{\al;\compac}\lesssim \|\uxi\|_{\al;\compac_0} \quad , \quad \|(\Pi^{\uxi},\gga^{\uxi});(\Pi^{\uxi'},\gga^{\uxi'})\|_{\al;\compac}\lesssim \|\uxi;\uxi'\|_{\al;\compac_0} \ ,$$
for some larger compact set $\compac_0$. 

\end{proposition}

\begin{proof}
For instance, due to the $K$-Chen relation,
\begin{eqnarray*}
\Pi^{\uxi}_x\big( \Gamma^{\uxi}_{xy}\big( \Xi \ci(\Xi)\big)\big) &=& \Pi^{\uxi}_x\big( \Xi \ci(\Xi)\big) +\big\{(K\ast \xi)(x)-(K\ast \xi)(y)\big\}\cdot \Pi^{\uxi}_x(\Xi)\\
&=&\xi^{\mathbf{2}}_x+\big\{(K\ast \xi)(x)-(K\ast \xi)(y)\big\}\cdot \xi \ = \ \xi^{\mathbf{2}}_y \ = \ \Pi^{\uxi}_y\big( \Xi \ci(\Xi)\big) \ .
\end{eqnarray*}

\end{proof}


We thus have the following picture: in the situation we are interested in, providing an $(\al,K)$-rough path is enough to construct (and control) a natural model above $\xi$. To this extent, the present introduction of the setting offers some compromise between a rough-path formulation ("process + L{\'e}vy area") and a regularity-structure formulation ("regularity structure + model") of the assumptions required by the procedure.

\smallskip

At this point, and as we mentionned it in the introduction, another important ingredient consists in "lifting" the classical Hölder topology to $\struc$-valued functions. Along this idea, and given an $(\al,K)$-rough path $\uxi$, we define the spaces $\cd^{\ga,\eta}(\uxi)$ ($\ga,\eta \in \R$) of \emph{singular modelled distributions} as follows. Introduce
$$P:=\{(s,x)\in \R^2: \ s=0\} \ ,$$
and then for every set $\compac\subset \R^2$, define (following \cite[Definition 6.2]{hai-14})
\begin{equation}\label{def:norm-ga-eta}
\|\bu\|_{\ga,\eta;\compac}:=\sup_{\beta<\ga}\sup_{x\in \compac\backslash P} \frac{|\struc_\be(\bu(x))|}{\|x\|^{(\eta-\beta)\wedge 0}_P}+ \sup_{\beta<\ga}\sup_{(x,y)\in \compac_P} \frac{\big| \struc_\be\big(\bu(x)-\gga^{\uxi}_{xy}(\bu(y))\big)\big|}{\| x-y\|_\scal^{\ga-\beta}\|x;y\|_P^{\eta-\ga}} \ ,
\end{equation}
where $\|x\|_P:=\inf(1,|x_1|)$, $\|x;y\|_P:=\inf(1,|x_1|,|y_1|)$ and
$$\compac_P:=\{(x,y)\in (\compac\backslash P)^2:  \ \|x-y\|_\scal \leq  \|x;y\|_P \}\ . $$
In the sequel, we will denote by $\cd^{\ga,\eta}(\uxi)$ the space of functions $\bu:\R^2 \to T$ for which the global norm $\|\bu \|_{\ga,\eta;\R^2}$ (written $\|\bu \|_{\ga,\eta}$) is finite. Also, we denote by $\cd^{\ga,\eta}_\beta(\uxi)$ the set of elements in $\cd^{\ga,\eta}(\uxi)$ which take values in $\struc_{\beta +}=\oplus_{\la \geq \beta} \struc_{\la}$.

\smallskip

When comparing two elements $\bu\in \cd^{\ga,\eta}(\uxi)$ and $\bu'\in \cd^{\ga,\eta}(\uxi')$, we use the natural quantity
\begin{multline}\label{def:norm-ga-eta-2}
\|\bu;\bu'\|_{\ga,\eta;\compac}:=\\
\sup_{\beta<\ga}\sup_{x\in \compac\backslash P} \frac{|\struc_\be(\bu(x)-\bu'(x))|}{\|x\|^{(\eta-\beta)\wedge 0}_P}+ \sup_{\beta<\ga}\sup_{(x,y)\in \compac_P} \frac{\big| \struc_\be\big(\bu(x)-\gga^{\uxi}_{xy}(\bu(y))-\bu'(x)+\gga^{\uxi'}_{xy}(\bu'(y))\big)\big|}{\| x-y\|_\scal^{\ga-\beta}\|x;y\|_P^{\eta-\ga}} \ .
\end{multline}

\begin{remark}
As reported in \cite[Remark 5.1]{hai-little}, the essential component in the two quantities (\ref{def:norm-ga-eta}) and (\ref{def:norm-ga-eta-2}) lies in the Hölder smoothness property of the process, that is in the consideration of the increment $\|x-y\|_\scal^{\ga-\be}$. Controlling the sigularity at time $0$ (through the term $\|x\|^{(\eta-\beta)\wedge 0}_P$ or $\|x;y\|_P^{\eta-\ga}$) is of minor importance as far as the global dynamics of the structures is concerned, although this lever will prove to be necessary in order to settle a fixed-point argument or to lift the initial condition.
\end{remark}

It turns out that under suitable regularity assumptions, we can \emph{reconstruct}, from a given modelled distribution, a real distribution (understood in the classical sense) along a procedure which continuously extends the "smooth" case, that is the situation where $\xi$ is a differentiable noise. This result, which defines the so-called \emph{reconstruction operator}, is one of the cornerstone of the regularity structures theory. The thorough statement that we provide here is obtained by combining Theorem 3.10, Remark 3.15, Lemma 6.7 and Proposition 6.9 of \cite{hai-14}.
\begin{theorem}{(Reconstruction operator)}\label{theo:reconstruction}
Let $\uxi$ be an $(\al,K)$-rough path and $\bu\in \cd^{\ga,\eta}_\beta(\uxi)$, for some parameters $\ga >0$, $\eta \leq \ga$ and $\beta \in [\al,0]$. Then there exists a unique element $\crr_{\uxi}\bu\in \cac_c^{\beta \wedge \eta}(\R^2)$ which satisfies the two following properties:

\smallskip

\noindent
(i) ("Globally") For every compact set $\compac \subset \R^2$,
\begin{equation}\label{recon-global}
\|\crr_{\uxi}(\bu)\|_{\beta \wedge \eta;\compac} \lesssim \|\uxi\|_{\al;\bar{\compac}} \|\bu\|_{\ga,\eta} \ ,
\end{equation}
where $\bar{\compac}$ stands for the 1-fattening of $\compac$.

\smallskip

\noindent
(ii) ("Locally") For every compact set $\compac \subset \R^2$, every $x\in \compac \backslash P$, every $\delta \in (0,\frac14 \|x\|_P^{1/2})$ and every $\vp\in \cac^2(\cb_\scal(0,1))$,
\begin{equation}\label{recon-local}
\big| \big( \crr_{\uxi}\bu-\Pi^{\uxi}_x(\bu(x))\big)\big( \cs^\delta_{\scal,x}\vp\big)\big| \lesssim  \delta^\ga \|\vp\|_{\cac^2}  \|x\|_P^{\eta-\ga}\cdot \|\uxi\|_{\al;\bar{\compac}} \|\bu\|_{\ga,\eta} \ .
\end{equation}

\smallskip

\noindent
Besides, if $\uxi'$ is another $(\al,K)$-rough path and $\bu'\in \cd^{\ga,\eta}_\beta(\uxi')$, one has, with similar notations,
\begin{equation}\label{conti-reco}
\| \crr_{\uxi}\bu-\crr_{\uxi'}\bu'\|_{\beta \wedge \eta;\compac}\leq Q_{\uxi,\uxi',\bu,\bu'} \cdot \big\{ \|\uxi;\uxi'\|_{\al,\bar{\compac}}+\|\bu;\bu'\|_{\ga,\eta;\bar{\compac}}\big\}\ ,
\end{equation}
as well as
\begin{multline}
\big| \big( \crr_{\uxi}\bu-\Pi^{\uxi}_x(\bu(x))-\crr_{\uxi'}\bu'+\Pi'_x(\bu'(x))\big)\big( \cs^\delta_{\scal,x}\vp\big)\big|\\
 \leq Q_{\uxi,\uxi',\bu,\bu'}\cdot \delta^\ga \|\vp\|_{\cac^2}  \|x\|_P^{\eta-\ga} \big\{ \|\uxi;\uxi'\|_{\al,\bar{\compac}}+\|\bu;\bu'\|_{\ga,\eta;\bar{\compac}}\big\}\ ,
\end{multline}
where in both inequalities, $Q_{\uxi,\uxi',\bu,\bu'}$ stands for a polynomial expression in $\|\uxi\|_{\al;\bar{\compac}}$, $\|\uxi'\|_{\al;\bar{\compac}}$, $\|\bu\|_{\ga,\eta;\bar{\compac}}$ and $\|\bu'\|_{\ga,\eta;\bar{\compac}}$. 
\end{theorem}

\smallskip

\begin{corollary}
If $\uxi$ is smooth, in the sense that $\xi$ (resp. $\xi^{\2}$) defines a smooth function $\xi:\R^2 \to \R$ (resp. $\xi^{\2}:\R^2 \times \R^2 \to  \R$), then, except on $P$, $\crr_{\uxi}\bu$ is a continuous function given by the formula: for all $x\in \R^2\backslash P$,
\begin{equation}\label{reconstruction-smooth}
(\crr_{\uxi}\bu)(x)=\Pi^{\uxi}_x(\bu(x))(x) \ .
\end{equation}
\end{corollary}

\smallskip

\begin{remark}
Here and in the sequel, the "smoothness" terminology is more of an additional reference to the vocabulary commonly used in rough-path theory, and in this previous statement, such a regularity assumption can of course be alleviated in a drastic way (see \cite[Remark 3.15]{hai-14}). This being said, for fixed $n$, the approximation $\xi^n$ which we shall then apply this particular result to does define a smooth (i.e., infinitely differentiable) function.
\end{remark}

\

We now have all the tools in hand to be more specific about the objective of the next subsections. Namely, we intend to show how the equation can be naturally transposed and solved in the space $\cd_0^{\ga,0}(\uxi)$, where the parameter $\ga$ is henceforth fixed as follows:
\begin{equation}\label{ga-eta}
\boxed{\ga=2\al+4 \ \in (1,2)  } \ .
\end{equation}
To get a clear insight into the topology induced by such parameters, observe that, by definition, the space $\cd_0^{\ga,0}(\uxi)$ corresponds to the set of functions 
\begin{equation}\label{expan-bu}
\bu=\bu^0\, \1+\bu^1 \, \ci(\Xi)+\bu^2\,  X_2
\end{equation}
such that
\begin{equation}\label{condition-sup}
\sup_{x\in \R^2}\big( |\bu^0(x)|,|\bu^1(x)| \cdot \|x\|_P^{\al+2}, |\bu^2(x)| \cdot \|x\|_P \big) \ < \infty \ ,
\end{equation}
and for all $(x,y)\in \R^2$ satisfying $x_1\neq 0$, $y_1\neq 0$ and $\|x-y\|_\scal \leq \|x;y\|_P$,
\begin{equation}\label{condition-holder-1}
|\bu^0(x)-\bu^0(y)-\bu^1(y)\cdot [(K\ast \xi)(x)-(K\ast \xi)(y)]-\bu^2(y)\cdot (x_2-y_2)| \leq C \cdot \|x-y\|_\scal^\ga\cdot \|x;y\|_P^{-\ga} \ ,
\end{equation}
\begin{equation}\label{condition-holder-2}
|\bu^1(x)-\bu^1(y)| \leq C \cdot \|x-y\|_\scal^{\al+2}\cdot \|x;y\|_P^{-\ga} \ , 
\end{equation}
\begin{equation}\label{condition-holder-3}
|\bu^2(x)-\bu^2(y)| \leq C \cdot \|x-y\|_\scal^{\ga-1}\cdot \|x;y\|_P^{-\ga} \ ,
\end{equation}
for some finite constant $C$.

\smallskip

An important remark here is that, due to the uniqueness property contained in Theorem \ref{theo:reconstruction} and given the above regularity conditions, the reconstruction $\mathcal{R}_{\uxi}(\bu)$ of such an element $\bu \in \cd^{\ga,0}_0(\uxi)$ is actually very easy to identify (see \cite[Proposition 3.28]{hai-14} for further details):

\begin{proposition}\label{prop:reconstruction-fonction}
For every $(\al,K)$-rough path $\uxi$ and every $\bu\in \cd^{\ga,0}_0(\uxi)$ with decomposition (\ref{expan-bu}), it holds that $\mathcal{R}_\xi(\bu)(x)=\bu^0(x)$ for every $x\in \R^2$.
\end{proposition}

\begin{remark}\label{rk:lift}
Let us try to give a better idea about what we mean by "lifting" the equation in $\struc$. In fact, with Proposition \ref{prop:reconstruction-fonction} in mind, it is natural to consider $\bu$ as a "lift" of $\bu^0$ in $\struc$, with $\bu^1$ and $\bu^2$ playing the role of artificial "derivatives" components. The objective can now be stated as follows: we wish to turn the ("ill-posed") equation (\ref{intro:eq-frac}) into a ("well-posed") equation in $\cd^{\ga,0}_0(\uxi)$ (with internal operations in $\struc$) and therein exhibit a solution $\bu$. Also, the procedure must be performed in such a way that, if $\xi$ happens to be smooth and $\uxi$ is the canonical $(\al,K)$-rough path defined by (\ref{levy-area-smoo}), then the reconstructed process $u:=\crr_\xi\bu$ is the solution of the original equation driven by $\xi$ (and understood in the classical sense). Such a consistency will therefore offer a strong evidence in favor of the viability of the modelling, which will be confirmed a posteriori thanks to the continuity properties of the procedure (see Section \ref{subsec:proof-point-ii}). 
\end{remark}

\

To conclude with these preliminaries, note that, since we are only interested in solutions on a small interval $[0,T]$ (with say $0 <T\leq 1$), we will rely on a localization-in-time of the equation based on cut-off functions. To be more specific, we recast the target equation as follows:
\begin{equation}\label{eq-loc}
u(x)=(G_{x_1} \Psi)(x_2)+\rho_T(x_1) \cdot (G\ast [\rho_+ \cdot F(u) \cdot \xi])(x) \ ,
\end{equation}
where:

\smallskip

\noindent
$\bullet$ $\rho_+(x):=\rho(x_1) \cdot \1_{\R_+}(x_1)$ for some smooth function $\rho:\R \to [0,1]$ with support in $[-2;2]$ and such that $\rho \equiv 1$ on $[-1,1]$ ;

\smallskip

\noindent
$\bullet$ $\rho_T:\R \to [0,1]$ is a smooth function with support in $[-3T;3T]$ such that $\rho_T \equiv 1$ on $[-T,T]$  and $\|\rho'_T\|_{L^\infty(\R)} \lesssim T^{-1}$ (see Lemma \ref{lem:rho-t});

\smallskip

\noindent
$\bullet$ $F(u)(x):=F(x_2,u(x))$ for some function $F\in \cac^\infty_{\compac_F}(\R^2)$ (see Definition \ref{def:vector-field})  ;

\smallskip

\noindent 
$\bullet$ we have set $(G_{x_1} \Psi)(x_2)=\int_{\R} dy\,  G(x_1,x_2-y) \Psi(y) $\, .

\smallskip

\begin{lemma}\label{lem:rho-t}
For any $T>0$, there exists a smooth function $\rho_T:\R\to [0,1]$ with support in $[-3T;3T]$ such that $\rho_T \equiv 1$ on $[-T,T]$  and $\|\rho'_T\|_{L^\infty(\R)} \lesssim T^{-1}$.
\end{lemma}

\begin{proof}
Consider a mollifier $\vp$ on $[-1,1]$, that is a smooth function $\vp:\R\to \R$ with support in $[-1,1]$ and such that $\int_{\R} \vp(u) \, du=1$. Set $\vp_T(u)=\frac{1}{T} \vp\big( \frac{u}{T}\big)$. Then it is easy to check that the function $\rho_T=\1_{[-2T,2T]} \ast \vp_T$ meets the required conditions.
\end{proof}

Given a smooth $\rho:\R \to \R$ and a function $\bu:\R^2 \to \struc$, we will denote by $\rho \cdot \bu:\R^2 \to \struc$ the function whose coordinates in $\struc$ are simply given by $\struc_\be(\rho \cdot \bu)(x)=\rho(x_1) \cdot \struc_\be(\bu)(x)$, $\be \in A$.

\smallskip

For compactness reasons that will prove to be fundamental in Section \ref{sec:construction}, we will also need to control the support of the process at each step of the procedure, both in time and in space.

\begin{definition}
We call the \emph{support} of a modelled distribution $\bg$, and we denote by $\text{supp} \, \bg$, the union of the supports of its components $\struc_\beta( \bg)$, $\beta \in A$.
\end{definition}

\begin{lemma}\label{lem:supp-recon}
For every $\bg\in \cd^{\ga,\eta}_\beta(\uxi)$ with $\ga >0$, $\eta \leq \ga$ and $\beta \in [\al,0]$, it holds that $\text{supp} \, \mathcal{R}_{\uxi}(\mathbf{g}) \subset\,  \text{supp} \, \mathbf{g}$.
\end{lemma}

\begin{proof}
Recall that by its very construction (see the beginning of the proof of \cite[Theorem 3.10]{hai-14}), $\mathcal{R}_{\uxi} \bg$ is the limit of a sequence of continuous functions $\mathcal{R}_{\uxi}^n \bg$ of the form:
$$(\crr_{\uxi}^n \bg)(y)=\sum_{x\in \Lambda_\scal^n} \langle \Pi_x(\bg(x)),\vp^n_x\rangle \cdot \vp^n_x(y)=\sum_{x\in \Lambda_\scal^n \cap \text{supp} \, \bg} \langle \Pi_x(\bg(x)),\vp^n_x\rangle \cdot \vp^n_x(y) \ ,$$
where $\vp^n_x(y)=2^{\frac{3n}{2}}\cdot \vp(2^{2n}(y_1-x_1))\cdot \vp(2^n(y_2-x_2))$, for some compactly supported function $\vp$. If for instance $\text{supp} \, \vp \subset [-C,C]$, then it is readily checked that
$$\text{supp} \, \crr_{\uxi}^n( \bg) \subset \, \text{supp} \, \bg +\cb_{\max}(0,C\cdot 2^{-n}) \ ,$$
where $\cb_{\max}$ refers to the ball with respect to the supremum norm in $\R^2$. The result is now immediate as we let $n$ tend to infinity. 
\end{proof}

\subsection{Composition and multiplication with the noise} 

The first operation involved in (\ref{eq-loc}) consists in composing $u$ with the vector field $F$. Let us see how the procedure can be lifted in $\cd^{\ga,0}_0(\uxi)$, by following the ideas of \cite[Theorem 4.16]{hai-14}.

\begin{proposition}{(Composition)}\label{prop:compo}
Consider an $(\al,K)$-rough path $\uxi$ (resp. $\uxi'$), and for every compact set $\compac \subset \R$, every $F\in \cac^\infty_{\compac}(\R^2)$ and every $\bu=\bu^0 \1+\bu^1 \ci(\Xi)+\bu^2 X_2  \in \cd_0^{\ga,0}(\uxi)$, define 
$$\bff(\bu)=\mathbf{v}^0\, \1+\mathbf{v}^1 \, \ci(\Xi)+\mathbf{v}^2\, X_2$$
with $\mathbf{v}^0(x)=F(x_2,\bu^0(x))$, $\mathbf{v}^1(x)=(\partial_2F)(x_2,\bu^0(x)) \cdot \bu^1(x)$ and
$$\mathbf{v}^2(x)=(\partial_2F)(x_2,\bu^0(x)) \cdot \bu^2(x)+(\partial_1 F)(x_2,\bu^0(x)) \ .$$
Then $\rho_+ \cdot \bff(\bu) \in \cd_0^{\ga,0}(\uxi)$ and  
\begin{equation}\label{bound-compo-1}
\|\rho_+ \cdot \bff(\bu)\|_{\ga,0}\leq Q_{\uxi}\cdot \{1+\|\bu\|_{\ga,0}^2\} \ .
\end{equation}
Besides, for every $\bu'\in \cd_0^{\ga,0}(\uxi')$, it holds that
\begin{equation}\label{bound-compo-2}
\|\rho_+ \cdot \bff(\bu);\rho_+ \cdot \bff(\bu')\|_{\ga,0} \leq Q_{\uxi,\uxi',\bu,\bu'}\cdot \{ \|\bu;\bu'\|_{\ga,0}+\|\uxi ;\uxi'\|_{\al;\compac_1}\} \ .
\end{equation}
In (\ref{bound-compo-1}), $Q_{\uxi}$ is a polynomial expression in $\|\uxi\|_{\al;\compac_1}$, for an appropriate compact set $\compac_1\subset \R^2$ depending only on $(F,\rho)$. In (\ref{bound-compo-2}), $Q_{\uxi,\uxi',\bu,\bu'}$ is a polynomial expression in $\|\bu\|_{\ga,0}$, $\|\bu'\|_{\ga,0}$, $\|\uxi\|_{\al;\compac_1}$ and $\|\uxi'\|_{\al;\compac_1}$.
\end{proposition}




\begin{proof}
First, let us introduce a compact $\compac_0\subset \R^2$ such that $[-2,2]\times \compac_{\bff} \subset \compac_0$, which implies in particular that $\text{supp}\big( \rho_+\cdot \bff(\bu)\big)\subset \compac_0$. By the very definition of $\|\cdot \|_{\ga,0}$, one has
$$\|\rho_+\cdot \bff(\bu)\|_{\ga,0}=\|\rho_+\cdot \bff(\bu)\|_{\ga,0;\ov{\compac}_0}=\|\rho \cdot \bff(\bu)\|_{\ga,0;\ov{\compac}_0\cap (\R_+\times \R)} \ ,$$
where we have denoted by $\ov{\compac}_0$ the $1$-fattening of $\compac_0$. Set $\compac_1:=\ov{\compac}_0\cap (\R_+\times \R)$. Then it is easy to check that due to the smoothness of $\rho$, 
$$\|\rho \cdot \bff(\bu)\|_{\ga,0;\compac_1}\lesssim \|\bff(\bu)\|_{\ga,0;\compac_1} \ ,$$
and so $\|\rho_+\cdot \bff(\bu)\|_{\ga,0} \lesssim \|\bff(\bu)\|_{\ga,0;\compac_1}$. The rest of the proof of (\ref{bound-compo-1}) now consists in a natural Taylor-expansion procedure. We only focus on the increment term in $\struc_0$, that is the one corresponding to (\ref{condition-holder-1}). Pick $x,y\in (\compac_1)_P$ and set $\mathbf{v}:=\bff(\bu)$, $\theta:=K\ast \xi$. Then decompose the increment
\begin{equation}\label{increment-v}
\mathbf{v}^0(x)-\mathbf{v}^0(y)-\mathbf{v}^1(y) \cdot [\theta(x)-\theta(y)] -\mathbf{v}^2(y) \cdot (x_2-y_2)
\end{equation}
as a sum of three terms $I,II,III$, with
$$I=F(x_2,\bu^0(x))-F(y_2,\bu^0(x))-(x_2-y_2)\cdot (\partial_1 F)(y_2,\bu^0(x)) \ ,$$
$$II=(x_2-y_2) \cdot [(\partial_1 F)(y_2,\bu^0(x))-(\partial_1 F)(y_2,\bu^0(y))] \ ,$$
\begin{multline*}
III=
F(y_2,\bu^0(x))-F(y_2,\bu^0(y))\\-(\partial_2 F)(y_2,\bu^0(y)) \cdot \bu^1(y) \cdot [\theta(x)-\theta(y)]-(\partial_2 F)(y_2,\bu^0(y)) \cdot \bu^2(y) \cdot (x_2-y_2)\ .
\end{multline*}
The estimation of $I$ is immediate:
$$|I|\lesssim \lln x_2-y_2\rrn^2 \lesssim \|x-y\|_\scal^2 \lesssim \|x-y\|_\scal^\ga \cdot \|x;y\|_P^{-\ga} \ .$$
For $II$, one has trivially
\begin{equation}\label{ii}
|II| \lesssim \lln x_2-y_2\rrn \cdot |\bu^0(x)-\bu^0(y)| \lesssim \|x-y\|_\scal \cdot |\bu^0(x)-\bu^0(y)| \ .
\end{equation}
The key observation at this point is that we can combine (\ref{condition-sup})-(\ref{condition-holder-1}) with the result of Lemma \ref{lem:convo-c-al} to derive the bound:
\begin{equation}\label{bound-hold-u-0}
|\bu^0(x)-\bu^0(y)| \lesssim \|\bu\|_{\ga,0}\cdot \|x-y\|_\scal^{\al+2}\cdot  \|x;y\|_P^{-(\al+2)}\cdot \{1+\|\xi\|_{\al;\text{rect}(\compac_1)}\} \ ,
\end{equation}
and hence, going back to (\ref{ii}),
$$|II| \lesssim  Q_\xi \cdot \|\bu\|_{\ga,0}\cdot \|x-y\|_\scal^{\ga}\cdot  \|x;y\|_P^{-\ga} \ ,$$
where we have used the fact that $2\al+4 < \al+3$. Finally, as far as $III$ is concerned, we decompose it in a natural way as a sum of two terms $III^{(1)}$ and $III^{(2)}$, with
$$III^{(1)}=\int_0^1 dr\, \big\{ (\partial_2 F)(y_2,\bu^0(y)+r(\bu^0(x)-\bu^0(y)))-(\partial_2 F)(y_2,\bu^0(y)) \big\}\cdot  \big\{\bu^0(x)-\bu^0(y)\big\} \, $$
$$III^{(2)}=(\partial_2 F)(y_2,\bu^0(y))\cdot \big\{ \mathbf{u}^0(x)-\mathbf{u}^0(y)-\mathbf{u}^1(y) \cdot [\theta(x)-\theta(y)] -\mathbf{u}^2(y) \cdot (x_2-y_2)\big\} \ .$$
The bound for $III^{(2)}$ is immediate. For $III^{(1)}$, we can use (\ref{bound-hold-u-0}) again to assert that
$$|III^{(1)}| \lesssim |\bu^0(x)-\bu^0(y)|^2 \lesssim Q_\xi \cdot \|\bu\|_{\ga,0}^2 \cdot \|x-y\|_\scal^{\ga} \cdot \|x;y\|_P^{-\ga} \ ,$$
which completes the estimation of (\ref{increment-v}).

\smallskip

The argument leading to (\ref{bound-compo-2}) follows the same general scheme (localization plus Taylor expansion), and we therefore omit it for the sake of conciseness.

\end{proof}

We can now turn to the second operation in (\ref{eq-loc}), namely the pointwise multiplication with the noise $\xi$. Observe that the well-posedness of such a product is not clear at all in the real world. In the "modelled" space, the operation becomes an elementary multiplication with the "modelled" noise $\Xi$.

\begin{proposition}{(Multiplication with the noise)}\label{prop:multi}
Consider an $(\al,K)$-rough path $\uxi$ (resp. $\uxi'$). Then, for every $\bu  \in \cd_0^{\ga,0}(\uxi)$, the (pointwise) product $\bu\star \Xi$ belongs to $\cd^{\ga+\al,\al}(\uxi)$ and 
\begin{equation}\label{multi-noise-1}
\|\bu\star \Xi\|_{\ga+\al,\al}=\|\bu\|_{\ga,0}\ .
\end{equation}
Besides, for any $\bu' \in \cd_0^{\ga,0}(\uxi')$, it holds that
\begin{equation}\label{multi-noise-2}
\|\bu \star \Xi;\bu' \star \Xi\|_{\ga+\al,\al} =\|\bu;\bu'\|_{\ga,0} \ .
\end{equation}
\end{proposition}

\begin{proof}
The statement immediately follows from the two relations contained in (\ref{def:gga-3}).
\end{proof}

\subsection{Integration} Lifting convolution with the heat kernel $G$ is clearly the most tricky step of the procedure. In fact, as we mentionned it earlier, and with the decomposition (\ref{decompo-g}) of $G$ in mind, convolution with $K$ and convolution with $G^\sharp$ will receive distinct treatments.

\smallskip

First, since $G^\sharp$ is smooth on $\R^2$, convolving with this kernel is an easy-to-handle task in our situation, due to the following elementary property.

\begin{lemma}\label{lem:convo-regu}
Consider a distribution $\zeta \in \cac^\al(\R^2)$ with support included in a ball $\cb_\scal(0,r_0)$, for some $r_0\geq 1$. Then $G^\sharp \ast \zeta$ defines a smooth function and one has
\begin{equation}\label{bound}
\| D^k(G^\sharp \ast \zeta)\|_{L^\infty(\R^2)} \leq C_{k,r_0}\, \|\zeta\|_{\al;\R^2} 
\end{equation}
for every multiindex $k$.
\end{lemma}

\begin{proof}
Consider a smooth function $\vp$ with support in $\cb_\scal(0,r_0+\frac12)$ such that $\vp \equiv 1$ on $\cb_\scal(0,r_0)$, and for every $x\in \R^2$, set $G^\sharp_x(y)=G^\sharp(y-x)$. Then, with Definition \ref{def:besov-space} in mind, and as $G^\sharp$ is a smooth function on $\R^2$,
\begin{eqnarray*}
|[(D^k G^\sharp) \ast \zeta](x)|\ = \ | \langle \zeta,(D^k G^\sharp_x)\cdot \vp \rangle | &=&|\langle \zeta,\cs_{\scal,0}^1((D^k G^\sharp_x)\cdot \vp)\rangle |\\
&\lesssim & \|\zeta\|_{\al,r_0+\frac12,\R^2}\cdot \|(D^k G^\sharp_x)\cdot \vp\|_{\cac^2(\R^2)}\ .
\end{eqnarray*}
From here, we can conclude by using Lemma \ref{lem:topo}, together with the uniform estimate
$$\|(D^k G^\sharp_x)\cdot \vp\|_{\cac^2(\R^2)}\lesssim \|D^k G^\sharp\|_{\cac^2(\R^2)} < \infty$$
for every multiindex $k$, which follows from the classical properties of the heat kernel (away from $0$).

\end{proof}


\begin{proposition}\label{prop:convo-smooth-kernel}
Consider an $(\al,K)$-rough path $\uxi$ (resp. $\uxi'$), and for every $\mathbf{v}\in \cd^{\ga+\al,\al}_\al(\uxi)$ with compact support included in $\R^+\times \R$, define
\begin{equation}\label{lift-conv-smooth}
\mathcal{G}^\sharp_{\uxi} \mathbf{v}:=[G^\sharp \ast \mathcal{R}_{\uxi}(\mathbf{v})] \, \mathbf{1}+[(D^{(0,1)}G^\sharp) \ast \mathcal{R}_{\uxi}(\mathbf{v})] \, X_2 \ .
\end{equation}
Then $\rho_T \cdot \mathcal{G}^\sharp_{\uxi} (\mathbf{v})\in \cd_0^{\ga,0}(\uxi)$ and 
\begin{equation}\label{convo-smooth-kernel-1}
\|\rho_T \cdot \mathcal{G}^\sharp_{\uxi} (\mathbf{v})\|_{\ga,0}\leq Q_{\uxi} \cdot T^\ka \cdot  \|\mathbf{v} \|_{\ga+\al,\al}\ ,
\end{equation}
for some constant $\ka >0$. Besides, for any $\mathbf{v}'\in \cd^{\ga+\al,\al}(\uxi ')$ with compact support included in $\R_+\times \R$,
\begin{equation}\label{convo-smooth-kernel-2}
\|\rho_T \cdot\mathcal{G}^\sharp_{\uxi} (\mathbf{v}) ;\rho_T \cdot\mathcal{G}^\sharp_{\uxi'} (\mathbf{v}')\|_{\ga,0}\leq Q_{\uxi,\uxi',\mathbf{v},\mathbf{v}'} \cdot T^\ka \cdot \lcl  \|\mathbf{v} ;\mathbf{v}'\|_{\ga+\al,\al}+\|\uxi;\uxi'\|_{\al;\compac_0} \rcl \ .
\end{equation}
In (\ref{convo-smooth-kernel-1}), $Q_{\uxi}$ is a polynomial expression in $\|\uxi\|_{\al;\compac_0}$, for an appropriate compact set $\compac_0\subset \R^2$, and in (\ref{convo-smooth-kernel-2}), $Q_{\uxi,\uxi',\mathbf{v},\mathbf{v}'}$ is a polynomial expression in $\|\mathbf{v}\|_{\ga+\al,\al}$, $\|\mathbf{v}'\|_{\ga+\al,\al}$, $\|\uxi\|_{\al;\compac_0}$ and $\|\uxi'\|_{\al;\compac_0}$.
\end{proposition}

\begin{proof}
Note first that from the very definition of $\|\cdot \|_{\ga,0}$ (and especially due to the condition $\|x-y\|_\scal \leq \|x;y\|_P$ in the second summand of (\ref{def:norm-ga-eta})), it holds that for any function $\bu:\R^2 \to \struc_{0+}$,
\begin{equation}\label{locali-1}
\|\rho_T \cdot \bu\|_{\ga,0}=\|\rho_T \cdot \bu\|_{\ga,0;[-12T,12T]\times \R} \ .
\end{equation}
Also, a close examination of the conditions (\ref{condition-sup})-(\ref{condition-holder-2}), combined with the bound $\|\rho'_T\|_{L^\infty(\R)} \lesssim T^{-1}$, shows that
\begin{equation}\label{locali-2}
\|\rho_T \cdot \bu\|_{\ga,0;[-12T,12T]\times \R} \lesssim \|\bu\|_{\ga,0;[-12T,12T]\times \R} \ .
\end{equation}
Indeed, observe for instance that for all $x,y$ with $0<|x_1|\leq 12T$, $0<|y_1|\leq 12T$ and $\|x-y\|_\scal \leq \|x;y\|_P$,
\begin{eqnarray*}
\big| [\rho_T(x_1)-\rho_T(y_1)] \cdot \bu^0(x)\big| &\lesssim & T^{-1} \cdot \lln x_1-y_1 \rrn \cdot \|\bu\|_{\ga,0;[-12T,12T]\times \R}\\
&\lesssim & T^{-1} \cdot \|x-y\|_\scal^2 \cdot \|\bu\|_{\ga,0;[-12T,12T]\times \R}\cdot (\|x;y\|^{-\ga}_P \cdot T^{\ga})\\
&\lesssim &\|\bu\|_{\ga,0;[-12T,12T]\times \R} \cdot \|x-y\|_\scal^\ga \cdot \|x;y\|^{-\ga}_P \cdot T \ .
\end{eqnarray*}
As a conclusion of this localization procedure, we can assert that
\begin{equation}\label{localiz-step}
\|\rho_T \cdot \mathcal{G}^\sharp_{\uxi}(\mathbf{v})\|_{\ga,0} \lesssim \| \mathcal{G}^\sharp_{\uxi}(\mathbf{v})\|_{\ga,0;[-12T,12T]\times \R} \ .
\end{equation}
Then another ingredient toward (\ref{convo-smooth-kernel-1}) lies in the fact that as $\bgg$ is compactly supported in $\R_+\times \R$, the same property holds true for $\crr_{\uxi}(\bgg)$ by Lemma \ref{lem:supp-recon}. So, since $G^\sharp(x)=0$ as soon as $x_1\leq 0$, we can assert that, for any multiindex $k$,
\begin{equation}\label{non-anticip}
[(D^k G^\sharp) \ast \crr_{\uxi}(\bgg)](x)=0 \quad\text{if} \ x_1\leq 0 \ .
\end{equation} 
Together with Lemma \ref{lem:convo-regu} and using basic Taylor estimates, it easily entails that
\begin{equation}\label{estim-g-sharp}
\| \mathcal{G}^\sharp_{\uxi}(\mathbf{v})\|_{\ga,0;[-12T,12T]\times \R} =\| \mathcal{G}^\sharp_{\uxi}(\mathbf{v})\|_{\ga,0;[0,12T]\times \R} \lesssim T^\ka\cdot  \|\crr_{\uxi}(\bgg)\|_{\al;\compac} \ ,
\end{equation}
for some parameter $\ka>0$ and some compact set $\compac \subset \R^2$. For instance, if $x\in [0,12T]\times \R$,
\begin{eqnarray*}
|(G^\sharp \ast \crr_{\uxi}(\bgg))(x)| &=&|(G^\sharp \ast \crr_{\uxi}(\bgg))(x_1,x_2)-(G^\sharp \ast \crr_{\uxi}(\bgg))(0,x_2)|\\
&\lesssim &T\cdot \|D^{(1,0)}(G^\sharp\ast \crr_{\uxi}(\bgg))\|_{L^\infty(\R^2)}\ \lesssim\  T\cdot \|\crr_{\uxi}(\bgg)\|_\al \, ,
\end{eqnarray*}
where the last inequality is derived from (\ref{bound}).

\smallskip

\noindent 
By combining (\ref{localiz-step}) and (\ref{estim-g-sharp}) with the property (\ref{recon-global}) of the reconstruction operator, we get the bound (\ref{convo-smooth-kernel-1}).

\smallskip

The proof of (\ref{convo-smooth-kernel-2}) goes along the same lines (localization, non-anticipativity (\ref{non-anticip}) and use of Lemma \ref{lem:convo-regu}), and we leave it to the reader as an exercice.

\end{proof}



Convolving a modelled distribution with $K$ is a much more intricate issue due to the singularity of $G$ at the origin: in \cite{hai-14}, it gives rise to the so-called \emph{multi-level Schauder estimates}, which are more specifically spread out in \cite[Section 5, 6.5 and 7.1]{hai-14}. In our situation, the result can be summed up through the following statement.



\begin{proposition}\label{prop:convo-sing-kernel}
Consider an $(\al,K)$-rough path $\uxi$ (resp. $\uxi'$), and define, for every 
\begin{equation}\label{decomp-g}
\mathbf{v}=\mathbf{v}^0 \, \Xi+\mathbf{v}^1 \, \Xi \ci(\Xi) +\mathbf{v}^2 \, \Xi X_2  \in \cd^{\ga+\al,\al}(\uxi)
\end{equation}
with compact support included in $\R_+ \times \R$, 
\begin{equation}\label{lift-conv}
(\mathcal{K}_{\uxi} \mathbf{v})(x):=[K\ast (\mathcal{R}_{\uxi} \mathbf{v})](x)  \, \1+\mathbf{v}^0(x) \, \ci(\Xi)+[(D^{(0,1)}K)\ast \{(\mathcal{R}_{\uxi} \mathbf{v})-\mathbf{v}^0(x)\cdot \xi\}](x) \, X_2 \ .
\end{equation}
Then $\rho_T\cdot \mathcal{K}_{\uxi}(\mathbf{v})$ is a well-defined element of $\cd_0^{\ga,0}({\uxi})$ and 
\begin{equation}\label{convo-kernel-1}
\|\rho_T \cdot \mathcal{K}_{\uxi} (\mathbf{v})\|_{\ga,0}\leq Q_{\uxi}\cdot  T^{\ka}\cdot \|\mathbf{v} \|_{\ga+\al,\al} \ ,
\end{equation}
for some constant $\ka>0$. Moreover, for any $\mathbf{v}'\in \cd^{\ga+\al,\al}(\uxi')$ with decomposition of the form (\ref{decomp-g}) and compact support included in $\R_+\times \R$, it holds that
\begin{equation}\label{convo-kernel-2}
\|\rho_T \cdot \mathcal{K}_{\uxi} (\mathbf{v}) ;\rho_T \cdot \mathcal{K}_{\uxi'} (\mathbf{v}')\|_{\ga,0}\leq Q_{\uxi,\uxi',\mathbf{v},\mathbf{v}'} \cdot  T^{\ka}\cdot \{ \|\mathbf{v} ;\mathbf{v}'\|_{\ga+\al,\al}+\|\uxi,\uxi'\|_{\al;\compac_1}\} \ .
\end{equation}
In (\ref{convo-kernel-1}), $Q_{\uxi}$ is a polynomial expression in $\|\uxi\|_{\al;\compac_0}$, for an approriate compact set $\compac_0\subset \R^2$, and in (\ref{convo-kernel-2}), $Q_{\uxi,\uxi',\mathbf{g},\mathbf{g}'}$ is a polynomial expression in $\|\mathbf{v}\|_{\ga+\al,\al}$, $\|\mathbf{v}'\|_{\ga+\al,\al}$, $\|\uxi\|_{\al;\compac_0}$ and $\|\uxi'\|_{\al;\compac_0}$.
\end{proposition}

\smallskip

\begin{remark}
At first sight, Formula (\ref{lift-conv}) is not very different from (\ref{lift-conv-smooth}): the $\bv^0$-component has just "slipped" from $X_2$ to $\ci(\Xi)$ in order to counterbalance some lack of regularity. And yet, the estimation of $\mathcal{K}_{\uxi} \mathbf{v}$ is much more knotty than the estimation of $\mathcal{G}^\sharp_{\uxi} \mathbf{v}$, not only because of this modification, but mostly due to the singular behaviour of $K$ at the origin.
\end{remark}

\smallskip

\begin{proof}
First, the localization procedure exhibited in the proof of Proposition \ref{prop:convo-smooth-kernel} (that is the combination of (\ref{locali-1})-(\ref{locali-2})) allows us to assert that $\|\rho_T \cdot \mathcal{K}_{\uxi} (\mathbf{v})\|_{\ga,0} \lesssim \|\mathcal{K}_{\uxi} (\mathbf{v})\|_{\ga,0;[-12T,12T] \times \R}$. Besides, since both $K$ and $\bv$ are compactly supported, this holds true for $\mathcal{K}_{\uxi} (\mathbf{v})$ as well, and we can thus conclude that
$$\|\rho_T \cdot \mathcal{K}_{\uxi} (\mathbf{v})\|_{\ga,0} \lesssim \|\mathcal{K}_{\uxi} (\mathbf{v})\|_{\ga,0;[-12T,12T] \times \compac}\ ,$$
for some appropriate compact set $\compac \subset \R$.

\smallskip

At this point, we are essentially in the same setting as in \cite[Theorem 7.1]{hai-14}, which theoretically provides us with the bound (\ref{convo-kernel-1}). However, for the paper to stay relatively self-contained, we have decided to provide details on this estimation in the appendix, and we therefore refer the reader to Lemma \ref{lem:schauder} for further details regarding the arguments of the rest of the proof of (\ref{convo-kernel-1}).

\smallskip

For the sake of conciseness, we do not elaborate on the proof of (\ref{convo-kernel-2}), but the patient reader could check that (as usual) the estimate goes along the very same lines as for (\ref{convo-kernel-1}).

\end{proof}

\subsection{Initial condition} With formulation (\ref{eq-loc}) in mind, it only remains us to deal with the lift of $G\Psi$ in $\cd^{\ga,0}_0(\uxi)$, which can actually be done in the most natural way, as follows.

\begin{proposition}\label{prop:ini-cond}
For every $\Psi\in L^\infty(\R)$, define
\begin{equation}\label{exp:ini-cond}
(\mathbf{G} \Psi)(x):=\int_{\R} dz \, G(x_1,x_2-z) \Psi(z) \, \1+\int_{\R} dz \, (D^{(0,1)}G)(x_1,x_2-z) \Psi(z) \, X_2 \ .
\end{equation}
Then for any $(\al,K)$-rough path $\uxi$, $\mathbf{G} \Psi \in \cd^{\ga,0}_0(\uxi)$ and 
\begin{equation}\label{ini-cond}
\|\mathbf{G} \Psi\|_{\ga,0} \lesssim  \|\Psi\|_{L^\infty(\R)} \ , 
\end{equation}
where the proportionality constant is independent of $\uxi$.
\end{proposition}

\begin{proof}
The control relies on the basic formula 
$$\int_\R dx_2 \, |(D^k G)(x_1,x_2)| =c_{k} \cdot |x_1|^{-\frac{|k|_\scal}{2}} $$
for every $\la \geq 0$, $x_1 \neq 0$ and $|k|_\scal \leq 3$. For instance, for all $x,y\in \R^2 \backslash P$ satisfying $\|x-y\|_\scal \leq \|x;y\|_P$, one has
\begin{eqnarray*}\label{lift-ini-2}
\lefteqn{\big| \int_{\R} dz \, \big[ G(x_1,x_2-z)-G(y_1,y_2-z)-(x_2-y_2) \cdot (D^{(0,1)}G)(y_1,y_2-z)\big] \cdot \Psi(z) \big|}\\
&\leq & \|\Psi\|_{L^\infty(\R)}\cdot \bigg\{ \int_{\R} dz \, |G(x_1,z)-G(y_1,z)|\\
& & +|x_2-y_2| \cdot \int_0^1 dr \int_{\R}dz \, |(D^{(0,1)}G)(y_1,y_2-z+r(x_2-y_2))-(D^{(0,1)}G)(y_1,y_2-z)| \bigg\} \\
&\lesssim& \|\Psi\|_{L^\infty(\R)}\cdot \big\{ |x_1-y_1| \cdot \|x;y\|_P^{-1}+|x_2-y_2|^2 \cdot \|x;y\|_P^{-1}\big\}\\
&\lesssim& \|\Psi\|_{L^\infty(\R)}\cdot \|x-y\|_\scal^\ga \cdot \|x;y\|_P^{1-\ga} \ .
\end{eqnarray*}
As for the supremum norms, one has obviously
$$\sup_{x\in \R^2\backslash P}\bigg(  \bigg| \int_{\R} dz \, G(x_1,x_2-z) \cdot \Psi(z) \bigg|,\|x\|_P\cdot \bigg| \int_{\R} dz \, (D^{(0,1)}G)(x_1,x_2-z) \cdot \Psi(z) \bigg|\bigg) \lesssim \|\Psi\|_{L^\infty(\R)}\ .$$

\end{proof}

\subsection{Solving the modelled equation}\label{subsec:solving}
Given an $(\al,K)$-rough path $\uxi$ and an initial condition $\Psi\in L^\infty(\R)$, we can now combine successively Propositions \ref{prop:compo}, \ref{prop:multi}, \ref{prop:convo-smooth-kernel}, \ref{prop:convo-sing-kernel} in order to transpose Equation (\ref{eq-loc}) in the space $\cd^{\ga,0}_0(\uxi)$ as
\begin{equation}\label{modelled-equation}
\bu=\mathbf{G}\Psi+\rho_T \cdot \big[ \mathcal{K}_{\uxi}+\mathcal{G}^\sharp_{\uxi}\big] \big( (\rho_+ \cdot \bff(\bu))\star  \Xi\big)  \ .
\end{equation}

\begin{remark}\label{rk:consistency-check}
Let us go back here on the "consistency" condition raised in Remark \ref{rk:lift}, and which explains why we call the procedure a "lift" of the equation. Assume that $\xi$ is a smooth function and that $\xi^{\2}$ is given by the canonical expression (\ref{levy-area-smoo}). Then if $\bu$ satisfies (\ref{modelled-equation}) and if we set $u(x):=(\crr_{\uxi}(\bu))(x)=(\Pi^{\uxi}_x( \bu(x)))(x)$ (by (\ref{reconstruction-smooth})), we can first conclude from (\ref{lift-conv-smooth}), (\ref{lift-conv}) and (\ref{exp:ini-cond}) that
\begin{eqnarray}
u(x)&=&(G_{x_1} \Psi)(x_2)+\rho_T(x_1) \cdot (K\ast \crr_{\uxi}\bv)(x)+\rho_T(x_1) \cdot (G^\sharp \ast \crr_{\uxi}\bv)(x) \nonumber\\
&=&(G_{x_1} \Psi)(x_2)+\rho_T(x_1) \cdot (G\ast \crr_{\uxi}\bv)(x) \ ,\label{cas-smooth-check}
\end{eqnarray}
where $\bv:=\rho_+\cdot (\bv^0\, \Xi+\bv^1 \, \Xi\ci(\Xi)+\bv^2 \, \Xi X_2)$ with $\bv^0,\bv^1,\bv^2$ defined as in Proposition \ref{prop:compo}. Then
\begin{eqnarray*}
(\crr_{\uxi}\bv)(x)&  =&  (\Pi^{\uxi}_x (\bv(x)))(x)\\
&=& \rho_+(x)\cdot \bv^0(x) \cdot \xi(x)+\rho_+(x)\cdot \bv^1(x) \cdot \xi^{\2}_x(x)+\rho_+(x)\cdot \bv^2(x) \cdot \Pi^{\uxi}_x(\Xi X_2)(x)\\
&=&\rho_+(x)\cdot \bv^0(x) \cdot \xi(x)\ =\ \rho_+(x)\cdot F(x_2,u(x)) \cdot \xi(x) \ .
\end{eqnarray*}
By injecting this expression back into (\ref{cas-smooth-check}), we get that $u$ is a classical solution of (\ref{eq-loc}), and hence the consistency condition is indeed satisfied by the above constructions.
\end{remark}

\begin{proposition}\label{prop:sol}
For every fixed $(\al,K)$-rough path $\uxi$ and initial condition $\Psi\in L^\infty(\R)$, there exists a time $T_0=T_0(\uxi,\Psi) >0$ and a radius $R=R(\uxi,\Psi)>0$ such that for every $0<T\leq T_0$, Equation (\ref{modelled-equation}) admits a unique solution $\mathbf{\Phi}(\uxi,\Psi,T)$ within the ball $\cb_{\uxi}(R):=\{\bu \in \cd_0^{\ga,0}(\uxi): \, \|\bu\|_{\ga,0}\leq R\}$.
\end{proposition}

\begin{proof}
As expected, we resort to a fixed-point argument based on the previous estimates. For every $\bu\in \cd^{\ga,0}_0(\uxi)$, denote by $\mathcal{M}_{\uxi,\Psi,T}(\bu)\in \cd^{\ga,0}_0(\uxi)$ the right-hand side of (\ref{modelled-equation}). By using successively (\ref{ini-cond}), (\ref{convo-kernel-1}), (\ref{convo-smooth-kernel-1}), (\ref{multi-noise-1}) and (\ref{bound-compo-1}), we get that
$$\|\mathcal{M}_{\uxi,\Psi,T}(\bu) \|_{\ga,0} \leq Q^{(1)}_{\uxi,\Psi} \cdot \{1+T^{\ka_1}\cdot \|\bu\|_{\ga,0}^2 \} \ ,$$
for some parameter $\ka_1 >0$. Here, $Q^{(1)}_{\uxi,\Psi}$ stands for a polynomial expression in $\|\Psi\|_{L^\infty(\R)}$ and $\|\uxi\|_{\al;\compac_0}$, where $\compac_0$ is a compact set in $\R^2$ depending only on $F$. With this notation, set
$$T_1=T_1(\uxi,\Psi)=\big( Q^{(1)}_{\uxi,\Psi} \cdot \big(1+Q^{(1)}_{\uxi,\Psi}\big)^2\big)^{-\frac{1}{\ka_1}} \ >0 \quad \text{and} \quad R=R(\uxi,\Psi):=1+Q^{(1)}_{\uxi,\Psi} \ .$$
Thus, for any $0<T\leq T_1$ and $\bu\in \cb_{\uxi}(R)$, it holds that $\|\mathcal{M}_{\uxi,\Psi,T}(\bu) \|_{\ga,0} \leq R$, that is, $\cb_{\uxi}(R)$ is invariant through the action of $\mathcal{M}_{\uxi,\Psi,T}$.

\smallskip

\noindent
Then, given two elements $\bu,\bu'\in \cb_{\uxi}(R)$, we can combine, for the fixed model $\uxi$, the bounds (\ref{convo-kernel-2}), (\ref{convo-smooth-kernel-2}), (\ref{multi-noise-2}), (\ref{bound-compo-2}), and assert that for every $0<T\leq T_1$,
$$\|\mathcal{M}_{\uxi,\Psi,T}(\bu);\mathcal{M}_{\uxi,\Psi,T}(\bu')\|_{\ga,0} \leq Q^{(2)}_{\uxi} \cdot T^{\ka_2}\cdot \|\bu;\bu'\|_{\ga,0} \ ,$$
for some parameter $\ka_2>0$ and some polynomial expression $Q^{(2)}_{\uxi}$ in $\|\uxi\|_{\al;\compac_0}$. Finally, set
$$T_0=T_0(\uxi,\Psi)=\inf\big( T_1, \big( 2 \, Q^{(2)}_{\uxi}\big)^{-\frac{1}{\ka_2}}\big) \ >0 \ .$$
In this way, for every $0<T\leq T_0$, the restriction $\mathcal{M}_{\uxi,\Psi,T} : \cb_{\uxi}(R) \to \cb_{\uxi}(R)$ is a contraction map, which ensures the existence of a unique fixed point $\bu$ in $ \cb_{\uxi}(R)$.

\end{proof}

\begin{remark}\label{rk:maximal}
Repeating the above fixed-point argument could actually lead us to the exhibition of a \emph{unique maximal solution for the modelled equation} (given a fixed $(\al,K)$-rough path $\uxi$). In other words, we could show the existence of a time $T>0$, a growing sequence of times $T_n \rightarrow T$ and a sequence $(u^n) \in \cd^{\ga,0}_0(\uxi)$ such that for every $n$, the three following conditions are satisfied: $(i)$ $(u^n)$ satisfies the modelled equation (\ref{modelled-equation}) on $[0,T_n]$; $(ii)$ if $\bu\in \cd^{\ga,0}_0(\uxi)$ satisfies the modelled equation on $[0,T_n]$, then $\bu_{|[0,T_n]}=\bu^n_{|[0,T_n]}$ ; $(iii)$ $\lim_{n\to \infty} \|\mathcal{R}_{\uxi}(\bu^n)(T^n,.)\|_{L^\infty(\R)} =\infty$. Roughly speaking, the strategy of this extension goes as follows: once endowed with a local solution $\bu$ up to some (small) time $T_0$, the equation is reloaded with new starting time $T_0$ and initial condition $\mathcal{R}_{\uxi}(\bu)(T_0,.) \in L^{\infty}(\R)$ (due to Proposition \ref{prop:reconstruction-fonction}), using additionally "time-shifted" topologies (where the hyperplane $\{x\in \R^2: \ x_1=0\}$ is replaced with $\{x\in \R^2: \ x_1=T_0\}$). The success of the procedure relies on the patching result stated in \cite[Proposition 7.11]{hai-14}. 
\end{remark}

We are finally in a position to state the main result of this section.

\begin{proposition}\label{prop:cont-flow}
Consider a sequence of $(\al,K)$-rough paths $\uxi^n$ and initial conditions $\Psi^n \in L^\infty(\R)$ such that, for every compact $\compac \subset \R^2$, 
\begin{equation}\label{converg-assump}
\|\uxi^n;\uxi\|_{\al;\compac} \to 0 \quad \text{and} \quad \|\Psi^n-\Psi\|_{L^\infty(\R)} \to 0 \ ,
\end{equation}
for some $(\al,K)$-rough path $\uxi$ and initial condition $\Psi$. Then, with the notations of Proposition \ref{prop:sol}, there exists a time $T^\ast=T^\ast(\uxi,\Psi) >0$ such that $\mathbf{\Phi}(\uxi^n,\Psi^n,T^\ast)$ is well defined for every $n$ large enough, as well as $\mathbf{\Phi}(\uxi,\Psi,T^\ast)$, and
\begin{equation}\label{conve-1}
\|\mathbf{\Phi}(\uxi^n,\Psi^n,T^\ast);\mathbf{\Phi}(\uxi,\Psi,T^\ast)\|_{\ga,0} \to 0 \ .
\end{equation}
In particular, if we set $\Phi(\uxi^n,\Psi^n,T^\ast)=\mathcal{R}_{\uxi^n}(\mathbf{\Phi}(\uxi^n,\Psi^n,T^\ast))$ and $\Phi(\uxi,\Psi,T^\ast)=\mathcal{R}_{\uxi}(\mathbf{\Phi}(\uxi,\Psi,T^\ast))$, it holds that
\begin{equation}\label{conve-2}
\|\Phi(\uxi^n,\Psi^n,T^\ast)-\Phi(\uxi,\Psi,T^\ast)\|_{L^\infty(\R^2)} \to 0 \ ,
\end{equation}
as well as
\begin{equation}\label{conve-3}
\|\Phi(\uxi^n,\Psi^n,T^\ast)-\Phi(\uxi,\Psi,T^\ast)\|_{\al+2;[s,T^\ast]\times \compac} \to 0
\end{equation}
for every compact set $\compac \subset \R$ and every fixed $s\in (0,T^\ast)$.
\end{proposition}

\begin{proof}
A quick examination of the proof of Proposition \ref{prop:sol} shows that we can choose $T_0$ and $R$ in a such a way that, due to (\ref{converg-assump}), one has
$$T_0(\uxi^n,\Psi^n) \to T_0(\uxi,\Psi) \quad \text{and} \quad R(\uxi^n,\Psi^n) \to R(\uxi,\Psi) \ .$$
Then, for $N$ (fixed) large enough, set $T^\ast_1(\uxi,\Psi):=T_0(\uxi,\Psi) \wedge \inf_{n\geq N} T_0(\uxi^n,\Psi^n) >0$, so that for every $0<T\leq T_1^\ast$ and $n\geq N$, $\Phi(\uxi^n,\Psi^n,T)$ and $\Psi(\uxi,\Psi,T)$ are both well defined. Now combine (\ref{ini-cond}), (\ref{convo-kernel-2}), (\ref{convo-smooth-kernel-2}), (\ref{multi-noise-2}) and finally (\ref{bound-compo-2}) to deduce that
\begin{align*}
& \|\Phi(\uxi^n,\Psi^n,T);\Phi(\uxi,\Psi,T)\|_{\ga,0}\\
& \leq 
Q^{(3)}_{\uxi,\Psi}\cdot \big\{T^{\ka_3}\cdot \|\Phi(\uxi^n,\Psi^n,T);\Phi(\uxi,\Psi,T)\|_{\ga,0}+\|\Psi^n-\Psi\|_{L^\infty(\R)}+\|\uxi^n;\uxi\|_{\al;\compac_0}\big\}
\end{align*}
for some parameter $\ka_3 >0$. As in the proof of Proposition \ref{prop:sol}, $Q^{(3)}_{\uxi,\Psi}$ stands for a polynomial expression in $\|\Psi\|_{L^\infty(\R)}$ and $\|\uxi\|_{\al;\compac_0}$, where $\compac_0$ is a compact set in $\R^2$ depending only on $F$. The convergence result (\ref{conve-1}) follows immediately for the value
$$T^\ast(\uxi,\Psi):=\inf\big( T^\ast_1(\uxi,\Psi),\big( 2Q^{(3)}_{\uxi,\Psi}\big)^{-\frac{1}{\ka_3}} \big) \ >0 .$$
As for the convergence results (\ref{conve-2})-(\ref{conve-3}), they are now mere consequences of Proposition \ref{prop:reconstruction-fonction} and the conditions involved in (\ref{condition-sup})-(\ref{condition-holder-1}).

\end{proof}

\section{Construction of an $(\al,K)$-rough path}\label{sec:construction}

With the result of Proposition \ref{prop:cont-flow} in hand, the route toward Theorem \ref{main-theo}, point $(ii)$, is now quite clear: we need to construct a $K$-Lévy area above the fractional noise $\xi=\partial_t \partial_x X$ involved in the equation, which is the purpose of the present section.

\smallskip

Before we go into the details, let us say a few words about our strategy. Considering the approximation $(X^n)$ of the $(H_1,H_2)$-fractional sheet given by (\ref{approx-noise}), we are going to show that there exists a sequence $\uxi^n$ of $(\al,K)$-rough paths above $\xi^n:=\partial_t \partial_x X^n$ such that $\uxi^n$ converges to an element $\uxi$ with respect to the (set of) norms involved in (\ref{converg-assump}). Given the smoothness of $X^n$, and accordingly the smoothness of $\xi^n$, the canonical choice for such an approximating sequence is given (see Remark \ref{rk:levy-area-smooth}) by $\uxi^n=(\xi^n,\xi^{\2,n})$, with
$$\xi^n(x):=(\partial_t \partial_x X^n)(x) \quad , \quad \xi^{\2,n}_x(y):= [(K\ast \xi^n)(y)-(K\ast \xi^n)(x)] \cdot \xi^n(y) \ .$$
It turns out that the sequence $\xi^{\2,n}$ defined in this way fails to converge in the case we focus on, that is when $2\geq 2H_1+H_2 > \frac53$. To this extent, the situation can be compared with the issue raised by the two-dimensional Brownian parabolic Anderson model, as it is presented in \cite[Section 1.5.1]{hai-14}. Just as in the latter example, we are going to show that there exists positive deterministic constants $C^n_{H_1,H_2}$ such that the sequence of \emph{renormalized} $(\al,K)$-rough paths given for all $x,y\in \R^2$ by
\begin{equation}\label{rp-approx}
\xi^n(x):=(\partial_t \partial_x X^n)(x) \quad , \quad \hat{\xi}^{\2,n}_x(y):= [(K\ast \xi^n)(y)-(K\ast \xi^n)(x)] \cdot \xi^n(y)-C^n_{H_1,H_2} \ ,
\end{equation}
does converge to an $(\al,K)$-rough path $\hat{\uxi}$ above $\xi$. We will then see how this renormalization trick reverberates on the equation itself, through the emergence of the correction term in (\ref{eq-base}) (see Section \ref{subsec:proof-point-ii}).

\subsection{Preliminaries and main statements}\label{subsec:main-statements}
Denote by $\mathcal{E}_\al^K$ the set of $(\al,K)$-rough paths, that is
\begin{multline*}
\mathcal{E}_\al^K:=\\
\{\uxi=(\xi,\xi^{\mathbf{2}}) \in \cac_c^\al(\R^2)\times \pmb{\cac}^{2\al+2}_c(\R^2): \, \text{for all} \ x,y\in \R^2 , \  \xi^{\mathbf{2}}_x-\xi^{\mathbf{2}}_y=[(K\ast \xi)(y)-(K\ast \xi)(x)] \cdot \xi \ \} \ ,
\end{multline*}
and define $d_\al : \mathcal{E}_\al^K \times \mathcal{E}_\al^K \to \R_+$ along the classical globalization procedure:
$$d_\al(\uxi,\uxi')=\sum_{k\geq 0} 2^{-k} \frac{\|\uxi;\uxi'\|_{\al;R_k}}{1+\|\uxi;\uxi'\|_{\al;R_k}} \ ,$$
where $\|\uxi;\uxi'\|_{\al;\compac}$ is defined by (\ref{norm-rp}) and we have set $R_k:=[-4k,4k]^2$.

\begin{proposition}\label{complet}
$(\mathcal{E}_\al^K,d_\al)$ is a complete metric space.
\end{proposition}

\begin{proof}
Although the reasoning only appeals to elementary arguments, we have found it useful to provide a few details here, insofar as Besov-Hairer (semi-)norms are not exactly standard topologies. Let us first check that $d_\al$ does define a metric on $(\mathcal{E}_\al^K,d_\al)$. To this end, let $\uxi=(\xi,\xi^\2)$ and $\pmb{\eta}=(\eta,\eta^\2)$ represent two $(\al,K)$-rough paths such that $d_\al(\uxi,\ueta)=0$. The identification of the first components is easy, because due to Lemma \ref{lem:topo}, if $\vp\in \cac_c^2(\R^2)$ is a test function with support in a ball $\cb_\scal(0,k)$, one has
\begin{eqnarray}
|\langle \xi-\eta,\vp\rangle | &=& |\langle \xi-\eta,\cs_{\scal,0}^1 \vp\rangle |\nonumber\\
&\leq & \|\vp\|_{\cac^2} \cdot \|\xi-\eta\|_{\al,k;\{0\}}\nonumber\\
&\leq & C_k \cdot \|\vp\|_{\cac^2} \cdot \|\xi-\eta\|_{\al;\cb_\scal(0,k)} \ \leq \ C_k \cdot \|\vp\|_{\cac^2} \cdot \|\xi-\eta\|_{\al;R_{k^2}} \ ,\label{first-component}
\end{eqnarray}
and hence $\xi=\eta$. The identification of the second components now follows from the $K$-Chen relation. Indeed, we can proceed as in the proof of Lemma \ref{lem:topo} and decompose $\vp$ as a finite sum $\vp=\sum_{i\in I_k}\cs_{\scal,x_i}^1 \vp_i$ with $\text{supp} \, \vp_i \subset \cb_\scal(0,1)$, $\|\vp_i\|_{\cac^2}\lesssim \|\vp\|_{\cac^2}$ and $x_i\in \cb_\scal(0,k)$. Then, since $\xi=\eta$, one has, for any $x$,
\begin{equation}\label{second-component}
|\langle \xi^\2_x-\eta^\2_x,\vp\rangle | =\Big| \sum_{i\in I_k} \langle \xi^\2_x-\eta^\2_x,\cs^1_{\scal,x_i} \vp_i \rangle\Big|=\Big| \sum_{i\in I_k} \langle \xi^\2_{x_i}-\eta^\2_{x_i},\cs^1_{\scal,x_i} \vp_i \rangle\Big|\leq C_k \cdot \|\vp\|_{\cac^2}\cdot \|\xi^\2-\eta^\2\|_{2\al+2;\cb_\scal(0,k)} \ ,
\end{equation}
which allows us to conclude that $\uxi=\ueta$. 

\smallskip

Consider now a Cauchy sequence $\uxi^n=(\xi^n,\xi^{\2,n})$ in $(\mathcal{E}_\al^K,d_\al)$. For every test-function $\vp$, it is clear by (\ref{first-component}) that $\langle \xi^n,\vp\rangle$ defines a Cauchy sequence in $\R$, which accordingly converges to some element $\langle \xi,\vp\rangle$. Also, for every $k\geq 0$ and $\ep >0$, there exists an integer $N(\ep,k)$ such that for all $n,m\geq N(k,\ep)$,
$$\big|\langle \xi^n-\xi^m,\cs_{\scal,x}^\delta \vp\rangle \big| \leq \ep \cdot \delta^\al \cdot \|\vp\|_{\cac^2} \ ,$$
where the bound holds uniformly over all $\vp\in \cac^2(\cb_\scal(0,1))$, $\delta \in (0,1]$ and $x\in R_k$. By letting $m$ tend to infinity, we retrieve that $\big|\langle \xi^n-\xi,\cs_{\scal,x}^\delta \vp\rangle \big| \leq \ep \cdot \delta^\al \cdot \|\vp\|_{\cac^2}$, which allows us to conclude that $\xi\in \cac^\al_c(\R^2)$ and $\|\xi^n-\xi\|_{\al;R_k} \to 0$ for every $k\geq 0$. 

\smallskip

Then, with the same notation as above, we can rely on the decomposition
\begin{multline*}
\langle \xi^{\2,n}_x-\xi^{\2,m}_x,\vp\rangle =
\sum_{i\in I_k} \Big\{ \langle \xi^{\2,n}_{x_i}-\xi^{\2,m}_{x_i},\cs^1_{\scal,x_i} \vp_i \rangle\\ +[\theta^n(x)-\theta^m(x)-\theta^n(x_i)+\theta^m(x_i)] \cdot \langle \xi^n, \cs^1_{\scal,x_i} \vp_i \rangle +[\theta^m(x)-\theta^m(x_i)] \cdot \langle \xi^n-\xi^m,\cs_{\scal,x_i}^1\vp_i \rangle\Big\}
\end{multline*}
where we have set $\theta^n:=K\ast \xi^n$, to assert (via Lemma \ref{lem:convo-c-al}) that for every $x\in \R^2$, $\langle \xi^{\2,n}_x,\vp \rangle$ is a Cauchy sequence converging to some element $\langle \xi^{\2}_x,\vp\rangle$. The $K$-Chen relation for $\uxi=(\xi,\xi^\2)$ is immediately derived from the $K$-Chen relation satisfied by $\uxi^n$. Finally, we can use the same limit procedure as with $\xi$ to deduce that $\xi^\2\in \pmb{\cac}_c^{2\al+2}(\R^2)$ and $\|\xi^{\2,n}-\xi^\2\|_{2\al+2;R_k} \to 0$ for every $k\geq 0$, which completes the proof of the lemma.

\end{proof}

The following important property, which somehow will play the role of the Garsia-Rodemich-Rumsey Lemma in this setting, is essentially a reformulation of the results of \cite[Section 3]{hai-14}. We recall that the notation $\cd'_2(\R^2)$ in this statement has been introduced in Section \ref{subsec:spaces}.

\begin{lemma}\label{GRR-gene}
Fix $\al\in (-2,0)$. Then there exists a finite set $\Psi$ of compactly supported functions in $\cac^2(\R^2)$ such that, if $\zeta: \R^2 \to \cd'_2(\R^2)$ is a map with increments of the form
$$\zeta_x-\zeta_y=\sum_{i=1,\ldots,r} [\theta^i(x)-\theta^i(y)] \cdot \zeta^{\sharp,i}_y$$
for some $\theta^i \in \cac_c^\la(\R^2)$, $\la\in [0,-\al)$, and $\zeta^{\sharp,i}\in \pmb{\cac}^\al_c(\R^2)$, one has, for every $k\geq 1$,
\begin{equation}\label{bound-GRR-gene}
\|\zeta\|_{\al+\la;R_k}\lesssim \sup_{\psi\in \Psi} \sup_{n\geq 0} \sup_{x\in \Lambda_\scal^n \cap R_{k+1}} 2^{n(\al+\la)} |\langle \zeta_x,\cs_{\scal,x}^{2^{-n}} \psi \rangle |+\sum_{i=1,\ldots,r} \|\theta^i\|_{\la;R_{k+1}}\|\zeta^{\sharp,i}\|_{\al;R_{k+1}} \ ,
\end{equation}
where we have set $\Lambda_\scal^n:=\{(2^{-2n} k_1,2^{-n}k_2), \ k_1,k_2\in \Z\}$
\end{lemma}

\begin{proof}
Since our formulation of the statement slightly differs from the abstract estimates given in \cite[Section 3]{hai-14}, let us briefly emphasize the main ideas leading to the bound (\ref{bound-GRR-gene}). Actually, regarding these particular considerations, the essential contribution of \cite{hai-14} lies in the construction of a basis
$$\{\, \psi^{n,\scal}_y:=2^{-\frac{3n}{2}} \, \cs^{2^{-n}}_{\scal,y}\psi \ : \ n\geq 0 \ , \ \psi\in \Psi \ \text{and} \ y\in \Lambda^n_\scal \, \}$$
of $L^2(\R^2)$ which satisfies the following properties:

\smallskip

\noindent
(i) $\Psi$ is a finite set of functions in $\cac^2(\R^2)$ with support included in $\cb_\scal(0,1)$ \ ;

\smallskip

\noindent
(ii) For every $\vp\in \cac^2(\cb_\scal(0,1))$, every $n\geq 0$ and $\delta \in (0,1]$, it holds that
\begin{equation}\label{wave-1}
\big| \langle \psi^{n,\scal}_y ,\cs_{\scal,x}^\delta \vp\rangle \big| \lesssim \|\vp\|_{L^1(\R^2)} \cdot 2^{\frac{3n}{2}} \ ,
\end{equation}
and
\begin{equation}\label{wave-2}
\big| \langle \psi^{n,\scal}_y ,\cs_{\scal,x}^\delta \vp\rangle \big| \lesssim \|\vp\|_{\cac^2} \cdot \delta^{-5} \cdot 2^{-\frac{7n}{2}} \ ,
\end{equation}
uniformly over all $x,y\in \R^2$.

\smallskip

Once endowed with such a basis, pick $\vp\in \cac^2(\cb_\scal(0,1))$, $x\in R_k$, $\delta\in (0,1]$ and decompose $\langle \zeta_x,\cs^\delta_{\scal,x}\vp\rangle$ as
$$
\langle \zeta_x,\cs^\delta_{\scal,x} \vp\rangle =
\sum_{n\geq 0}\sum_{\psi\in \Psi}\sum_{y\in \Lambda^n_\scal}\big\{ \langle \zeta_y,\psi^{n,\scal}_y \rangle +[\theta^i(x)-\theta^i(y)] \cdot \langle \zeta^{\sharp,i}_y,\psi^{n,\scal}_y \rangle  \big\} \cdot \langle \psi^{n,\scal}_y,\cs_{\scal,x}^\delta \vp\rangle \ .
$$
In particular, if we denote by $M$ the right-hand side of (\ref{bound-GRR-gene}), it holds that
\begin{equation}\label{grr-interm}
|\langle \zeta_x,\cs^\delta_{\scal,x} \vp\rangle| \lesssim M \cdot \sum_{n\geq 0}\sum_{\psi\in \Psi}\sum_{y\in \Lambda^n_\scal} \big\{2^{-n(\al+\la+\frac32)}+\|x-y\|_\scal^\la \cdot 2^{-n(\al+\frac32)} \big\}\cdot |\langle \psi^{n,\scal}_y,\cs_{\scal,x}^\delta \vp\rangle | \ .
\end{equation}
At this point, the key observation is that, for support reasons,
\begin{equation}\label{cond-supps}
\langle \psi^{n,\scal}_y,\cs^\delta_{\scal,x}\vp \rangle =0 \quad \text{if} \quad \|x-y\|_\scal\geq 2^{-n}+\delta \ .
\end{equation}
Now choose $n_0\geq 0$ such that $2^{-(n_0+1)}< \delta \leq 2^{-n_0}$. On the one hand, we can use (\ref{wave-1}) and (\ref{cond-supps}) to assert that
\begin{eqnarray*}
\lefteqn{\sum_{0\leq n\leq n_0}\sum_{\psi\in \Psi}\sum_{y\in \Lambda^n_\scal} \big\{2^{-n(\al+\la+\frac32)}+\|x-y\|_\scal^\la \cdot 2^{-n(\al+\frac32)} \big\}\cdot |\langle \psi^{n,\scal}_y,\cs_{\scal,x}^\delta \vp\rangle |}\\
&\lesssim & \|\vp\|_{L^1(\R^2)} \sum_{0\leq n\leq n_0} 2^{-n(\al+\la)} \cdot \sharp\,\big( \Lambda^n_\scal \cap \cb_\scal(x,2^{-(n-1)}) \big)\\
&\lesssim & \|\vp\|_{L^1(\R^2)} \sum_{0\leq n\leq n_0} 2^{-n(\al+\la)} \ \lesssim \ \|\vp\|_{L^1(\R^2)}\cdot 2^{-n_0(\al+\la)} \ \lesssim \ \|\vp\|_{L^1(\R^2)}\cdot \delta^{\al+\la} \ ,
\end{eqnarray*} 
where we have used the fact that $\al+\la < 0$. On the other hand, by combining (\ref{wave-2}) and (\ref{cond-supps}), we get
\begin{eqnarray*}
\lefteqn{\sum_{ n> n_0}\sum_{\psi\in \Psi}\sum_{y\in \Lambda^n_\scal} \big\{2^{-n(\al+\la+\frac32)}+\|x-y\|_\scal^\la \cdot 2^{-n(\al+\frac32)} \big\}\cdot |\langle \psi^{n,\scal}_y,\cs_{\scal,x}^\delta \vp\rangle |}\\
&\lesssim & \|\vp\|_{\cac^2}\cdot \delta^{-5} \sum_{n>n_0}\big\{ 2^{-n(\al+\la+5)}+2^{-n(\al+5)} \delta^\la \big\}  \cdot \sharp\,\big( \Lambda^n_\scal \cap \cb_\scal(x,2\delta)\big)\\
&\lesssim & \|\vp\|_{\cac^2}\cdot \delta^{-2} \sum_{n>n_0}\big\{ 2^{-n(\al+\la+2)}+2^{-n(\al+2)} \delta^\la \big\}\  \lesssim \ \|\vp\|_{\cac^2}\cdot \delta^{\al+\la} \ ,
\end{eqnarray*} 
where this time we have used the fact that $\al+2>0$. Going back to (\ref{grr-interm}), this achieves the proof of (\ref{bound-GRR-gene}).

\end{proof}

The main result of this section now reads as follows:

\begin{proposition}
Fix $(H_1,H_2) \in (0,1)^2$ such that $\frac53 < 2H_1+H_2 \leq 2$. Then there exists a sequence of positive deterministic reals $C^n_{H_1,H_2}$ such that, if $\hat{\uxi}^n=(\xi^n,\hat{\xi}^{\2,n})$ is defined along (\ref{rp-approx}), one has, for all $\ell\geq 1$, all $m\geq n \geq 1$, all $x\in \R^2$ and all compactly supported $\psi\in \cac^2(\R^2)$ with $\|\psi\|_{\cac^2}\leq 1$,
\begin{equation}\label{estim-1}
\esp\Big[ |\langle \xi^n,\cs_{\scal,x}^{2^{-\ell}}\psi \rangle |^2 \Big] \lesssim   2^{2\ell(3-2H_1-H_2)} \ ,
\end{equation}
\begin{equation}\label{estim-3}
\esp\Big[ |\langle \xi^n-\xi^m,\cs_{\scal,x}^{2^{-\ell}}\psi \rangle |^2 \Big] \lesssim 2^{-n\ep} 2^{2\ell(3-2H_1-H_2+\ep)} \ ,
\end{equation}
\begin{equation}\label{estim-4}
\esp\Big[ |\langle \hat{\xi}^{\2,n}_x-\hat{\xi}^{\2,m}_x,\cs_{\scal,x}^{2^{-\ell}}\psi \rangle |^2 \Big] \lesssim 2^{-n\ep} 2^{2\ell(4-4H_1-2H_2+2\ep)} \ ,
\end{equation}
for some small $\ep >0$, and where the proportional constants are uniform in $\ell,m,n \geq 1$, $x\in \R^2$.
\end{proposition}

\

For the sake of clarity, we postpone the proof of this statement to Sections \ref{subsec:first-level}-\ref{sec:estim-4}, and we first examine how these estimates entail the desired conclusions.

\begin{corollary}
Fix $(H_1,H_2) \in (0,1)^2$ such that $\frac53 < 2H_1+H_2 \leq 2$ and let $\al\in (-\frac43,-3+2H_1+H_2)$. Then there exists a sequence of positive deterministic reals $C^n_{H_1,H_2}$ such that, for all $p\geq 1$, all $k\geq 1$ and all $m\geq n \geq 1$,
\begin{equation}\label{estim-besov-2}
\esp\Big[ \|\hat{\uxi}^n;\hat{\uxi}^m \|_{\al;R_k}^{2p}\Big] \leq C_p \cdot k^2 \cdot 2^{-n\ep p} \ ,
\end{equation}
for some small $\ep >0$, and where the proportional constant $C_p$ is uniform in $n,m,k \geq 1$.
\end{corollary}

\begin{proof}
Note first that since the processes involved in (\ref{estim-1})-(\ref{estim-3}) all belong to a finite sum of Wiener chaoses, we can immediately turn the latter bounds into $L^{2p}(\Omega)$-estimates.

\smallskip

\noindent
Let us now exhibit a useful intermediate bound on $\|\xi^n\|_{\al;R_k}$. By using Lemma \ref{GRR-gene} with $\zeta_x=\xi^n$ (and so $\theta=\zeta^\sharp=0$), we derive that
\begin{eqnarray*}
\|\xi^n\|_{\al;R_k}^{2p}&\lesssim & \sup_{\psi\in \Psi} \sup_{\ell\geq 0} \sup_{x\in \Lambda_\scal^\ell \cap R_{k+1}} 2^{2p\ell\al} |\langle \xi^n,\cs_{\scal,x}^{2^{-\ell}} \psi \rangle |^{2p}\\
&\lesssim & \sum_{\psi\in \Psi} \sum_{\ell\geq 0} \sum_{x\in \Lambda_\scal^\ell \cap R_{k+1}} 2^{2p\ell\al} |\langle \xi^n,\cs_{\scal,x}^{2^{-\ell}} \psi \rangle |^{2p} \ .
\end{eqnarray*}
Consequently, we can invoke (\ref{estim-1}) and conclude that, for $p$ large enough,
\begin{equation}\label{estim-interm-1}
\esp\Big[ \|\xi^n\|_{\al;R_k}^{2p}\Big] \lesssim k^{2}\cdot \sum_{\psi\in \Psi} \sum_{\ell \geq 0} 2^{3\ell} 2^{2p\ell(\al-(-3+2H_1+H_2))} \leq C_p\cdot k^{2} \ .
\end{equation}
It is readily checked that we can use the same argument to deal with $\xi^{(n,m)}:=\xi^n-\xi^m$ (take $\zeta_x=\xi^{(n,m)}$, $\theta=\zeta^\sharp=0$ and replace (\ref{estim-1}) with (\ref{estim-3})), which yields
\begin{equation}\label{estim-interm-2}
\esp \Big[ \| \xi^{(n,m)}\|_{\al;R_k}^{2p} \Big] \leq C_p \cdot k^2 \cdot 2^{-n\ep p} \ ,
\end{equation} 
for $\ep >0$ small enough. As far as $\hat{\xi}^{\2,(n,m)}:=\hat{\xi}^{\2,n}-\hat{\xi}^{\2,m}$ is concerned, observe that it also fits the pattern of Lemma \ref{GRR-gene}, since
$$\hat{\xi}^{\2,(n,m)}_x-\hat{\xi}^{\2,(n,m)}_y=[(K\ast \xi^{(n,m)})(y)-(K\ast \xi^{(n,m)})(x)] \, \xi^n+[(K\ast \xi^m)(y)-(K\ast \xi^{m})(x)] \, \xi^{(n,m)} \ .$$
Therefore, as above, we deduce that
\begin{eqnarray*}
\lefteqn{\esp\Big[\|\hat{\xi}^{\2,(n,m)}\|_{2\al+2;R_k}^{2p}\Big]}\\
&\lesssim & \sum_{\psi\in \Psi} \sum_{\ell\geq 0} \sum_{x\in \Lambda_\scal^\ell \cap R_{k+1}} 2^{2p\ell(2\al+2)} \esp\big[ |\langle \hat{\xi}^{\2,(n,m)}_x,\cs_{\scal,x}^{2^{-\ell}} \psi \rangle |^{2p}\big]\\
& &+\esp\Big[ \|K\ast \xi^{(n,m)}\|_{\al+2;R_{k+1}}^{2p} \|\xi^n\|_{\al;R_{k+1}}^{2p} \Big]+\esp\Big[ \|K\ast \xi^m\|_{\al+2;R_{k+1}}^{2p} \|\xi^{(n,m)}\|_{\al;R_{k+1}}^{2p} \Big]\\
&\lesssim & \sum_{\psi\in \Psi} \sum_{\ell\geq 0} \sum_{x\in \Lambda_\scal^\ell \cap R_{k+1}} 2^{2p\ell(2\al+2)} \esp\big[ |\langle \hat{\xi}^{\2,(n,m)}_x,\cs_{\scal,x}^{2^{-\ell}} \psi \rangle |^{2p}\big]\\
& &\hspace{2cm}+\esp\big[  \|\xi^{(n,m)}\|_{\al;R_{k+1}}^{4p} \big]^{\frac12}\cdot\Big\{\esp\big[  \|\xi^n\|_{\al;R_{k+1}}^{4p} \big]^{\frac12}+\esp\big[  \|\xi^m\|_{\al;R_{k+1}}^{4p} \big]^{\frac12} \Big\}
\end{eqnarray*}
thanks to Lemma \ref{lem:convo-c-al}. The conclusion easily follows from (\ref{estim-4}), (\ref{estim-interm-1}) and (\ref{estim-interm-2}).

\end{proof}

\begin{corollary}\label{cor:converg}
Fix $(H_1,H_2) \in (0,1)^2$ such that $\frac53 < 2H_1+H_2 \leq 2$ and let $\al\in (-\frac43,-3+2H_1+H_2)$. Then there exists a sequence of positive deterministic reals $C^n_{H_1,H_2}$ such that, almost surely, the sequence $\hat{\uxi}^n$ defined by (\ref{rp-approx}) converges in $(\mathcal{E}^K_\al,d_\al)$ to an $(\al,K)$-rough $\hat{\uxi}=(\xi,\hat{\xi}^{\2})$. Besides, $\xi=\partial_t\partial_x X$ in the sense of distributions.
\end{corollary}

\begin{proof}
As a consequence of (\ref{estim-besov-2}), we can assert that, for every $p\geq 1$, $(\hat{\uxi}^n)$ is a Cauchy sequence in $L^{2p}(\Omega,(\mathcal{E}^K_\al,d_\al))$. Therefore, due to Proposition \ref{complet}, it converges to an element $\hat{\uxi}$ and one has, for every $p\geq 1$,
$$\esp\big[d_\al(\hat{\uxi}^n,\hat{\uxi})^{2p}\big] \leq C_p \cdot 2^{-n\ep p} \ .$$
The conclusion easily follows from Borell-Cantelli Lemma. 
\end{proof}

\subsection{Proof of Theorem \ref{main-theo}, point (ii)}\label{subsec:proof-point-ii}
We are now in a position to prove the main statement of the paper. First, by combining Proposition \ref{prop:cont-flow} with Corollary \ref{cor:converg}, we get that, almost surely, there exists a time $T=T^\ast>0$ such that, if $\hat{\uxi}^n$ is defined by (\ref{rp-approx}), the flows $\bu^n:=\pmb{\Phi}(\hat{\uxi}^n,\Psi^n,T)$ are all well defined through Proposition \ref{prop:sol}, as well as $\bu:=\pmb{\Phi}(\hat{\uxi},\Psi,T)$. Besides, if we set $u^n:=\crr_{\hat{\uxi}^n}(\bu^n)$ and $u:=\crr_{\hat{\uxi}}(\bu)$, it holds that
$$\|u^n-u\|_{L^\infty([0,T]\times \R )} \stackrel{n\to\infty}{\longrightarrow} 0 \quad \text{and} \quad \|u^n-u\|_{\al+2;[s,T]\times \compac} \stackrel{n\to\infty}{\longrightarrow} 0$$
for every compact set $\compac \subset \R$ and every $s\in (0,T)$.

\smallskip

\noindent
It only remains us to identify $u^n$ with the classical solution of Equation (\ref{eq-base}) on $[0,T]$. To do so, recall that since $\hat{\uxi}^n$ is smooth, the reconstruction operator $\crr_{\hat{\uxi}^n}$ is explicitly given by (\ref{reconstruction-smooth}), that is, for all $(s,x)\in \R^2$ with $s\neq 0$,
$$\crr_{\hat{\uxi}^n}(\bu)(s,x)=\Pi^{\hat{\uxi}^n}_{(s,x)}(\bu(s,x))(s,x) \ .$$
Therefore, just as in Remark \ref{rk:consistency-check}, by applying $\crr_{{\hat{\uxi}^n}}$ to the modelled equation (\ref{modelled-equation}), we derive that for all $(s,x)\in \R^2$ with $s\neq 0$,
$$u^n(s,x)=\int_{\R} G(s,x-y)\Psi^n(y) \, dy+\rho_T(s)\cdot \big[G\ast \crr_{{\hat{\uxi}^n}}(\bv)\big](s,x) \ ,$$
where $\bv:=\rho_+\cdot (\bv^0\, \Xi+\bv^1 \, \Xi\ci(\Xi)+\bv^2 \, \Xi X_2)$ with $\bv^0,\bv^1,\bv^2$ defined as in Proposition \ref{prop:compo}. Then
\begin{align}
&\crr_{\hat{\uxi}^n}(\bv)(t,y)\ =\ \Pi^{\hat{\uxi}^n}_{(t,y)}(\bv(t,y))(t,y)\nonumber\\
&=\rho_+(t,y)\cdot  \bv^0(t,y)\cdot \Pi^{\hat{\uxi}^n}_{(t,y)}(\Xi)(t,y) +\rho_+(t,y)\cdot \bv^1(t,y)\cdot\Pi^{\hat{\uxi}^n}_{(t,y)}(\Xi \ci(\Xi))(t,y)\nonumber\\
&= \ \rho_+(t,y)\cdot F(y,u^n(t,y))\cdot \xi^n(t,y)\nonumber\\
& \ \ -C^n_{H_1,H_2}\cdot \rho_+(t,y)^2\cdot \rho_T(t)\cdot F(y,u^n(t,y))\cdot (\partial_2 F)(y,u^n(t,y))\ .\label{reconstr-renorm}
\end{align}
In particular, if $s\in [0,T]$,
\begin{eqnarray*}
u^n(s,x)&=& \int_{\R} G(s,x-y)\Psi^n(y) \, dy+\int_0^s G(s-t,x-y) \\
& &\hspace{1cm} \big[ F(y,u^n(t,y))\cdot \xi^n(t,y)-C^n_{H_1,H_2}\cdot F(y,u^n(t,y))\cdot (\partial_2 F)(y,u^n(t,y))\big] \ ,
\end{eqnarray*}
which is precisely the mild formulation of Equation (\ref{eq-base}).

\smallskip

Finally, and as we announced it in the introduction, the identification of $u$ with the Itô solution of (\ref{ito-equation}) in the case where $H_1=\frac12$ is the purpose of Section \ref{sec:identif}.

\

\subsection{Proof of (\ref{estim-1})-(\ref{estim-3})}\label{subsec:first-level} We will use the notation $\Delta$ to represent the rectangular incremental operator, that is,
\begin{equation}\label{rectangular}
( \Delta_{(t,y)}X)(s,x)=X(s+t,x+y)-X(s,x+y)-X(s+t,x)+X(s,x).
\end{equation}
Let us first focus on the estimates for the approximated sheet itself.

\begin{lemma}
Let $(X^n)$ be the approximation of the $(H_1,H_2)$-fractional sheet given by (\ref{approx-noise}). For every $(H_1,H_2)\in (0,1)^2$, every $(s,x),(t,y) \in \R^2$ and every $m\geq n \geq 1$, one has
\begin{equation}\label{estim-sheet-1}
\esp\big[ |(\Delta_{(t,y)} X^n)(s,x)|^2\big] \lesssim |t|^{2H_1} |y|^{2H_2} \ ,
\end{equation}
\begin{equation}\label{estim-sheet-2}
\esp\big[ |(\Delta_{(t,y)} [X^n-X^m])(s,x)|^2\big] \lesssim 2^{-n\ep} \{|t|^{2H_1-2\ep} |y|^{2H_2}+ |t|^{2H_1} |y|^{2H_2-2\ep}\} \ ,
\end{equation}
for every $\ep\in (0,H_1\wedge H_2)$, and where the proportional constants are uniform in $m,n \geq 1$, $(s,x),(t,y)\in \R^2$. 
\end{lemma}

\begin{proof}
It is readily checked that
$$(\Delta_{(t,y)} X^n)(s,x)=c_{H_1,H_2}\int_{D_n} \widehat{W}(d\xi,d\eta) \, e^{\imath \xi s} e^{\imath \eta x} \frac{e^{i\xi t} -1}{|\xi|^{H_1+\frac12}} \frac{e^{i\eta y} -1}{|\eta|^{H_2+\frac12}} \ ,$$
and hence
\begin{eqnarray*}
\esp\big[ |(\Delta_{(t,y)} X^n)(s,x)|^2\big]& = &c_{H_1,H_2}^2 \int_{D_n} d\xi d\eta \, \frac{|e^{i\xi t} -1|^2}{|\xi|^{2 H_1+1}} \frac{|e^{i\eta y} -1|^2}{|\eta|^{2 H_2+1}}\\
& \lesssim & \int_{\R^2} d\xi d\eta \, \frac{|e^{i\xi t} -1|^2}{|\xi|^{2 H_1+1}} \frac{|e^{i\eta y} -1|^2}{|\eta|^{2 H_2+1}} \ \lesssim \ |t|^{2H_1} |y|^{2H_2} \ .
\end{eqnarray*}
Similarly,
\begin{eqnarray*}
\lefteqn{\esp\big[ |(\Delta_{(t,y)} [X^n-X^m])(s,x)|^2\big]}\\
& = & c_{H_1,H_2}^2\int_{D_m\backslash D_n} d\xi d\eta \, \frac{|e^{i\xi t} -1|^2}{|\xi|^{2 H_1+1}} \frac{|e^{i\eta y} -1|^2}{|\eta|^{2 H_2+1}}\\
& \lesssim  & |y|^{2H_2} \int_{|\xi| \geq 2^{2n}} d\xi  \, \frac{|e^{i\xi t} -1|^2}{|\xi|^{2 H_1+1}}+|t|^{2H_2}\int_{|\eta|\geq 2^n}d\eta \, \frac{|e^{i\eta y} -1|^2}{|\eta|^{2 H_2+1}} \\
&\lesssim & |t|^{2H_1-2\ep}|y|^{2H_2} \int_{|\xi| \geq 2^{2n}} d\xi  \, \frac{1}{|\xi|^{1+2\ep}}+|t|^{2H_2}|y|^{2H_2-2\ep}\int_{|\eta|\geq 2^n}d\eta \, \frac{1}{|\eta|^{1+2\ep}} \ .
\end{eqnarray*}

\end{proof}

\begin{proof}[Proof of (\ref{estim-1})-(\ref{estim-3})]
By definition, one has
\begin{eqnarray*}
\langle \xi^n,\cs_{\scal,(s,x)}^{2^{-\ell}} \psi \rangle &=& 2^{6\ell} \iint_{\R^2} dt dy \, (D^{(1,1)} \psi)(2^{2\ell} t,2^\ell y) \, X^n(t+s,y+x) \\
&=&  2^{6\ell} \iint_{\R^2} dt dy \, (D^{(1,1)} \psi)(2^{2\ell} t,2^\ell y) \, (\Delta_{(t,y)}X^n)(s,x) \ , 
\end{eqnarray*}
so
$$|\langle \xi^n,\cs_{\scal,(s,x)}^{2^{-\ell}} \psi \rangle|^2 \leq 2^{12\ell} |\cb_\scal(0,C\cdot 2^{-\ell})| \iint_{\R^2}dt dy \, |(D^{(1,1)} \psi)(2^{2\ell} t,2^\ell y)|^2 |(\Delta_{(t,y)}X^n)(s,x)|^2$$
and the conclusion easily follows from (\ref{estim-sheet-1}). Thanks to (\ref{estim-sheet-2}), (\ref{estim-3}) can be derived from the same argument.

\end{proof}

\subsection{Proof of (\ref{estim-4})}\label{sec:estim-4}

\begin{lemma}\label{lem:estim-fourier-k}
For every $a,b\in [0,1]$ such that $a+b<1$, one has, for every $\xi,\eta\in \R$,
$$|\widehat{K}(\xi,\eta)|\lesssim |\xi|^{-a} |\eta|^{-2b} \ .$$
\end{lemma}

\begin{proof}
It follows immediately from the decomposition $K=\sum_{\ell \geq 0} K_\ell$ introduced in Lemma \ref{lem:decompo-noyau} and the three bounds
$$\|K_\ell\|_{L^1(\R^2)} \lesssim 2^{-2\ell} \quad , \quad \|D^{(1,0)}K_\ell\|_{L^1(\R^2)} \lesssim 1 \quad , \quad  \|D^{(0,2)}K_\ell\|_{L^1(\R^2)} \lesssim 1\ ,$$
for every $\ell \geq 0$.
\end{proof}

We set from now on $\psi^\ell_{x}:=\cs_{\scal,x}^{2^{-\ell}}\psi$ for every $x\in \R^2$.

\smallskip

By definition, one has, for every $(s,x)\in \R^2$,
\begin{equation}\label{first-expression}
\langle  \hat{\xi}^{\2,(n,m)}_{(s,x)}, \psi_{(s,x)}^\ell \rangle=\int_{\R^2} dt dy \, (D^{(1,1)} \psi^{\ell}_{(s,x)})(t,y) \, \Big\{ \mathbf{X}^{\mathbf{2},(n,m)}_{(s,x)}(t,y)-C^{(n,m)}_{H_1,H_2} \cdot t\cdot y\Big\} \ ,
\end{equation}
where we have set $C^{(n,m)}_{H_1,H_2}:=C^{n}_{H_1,H_2}-C^{m}_{H_1,H_2}$,  $\mathbf{X}^{\mathbf{2},(n,m)}:=\mathbf{X}^{\mathbf{2},n}-\mathbf{X}^{\mathbf{2},m}$ and 
\begin{multline*}
\mathbf{X}^{\mathbf{2},n}_{(s,x)}(t,y):=\int_s^t du \int_x^y dz \, (D^{(1,1)} X^n)(u,z) \cdot\\
\int_{\R^2} dr dw \, K(r,w)\cdot \big[ (D^{(1,1)}X^n)(u-r,z-w)-(D^{(1,1)}X^n)(s-r,x-w) \big].
\end{multline*}
Using the representation (\ref{approx-noise}) of $X^n$, we can also write $\mathbf{X}^{\mathbf{2},(n,m)}$ as
\begin{equation}\label{levy-area-complex}
\mathbf{X}^{\mathbf{2},(n,m)}_{(s,x)}(t,y)=-\int_{D_m^2 \backslash D_n^2} d\xi d\eta d\xi_2 d\eta_2 \, \widehat{W}(d\xi,d\eta) \widehat{W}(d\xi_2,d\eta_2) \, A_{(s,x),(t,y)}((\xi_1,\eta_1),(\xi_2,\eta_2)),
\end{equation}
where
\begin{align*}
&A_{(s,x),(t,y)}((\xi,\eta),(\xi_2,\eta_2))\\
&=c_{H_1,H_2}^2\, \frac{\xi_1 \cdot \eta_1}{|\xi|^{H_1+\frac12}|\eta|^{H_2+\frac12}} \frac{\xi_2\cdot \eta_2}{|\xi_2|^{H_1+\frac12}|\eta_2|^{H_2+\frac12}} \widehat{K}(\xi_1,\eta_1) P_{(s,x),(t,y)}((\xi_1,\eta_1),(\xi_2,\eta_2))
\end{align*}
and
$$P_{(s,x),(t,y)}((\xi_1,\eta_1),(\xi_2,\eta_2))=\int_s^t du \int_x^y dz \, e^{\imath \xi_2 u} e^{\imath \eta_2 z} \lcl e^{\imath \xi_1 u} e^{\imath \eta_1 z}-e^{\imath \xi_1 s}e^{\imath \eta_1 x} \rcl.$$
Going back to (\ref{first-expression}), it holds that
\begin{multline}\label{decomp-proof}
\esp\Big[ \big|\langle  \hat{\xi}^{\2,(n,m)}_{(s,x)}, \psi_{(s,x)}^\ell \rangle \big|^2\Big]=\int_{\R^4} dt_1 dy_1 dt_2 dy_2 \, (D^{(1,1)} \psi^\ell_{(s,x)})(t_1,y_1)\cdot (D^{(1,1)} \psi^\ell_{(s,x)})(t_2,y_2)\cdot\\
\bigg\{ \esp \Big[\mathbf{X}^{\mathbf{2},(n,m)}_{(s,x)}(t_1,y_1)\overline{\mathbf{X}^{\mathbf{2},(n,m)}_{(s,x)}(t_2,y_2)}\Big]-C^{(n,m)}_{H_1,H_2} \cdot t_2\cdot y_2 \cdot \esp\big[\mathbf{X}^{\mathbf{2},(n,m)}_{(s,x)}(t_1,y_1)\big]\\
-C^{(n,m)}_{H_1,H_2} \cdot t_1 \cdot y_1 \cdot \esp\Big[ \overline{\mathbf{X}^{\mathbf{2},(n,m)}_{(s,x)}(t_2,y_2)}\Big]+\big( C^{(n,m)}_{H_1,H_2}\big)^2 \cdot t_1 \cdot y_1 \cdot t_2 \cdot y_2 \bigg\}.
\end{multline}
At this point, let us apply Wick's formula to decompose $\esp \Big[\mathbf{X}^{\mathbf{2},(n,m)}_{(s,x)}(t_1,y_1)\overline{\mathbf{X}^{\mathbf{2},(n,m)}_{(s,x)}(t_2,y_2)}\Big]$ as
\begin{align*}
&\esp\Big[\mathbf{X}^{\mathbf{2},(n,m)}_{(s,x)}(t_1,y_1)\overline{\mathbf{X}^{\mathbf{2},(n,m)}_{(s,x)}(t_2,y_2)}\Big]\\
&=\esp\Big[\mathbf{X}^{\mathbf{2},(n,m)}_{(s,x)}(t_1,y_1)\Big]\esp\Big[ \overline{\mathbf{X}^{\mathbf{2},(n,m)}_{(s,x)}(t_2,y_2)}\Big]+S^{(n,m)}_{(s,x)}((t_1,y_1),(t_2,y_2)) \ ,
\end{align*}
with
\begin{align}
& S^{(n,m)}_{(s,x)}((t_1,y_1),(t_2,y_2)) := \nonumber\\
&- \int_{D_m^2\backslash D_n^2} d\xi_1 d\eta_1 d\xi_2 d\eta_2 \, A_{(s,x),(t_1,y_1)}((\xi_1,\eta_1),(\xi_2,\eta_2))\overline{A_{(s,x),(t_2,y_2)}((\xi_1,\eta_1),(\xi_2,\eta_2))}\nonumber\\
&- \int_{D_m^2\backslash D_n^2} d\xi d\eta d\xi_2 d\eta_2 \, A_{(s,x),(t_1,y_1)}((\xi_1,\eta_1),(\xi_2,\eta_2))\overline{A_{(s,x),(t_2,y_2)}((\xi_2,\eta_2),(\xi_1,\eta_1))}.\label{decompo-wick}
\end{align}
Therefore, by choosing $C^n_{H_1,H_2}$ as in the subsequent Lemma \ref{lem:renorm}, we can rewrite (\ref{decomp-proof}) as
$$
\esp\Big[ \big|\langle  \hat{\xi}^{\2,(n,m)}_{(s,x)}, \psi_{(s,x)}^\ell \rangle \big|^2\Big]=\big|R^{(n,m)}_\ell\big|^2+S^{(n,m)}_\ell(s,x)
$$
with
$$S^{(n,m)}_\ell(s,x):= \int_{\R^4} dt_1 dy_1 dt_2 dy_2 \, (D^{(1,1)} \psi_{(s,x)}^\ell)(t_1,y_1) (D^{(1,1)} \psi_{(s,x)}^\ell)(t_2,y_2)\, S^{(n,m)}_{(s,x)}((t_1,y_1),(t_2,y_2)) \ .$$
Using Lemma \ref{lem:renorm} again, the proof of our assertion reduces to showing that
\begin{equation}\label{estim-s-n-0}
|S^{(n,m)}_\ell(s,x)| \lesssim 2^{2\ell (4-4H_1-2H_2+\ep)}2^{-n\ep}
\end{equation}
for some small $\ep >0$. To this end, observe first that by using Cauchy-Schwarz inequality and a basic change of variables, we get that
$$
|S^{(n,m)}_\ell(s,x)|\lesssim \int_{D_m^2 \backslash D_n^2}d\xi_1 d\eta_1 d\xi_2 d\eta_2 \,
\frac{|\hat{K}(\xi_1,\eta_1)|^2}{|\xi_1|^{2H_1-1}|\xi_2|^{2H_1-1}|\eta_1|^{2H_2-1}|\eta_2|^{2H_2-1}} |P^\ell_\psi((\xi_1,\eta_1),(\xi_2,\eta_2))|^2 \ ,
$$
where we have set
$$P^\ell_\psi((\xi_1,\eta_1),(\xi_2,\eta_2))=\int_{\R^2} dt dy \, (D^{(1,1)}\psi^\ell_{(0,0)})(t,y) \int_0^t du \int_0^y dz \, e^{\imath \xi_2 u}e^{\imath \eta_2 z} \{ e^{\imath \xi_1 u} e^{\imath \eta_1 z}-1 \} \ .$$
Let us introduce the following natural domains:
$$F_1^n:=\{(\xi_1,\eta_1,\xi_2,\eta_ 2)\in \R^4: \ |\xi_1|\geq 2^{2n}\} \quad , \quad F_2^n:=\{(\xi_1,\eta_1,\xi_2,\eta_ 2)\in \R^4: \ |\eta_1|\geq 2^{n}\} \ ,$$
$$F_3^n:=\{(\xi_1,\eta_1,\xi_2,\eta_ 2)\in \R^4: \ |\xi_3|\geq 2^{2n}\} \quad , \quad F_4^n:=\{(\xi_1,\eta_1,\xi_2,\eta_ 2)\in \R^4: \ |\eta_2|\geq 2^{n}\} \ .$$
With this notation, it is readily checked that
\begin{eqnarray}
\lefteqn{|S^{(n,m)}_\ell(s,x)|}\nonumber\\
&\lesssim & \sum_{i=1,\ldots,4} \int_{F_i^n}d\xi_1 d\eta_1 d\xi_2 d\eta_2 \,
\frac{|\hat{K}(\xi_1,\eta_1)|^2}{|\xi_1|^{2H_1-1}|\xi_2|^{2H_1-1}|\eta_1|^{2H_2-1}|\eta_2|^{2H_2-1}} |P^\ell_\psi((\xi_1,\eta_1),(\xi_2,\eta_2))|^2 \nonumber\\
&\lesssim & 2^{2\ell(6-4H_1-2H_2)} \sum_{i=1,\ldots,4} \cj^{\ell,n}_i \ ,\label{estim-s-n}
\end{eqnarray}
with
\begin{equation}\label{i-l-n-i}
\cj^{\ell,n}_i=\int_{F_i^{n-\ell}}d\xi_1 d\eta_1 d\xi_2 d\eta_2 \,
\frac{|\hat{K}(2^{2\ell}\xi_1,2^\ell\eta_1)|^2}{|\xi_1|^{2H_1-1}|\xi_2|^{2H_1-1}|\eta_1|^{2H_2-1}|\eta_2|^{2H_2-1}} |P^0_\psi((\xi_1,\eta_1),(\xi_2,\eta_2))|^2 \ .
\end{equation}
For the sake of clarity, we have postponed the estimation of the latter integrals to Section \ref{subsec:tech-estim} (see Lemma \ref{lem:estim-i-l-n-i}) and we only report the result here: there exists a small $\ep >0$ such that for all $i\in \{1,\ldots,4\}$, 
$$| \cj^{\ell,n}_i | \lesssim 2^{-4\ell (1-\ep)} \quad \text{and if} \ n\geq \ell \ , \quad | \cj^{\ell,n}_i | \lesssim 2^{-4\ell (1-\ep)}2^{-\ep (n-\ell)} \ .$$
Going back to (\ref{estim-s-n}), we have thus shown (\ref{estim-s-n-0}), which achieves the proof of (\ref{estim-4}).

\subsection{A few technical estimates}\label{subsec:tech-estim}

\begin{lemma}\label{lem:renorm}
With the notations of the proof of (\ref{estim-4}), it holds that
\begin{equation}\label{decomp-trace}
\int_{\R^2} dt dy \, (D^{(1,1)}\psi^\ell_{(s,x)})(t,y) \, \esp\big[\mathbf{X}^{\mathbf{2},(n,m)}_{(s,x)}(t,y) \big]=C^{(n,m)}_{H_1,H_2}\cdot \bigg( \int_{\R^2}dt dy\, \psi(t,y) \bigg)+R^{(n,m)}_\ell \ ,
\end{equation}
where 
\begin{equation}\label{cstt}
C^n_{H_1,H_2}:=C_{H_1,H_2}\cdot \int_{D_n} d\xi d\eta \, \frac{1}{|\xi|^{2H_1-1}|\eta|^{2H_2-1}} \widehat{K}(\xi,\eta)
\end{equation}
for some constant $C_{H_1,H_2} > 0$, and $R^{(n,m)}_\ell$ is such that
$$|R^{(n,m)}_\ell| \lesssim 2^{\ell( 4+4H_1+2H_2+2\ep)}2^{-n\ep} \ ,$$
for some small $\ep >0$, and where the proportional constant is uniform in $\ell,m,n$.
\end{lemma}


\begin{proof}
First, observe that
$$\int_{\R^2} dt dy \, (D^{(1,1)} \psi^\ell_{(s,x)})(t,y) \, \esp\big[\mathbf{X}^{\mathbf{2},(n,m)}_{(s,x)}(t,y) \big]=\int_{\R^2} dt dy \, (D^{(1,1)} \psi^\ell_{(0,0)})(t,y) \, \esp\big[\mathbf{X}^{\mathbf{2},(n,m)}_{(s,x)}(s+t,x+y) \big]$$
and with the above notations,
\begin{eqnarray*}
\lefteqn{\esp\big[\mathbf{X}^{\mathbf{2},(n,m)}_{(s,x)}(s+t,x+y) \big]}\\
&=& -c_{H_1,H_2}^2 \cdot \int_{D_m\backslash D_n}d\xi d\eta \, \frac{1}{|\xi|^{2H_1-1}|\eta|^{2H_2-1}} \widehat{K}(\xi,\eta) P_{(s,x),(s+t,x+y)}((\xi,\eta),(-\xi,-\eta))\ .
\end{eqnarray*}
Then
$$
P_{(s,x),(s+t,x+y)}((\xi,\eta),(-\xi,-\eta)) = \int_s^{s+t} du \int_x^{x+y} dz \, \big\{1-e^{-\imath \xi (u-s)} e^{-\imath \eta (z-x)} \big\}= t\cdot y +D_{(t,y)}(\xi,\eta),
$$
where we have set $D_{(t,y)}(\xi,\eta):=-\big( \int_0^t du \, e^{-\imath \xi u} \big)\cdot\big( \int_0^y dz \, e^{-\imath \eta z} \big)$, so that
\begin{equation}\label{dec-proof}
\int_{\R^2} dt dy \, (D^{(1,1)} \psi^\ell_{(s,x)})(t,y) \, \esp\big[\mathbf{X}^{\mathbf{2},(n,m)}_{(s,x)}(t,y) \big] 
=C^{(n,m)}_{H_1,H_2}\cdot \bigg( \int_{\R^2}dt dy\,  \psi(t,y) \bigg)+R^{(n,m)}_\ell \ ,
\end{equation}
with
$$R^{(n,m)}_\ell :=-c^2_{H_1,H_2}\cdot \int_{D_m\backslash D_n} d\xi d\eta \, \frac{1}{|\xi|^{2H_1-1}|\eta|^{2H_2-1}} \widehat{K}(\xi,\eta) B_\ell(\xi,\eta)$$
and
$$B_\ell(\xi,\eta):=\int_{\R^2} dt dy \, (D^{(1,1)}\psi^\ell_{(0,0)})(t,y) \, D_{(t,y)}(\xi,\eta).$$

\smallskip

\noindent
It is readily checked that $B_\ell(\xi,\eta)=B_0(2^{-2\ell} \xi,2^{-\ell} \eta)$ and hence
\begin{eqnarray*}
\big| R^{(n,m)}_\ell \big|&=&\bigg| c^2_{H_1,H_2}\cdot \int_{D_m \backslash D_n} d\xi d\eta \, \frac{1}{|\xi|^{2H_1-1}|\eta|^{2H_2-1}} \widehat{K}(\xi,\eta) B_0(2^{-2\ell}\xi,2^{-\ell} \eta)\bigg|\\
&\lesssim &  2^{\ell(6-4H_1-2H_2)}\bigg\{ \int_{|\xi|\geq 2^{2(n-\ell)}} d\xi \int_{\R} d\eta \, \frac{1}{|\xi|^{2H_1-1}|\eta|^{2H_2-1}} |\widehat{K}(2^{2\ell}\xi,2^\ell \eta) B_0(\xi, \eta)| \\
& &\hspace{3cm}+\int_{\R} d\xi \int_{|\eta|\geq 2^{n-\ell} } d\eta \, \frac{1}{|\xi|^{2H_1-1}|\eta|^{2H_2-1}} |\widehat{K}(2^{2\ell}\xi,2^\ell \eta) B_0(\xi, \eta)| \bigg\}.\\
\end{eqnarray*}
Due to the assumption $\frac53 <2H_1+H_2 <2$, we can pick, for any $\ep >0$ small enough, two parameters $a,b\in [0,1]$ such that one has simultaneously $a+b=1-\frac{\ep}{2}$ and
$$\ep <2H_1-1+a <1  \quad , \quad \ep <2H_2-1+2b <1 \ .$$
Therefore, by applying Lemma \ref{lem:estim-fourier-k} to such a pair $(a,b)$, we get that
\begin{multline}\label{r-n-m}
|R^{(n,m)}_\ell | \lesssim 2^{\ell(4-4H_1-2H_2+\ep)}
\bigg\{ \int_{|\xi|\geq 2^{2(n-\ell)}} d\xi \int_{\R} d\eta \, \frac{ |B_0(\xi, \eta)|}{|\xi|^{2H_1+a-1}|\eta|^{2H_2+2b-1}}\\+\int_{\R} d\xi \int_{|\eta|\geq 2^{n-\ell}} d\eta \, \frac{ |B_0(\xi, \eta)|}{|\xi|^{2H_1+a-1}|\eta|^{2H_2+2b-1}}  \bigg\}\ .
\end{multline}
At this point, observe that we can rely on the estimate
$$|D_{(t,y)}(\xi,\eta)| \lesssim \inf \bigg( \frac{1}{|\xi|\,|\eta|}, \frac{|t|}{|\eta|},\frac{|y|}{|\xi|}, |t| |y| \bigg) \ ,$$
which immediately yields
$$| B_0(\xi, \eta)| \lesssim \inf \bigg( \frac{1}{|\xi|\,|\eta|}, \frac{1}{|\eta|},\frac{1}{|\xi|}, 1 \bigg)\ .$$
In particular, the integral
$\int_{\R^2} d\xi d\eta \, \frac{ |B_0(\xi, \eta)|}{|\xi|^{2H_1+a-1}|\eta|^{2H_2+2b-1}} $
is finite, so, if $\ell \geq n$, (\ref{r-n-m}) entails that
$$|R^{(n,m)}_\ell | \lesssim 2^{\ell(4-4H_1-2H_2+\ep)} \lesssim 2^{\ell(4-4H_1-2H_2+2\ep)}2^{-n\ep} \ .$$
On the other hand, if $\ell \leq n$, one has
$$\int_{|\xi|\geq 2^{2(n-\ell)}} d\xi \int_{\R} d\eta \, \frac{|B_0(\xi, \eta)|}{|\xi|^{2H_1+a-1}|\eta|^{2H_2+2b-1}}\lesssim 2^{-2\ep (n-\ell)} \ ,$$
and similarly
$$\int_{\R} d\xi \int_{|\eta|\geq 2^{n-\ell}} d\eta \, \frac{ |B_0(\xi, \eta)|}{|\xi|^{2H_1+a-1}|\eta|^{2H_2+2b-1}}\lesssim 2^{-\ep (n-\ell)}\ , $$
which still allows us to conclude by (\ref{r-n-m}).
\end{proof}

\

\begin{lemma}\label{lem:estim-i-l-n-i}
For all $n\geq 0$, $\ell \geq 0$ and $i\in \{1,\ldots,4\}$, consider the quantity $\cj^{n,\ell}_i$ defined by (\ref{i-l-n-i}). Then for every $\ep >0$ small enough, it holds that
$$| \cj^{\ell,n}_i | \lesssim 2^{-4\ell (1-\ep)} \quad \text{and if} \ n\geq \ell \ , \quad | \cj^{\ell,n}_i | \lesssim 2^{-4\ell (1-\ep)}2^{-\ep (n-\ell)} \ ,$$
where the proportional constants are uniform in $\ell,n$.
\end{lemma}

\begin{proof}
Let us first introduce the two quantities at the core of our argument: for all $\xi_1,\xi_2\in \R$, define
\begin{equation}\label{def-t-psi}
T_\psi(\xi_1)=\bigg( \int_{\R^2} dt dy \, |(D^{(1,1)}\psi)(t,y)|^2 \cdot | \int_0^t du \, e^{\imath \xi_1 u}|^2 \bigg)^{1/2}
\end{equation}
and
\begin{equation}\label{def-q-psi}
Q_\psi(\xi_1,\xi_2)=\bigg( \int_{\R^2} dt dy \, |(D^{(1,1)}\psi)(t,y)|^2 \cdot | \int_0^y dz\int_0^z dw \, e^{\imath \eta_1 w} e^{\imath \eta_2 z} |^2 \bigg)^{1/2} \ .
\end{equation}
It turns out that we have to treat the cases $H_1 >\frac12$ and $\frac13 <H_1 <\frac12$ separately.

\smallskip

\noindent
\textbf{Case 1: $H_1 >\frac12$ .} In this situation, we start from the estimate
\begin{equation}\label{starting}
|P_\psi^0((\xi_1,\eta_1),(\xi_2,\eta_2))|^2 \lesssim |P_\psi^{(1)}((\xi_1,\eta_1),(\xi_2,\eta_2))|^2+|P_\psi^{(2)}((\xi_1,\eta_1),(\xi_2,\eta_2))|^2
\end{equation}
where we have set
$$P_\psi^{(1)}((\xi_1,\eta_1),(\xi_2,\eta_2)):=T_\psi(\xi_1+\xi_2) \cdot \eta_1 \cdot Q_\psi(\eta_1,\eta_2) \quad , $$
$$P_\psi^{(2)}((\xi_1,\eta_1),(\xi_2,\eta_2)):=T_\psi(\eta_2) \cdot \xi_1 \cdot Q_\psi(\xi_1,\xi_2)\ .$$
Then, by Lemma \ref{lem:estim-fourier-k}, we know that for all $a,b\in [0,1]$ such that $a+b=1-\ep$, one has
\begin{multline}\label{boun-1}
\cj^{\ell,n,(1)}_i:=\int_{F_i^{n-\ell}}d\xi_1 d\eta_1 d\xi_2 d\eta_2 \,
\frac{|\hat{K}(2^{2\ell}\xi_1,2^\ell\eta_1)|^2}{|\xi_1|^{2H_1-1}|\xi_2|^{2H_1-1}|\eta_1|^{2H_2-1}|\eta_2|^{2H_2-1}} |P^{(1)}_\psi((\xi_1,\eta_1),(\xi_2,\eta_2))|^2\\
\lesssim 2^{-4\ell (1-\ep)} \int_{F_i^{n-\ell}}d\xi_1 d\eta_1 d\xi_2 d\eta_2 \,
\frac{|T_\psi(\xi_1+\xi_2)|^2\cdot |Q_\psi(\eta_1,\eta_2)|^2}{|\xi_1|^{(2H_1+2a)-1}|\xi_2|^{2H_1-1}|\eta_1|^{(2H_2+4b-2)-1}|\eta_2|^{2H_2-1}} \ .
\end{multline}
Due to the assumptions on the pair $(H_1,H_2)$, we can actually pick $a\in (0,1)$ such that
$$\max (6-8H_1,2H_2) < 4a < \min(2H_2+2,4H_2+1,4-4H_1) \ ,$$
and set $b=1-\ep-a \in (0,1)$, for $\ep$ small enough. For such a choice of $(a,b)$, it is readily checked that the conditions in Lemmas \ref{lem:t-psi} and \ref{lem:q-psi} are all satisfied by the bound (\ref{boun-1}), which leads us to both
$$| \cj^{\ell,n,(1)}_i | \lesssim 2^{-4\ell (1-\ep)} \quad \text{and if} \ n\geq \ell \ , \quad | \cj^{\ell,n,(1)}_i | \lesssim 2^{-4\ell (1-\ep)}2^{-\ep (n-\ell)} \ .$$
The treatment of
$$\cj^{\ell,n,(2)}_i:=\int_{F_i^{n-\ell}}d\xi_1 d\eta_1 d\xi_2 d\eta_2 \,
\frac{|\hat{K}(2^{2\ell}\xi_1,2^\ell\eta_1)|^2}{|\xi_1|^{2H_1-1}|\xi_2|^{2H_1-1}|\eta_1|^{2H_2-1}|\eta_2|^{2H_2-1}} |P^{(2)}_\psi((\xi_1,\eta_1),(\xi_2,\eta_2))|^2$$
is slightly different. Note that $\cj^{\ell,n,(2)}_i \lesssim \cj^{\ell,n,(2,1)}_i +\cj^{\ell,n,(2,2)}_i$, with
\begin{multline}\label{est-1}
\cj^{\ell,n,(2,1)}_i :=\\
2^{-4\ell (a_1+b_1)}\int_{F_i^{n-\ell} \cap \{|\eta_1|\leq 1\}}d\xi_1 d\eta_1 d\xi_2 d\eta_2 \, \frac{|Q_\psi(\xi_1,\xi_2)|^2\cdot |T_\psi(\eta_2)|^2}{|\xi_1|^{(2H_1-2+2a_1)-1}|\xi_2|^{2H_1-1}|\eta_1|^{(2H_2+4b_1)-1}|\eta_2|^{2H_2-1}}
\end{multline}
and
\begin{multline}\label{est-2}
\cj^{\ell,n,(2,2)}_i :=\\
2^{-4\ell (a_2+b_2)} \int_{F_i^{n-\ell} \cap \{|\eta_1|\geq 1\}}d\xi_1 d\eta_1 d\xi_2 d\eta_2 \, \frac{|Q_\psi(\xi_1,\xi_2)|^2\cdot |T_\psi(\eta_2)|^2}{|\xi_1|^{(2H_1-2+2a_2)-1}|\xi_2|^{2H_1-1}|\eta_1|^{(2H_2+4b_2)-1}|\eta_2|^{2H_2-1}} \ .
\end{multline}
Observe also that $|T_\psi(\eta_2)|^2\lesssim \inf(1,\frac{1}{|\eta_2|^2})$. So, in order to apply Lemma \ref{lem:q-psi} to (\ref{est-1}), it suffices to choose $a_1,b_1\in [0,1]$ such that $a_1+b_1=1-\ep$ and
$$\max(2+2H_2,4-4H_1,6-8H_1) < 4a_1 < 4 \ .$$
Similarly, in order to apply Lemma \ref{lem:q-psi} to (\ref{est-2}), it suffices to pick $a_2,b_2\in [0,1]$ such that $a_2+b_2=1-\ep$ and
$$\max(4-4H_1,6-8H_1) < 4a_2 < \min(2+2H_2,4) \ .$$
We can therefore conclude that
$$| \cj^{\ell,n,(2)}_i | \lesssim 2^{-4\ell (1-\ep)} \quad \text{and if} \ n\geq \ell \ , \quad | \cj^{\ell,n,(2)}_i | \lesssim 2^{-4\ell (1-\ep)}2^{-\ep (n-\ell)} \ ,$$
which completes the proof of the lemma in the situation where $H_1 >\frac12$.

\smallskip

\noindent
\textbf{Case 2: $\frac13 < H_1 <\frac12$ .} In particular, it holds that $H_2>\frac23$. Here, we reverse the roles of the variables by replacing (\ref{starting}) with
\begin{equation}
|P_\psi^0((\xi_1,\eta_1),(\xi_2,\eta_2))|^2 \lesssim |P_{\bar{\psi}}^{(1)}((\xi_1,\eta_1),(\xi_2,\eta_2))|^2+|P_{\bar{\psi}}^{(2)}((\xi_1,\eta_1),(\xi_2,\eta_2))|^2 \ ,
\end{equation}
where we have set this time $\bar{\psi}(t,y):=\psi(y,t)$,
$$P^{(1)}_{\bar{\psi}}((\xi_1,\eta_1),(\xi_2,\eta_2)):=T_{\bar{\psi}}(\eta_1+\eta_2) \cdot \xi_1 \cdot Q_{\bar{\psi}}(\xi_1,\xi_2) \ , $$
$$P^{(2)}_{\bar{\psi}}((\xi_1,\eta_1),(\xi_2,\eta_2)):=T_{\bar{\psi}}(\xi_2) \cdot \eta_1 \cdot Q_{\bar{\psi}}(\eta_1,\eta_2) \ .$$
As above, we first deal with
\begin{multline}\label{boun-2}
\cj^{\ell,n,(1)}_i:=\int_{F_i^{n-\ell}}d\xi_1 d\eta_1 d\xi_2 d\eta_2 \,
\frac{|\hat{K}(2^{2\ell}\xi_1,2^\ell\eta_1)|^2}{|\xi_1|^{2H_1-1}|\xi_2|^{2H_1-1}|\eta_1|^{2H_2-1}|\eta_2|^{2H_2-1}} |P^{(1)}_{\bar{\psi}}((\xi_1,\eta_1),(\xi_2,\eta_2))|^2\\
\lesssim 2^{-4\ell (a+b)} \int_{F_i^{n-\ell}}d\xi_1 d\eta_1 d\xi_2 d\eta_2 \,
\frac{|T_{\bar{\psi}}(\eta_1+\eta_2)|^2\cdot |Q_{\bar{\psi}}(\xi_1,\xi_2)|^2}{|\xi_1|^{(2H_1+2a-2)-1}|\xi_2|^{2H_1-1}|\eta_1|^{(2H_2+4b)-1}|\eta_2|^{2H_2-1}} \ ,
\end{multline}
and we pick $a,b\in [0,1]$ satisfying $a+b=1-\ep$ and
$$\max(4-4H_1,6-8H_1,2H_2+2) < 4 a < \min (4,4H_2+1) \ ,$$
so that the conditions in Lemmas \ref{lem:t-psi} and \ref{lem:q-psi} are all met by the bound (\ref{boun-2}), as it can be easily verified.

\smallskip

Also, if we set
$$\cj^{\ell,n,(2)}_i:=\int_{F_i^{n-\ell}}d\xi_1 d\eta_1 d\xi_2 d\eta_2 \,
\frac{|\hat{K}(2^{2\ell}\xi_1,2^\ell\eta_1)|^2}{|\xi_1|^{2H_1-1}|\xi_2|^{2H_1-1}|\eta_1|^{2H_2-1}|\eta_2|^{2H_2-1}} |P^{(2)}_{\bar{\psi}}((\xi_1,\eta_1),(\xi_2,\eta_2))|^2 \ ,$$
we have $\cj^{\ell,n,(2)}_i \lesssim \cj^{\ell,n,(2,1)}_i +\cj^{\ell,n,(2,2)}_i$, with
\begin{multline}\label{est-3}
\cj^{\ell,n,(2,1)}_i :=\\
2^{-4\ell (a_1+b_1)}\int_{F_i^{n-\ell} \cap \{|\xi_1|\leq 1\}}d\xi_1 d\eta_1 d\xi_2 d\eta_2 \, \frac{|Q_\psi(\eta_1,\eta_2)|^2\cdot |T_\psi(\xi_2)|^2}{|\xi_1|^{(2H_1+2a_1)-1}|\xi_2|^{2H_1-1}|\eta_1|^{(2H_2+4b_1-2)-1}|\eta_2|^{2H_2-1}}
\end{multline}
and
\begin{multline}\label{est-4}
\cj^{\ell,n,(2,2)}_i :=\\
2^{-4\ell (a_2+b_2)} \int_{F_i^{n-\ell} \cap \{|\xi_1|\geq 1\}}d\xi_1 d\eta_1 d\xi_2 d\eta_2 \, \frac{|Q_\psi(\eta_1,\eta_2)|^2\cdot |T_\psi(\xi_2)|^2}{|\xi_1|^{(2H_1+2a_2)-1}|\xi_2|^{2H_1-1}|\eta_1|^{(2H_2+4b_2-2)-1}|\eta_2|^{2H_2-1}} \ .
\end{multline}
In (\ref{est-3}), we fix $a_1,b_1\in [0,1]$ such that $a_1+b_1=1-\ep$ and
$$2H_2 < 4a_2 < \min(2H_2+2,4-4H_1,4H_2+1) \ ,$$
while in (\ref{est-4}), we fix $a_2,b_2\in [0,1]$ such that $a_2+b_2=1-\ep$ and
$$\max(2H_2,4-4H_2) < 4a_2 < \min(2H_2+2,4H_2+1) \ .$$
In this way, we are again in positions to appeal to Lemmas \ref{lem:t-psi}-\ref{lem:q-psi} and therefore conclude, which achieves the proof of our statement.
\end{proof}

\

\begin{lemma}\label{lem:t-psi}
Fix $\psi\in \cac^2(\R^2;\R)$ with compact support and let $T_\psi$ be the quantity defined by (\ref{def-t-psi}). Then, for all $\la_1,\la_2 \in (0,2)$ such that $\la_1+\la_2 >3$, the integral
\begin{equation}
\int_{\R^2} dx_1dx_2 \, \frac{|T_\psi(x_1+x_2)|^2}{|x_1|^{\la_1-1}|x_2|^{\la_2-1}}
\end{equation}
is finite. Besides, under the same conditions, for every $c\geq 1$, $i\in \{1,2\}$ and every $\ep>0$ small enough, one has
\begin{equation}\label{estim-t-psi}
\int_{|x_i|\geq c} dx_1dx_2 \, \frac{|T_\psi(x_1+x_2)|^2}{|x_1|^{\la_1-1}|x_2|^{\la_2-1}}\lesssim c^{-\ep} \ .
\end{equation}
\end{lemma}

\begin{proof}
It is a matter of elementary estimates based on the fact that
$$|T_\psi(x_1+x_2)| \lesssim \inf\bigg(1,\frac{1}{|x_1+x_2|}\bigg) \ .$$
We only give details for (\ref{estim-t-psi}). One has, for any small $\ep >0$,
\begin{eqnarray*}
\lefteqn{\int_{|x_1|\geq c} dx_1 \int_{\R} dx_2 \, \frac{|T_\psi(x_1+x_2)|^2}{|x_1|^{\la_1-1}|x_2|^{\la_2-1}}}\\
&\lesssim & \int_{|x_1|\geq c} dx_1 \int_{|x_2|\leq \frac{c}{2}} dx_2 \, \frac{1}{|x_1|^{\la_1-1}|x_2|^{\la_2-1}}\frac{1}{|x_1+x_2|^2}\\
& & +\int_{|x_1|\geq \frac{c}{2}} dx_1 \int_{|x_2|\geq \frac{c}{2}} dx_2 \, \frac{1}{|x_1|^{\la_1-1}|x_2|^{\la_2-1}}\frac{1}{|x_1+x_2|^{1-\ep}}\\
&\lesssim & c^{2-(\la_1+\la_2)} \int_{|x_1|\geq 1} dx_1 \int_{|x_2|\leq \frac{1}{2}} dx_2 \, \frac{1}{|x_1|^{\la_1-1}|x_2|^{\la_2-1}}\frac{1}{|x_1+x_2|^2}\\
& &+c^{3+\ep-(\la_1+\la_2)}\int_{|x_1|\geq \frac{1}{2}} dx_1 \int_{|x_2|\geq \frac{1}{2}} dx_2 \, \frac{1}{|x_1|^{\la_1-1}|x_2|^{\la_2-1}}\frac{1}{|x_1+x_2|^{1-\ep}} \ .
\end{eqnarray*}
It remains to observe that 
\begin{multline*}
\int_{|x_1|\geq 1} dx_1 \int_{|x_2|\leq \frac{1}{2}} dx_2 \, \frac{1}{|x_1|^{\la_1-1}|x_2|^{\la_2-1}}\frac{1}{|x_1+x_2|^2}\\
\lesssim \bigg( \int_{|x_1|\geq 1} \frac{dx_1}{|x_1|^{1+\la_1}} \bigg) \cdot \bigg(\int_{|x_2|\leq \frac12} \frac{dx_2}{|x_2|^{\la_2-1}} \bigg) \ < \infty \ ,
\end{multline*}
and  
\begin{multline*}
\int_{|x_1|\geq \frac{1}{2}} dx_1 \int_{|x_2|\geq \frac{1}{2}} dx_2 \, \frac{1}{|x_1|^{\la_1-1}|x_2|^{\la_2-1}}\frac{1}{|x_1+x_2|^{1-\ep}}\\
\lesssim \bigg( \int_{\frac12}^\infty \frac{dr}{r^{\la_1+\la_2-2-\ep}} \bigg) \cdot \bigg( \int_0^{2\pi} d\theta \, \frac{1}{\lln \cos \theta \rrn^{\la_1-1}\lln \sin \theta \rrn^{\la_2-1}} \frac{1}{\lln \cos \theta +\sin \theta \rrn^{1-\ep}}\bigg) \ < \ \infty \ ,
\end{multline*}
for any $\ep >0$ small enough.
\end{proof}

\

\begin{lemma}\label{lem:q-psi}
Fix $\psi\in \cac^2(\R^2;\R)$ with compact support and let $Q_\psi$ be the quantity defined by (\ref{def-q-psi}). Then, for all $\la_1,\la_2 \in (0,2)$ such that $\la_1+\la_2 >1$, the integral
\begin{equation}\label{estim-q-psi-1}
\int_{\R^2} dx_1dx_2 \, \frac{|Q_\psi(x_1,x_2)|^2}{|x_1|^{\la_1-1}|x_2|^{\la_2-1}}
\end{equation}
is finite. Besides, under the same conditions, for every $c\geq 1$, $i\in \{1,2\}$ and every $\ep>0$ small enough, one has
\begin{equation}\label{estim-q-psi-2}
\int_{|x_i|\geq c} dx_1dx_2 \, \frac{|Q_\psi(x_1,x_2)|^2}{|x_1|^{\la_1-1}|x_2|^{\la_2-1}}\lesssim c^{-\ep} \ .
\end{equation}
\end{lemma}

\begin{proof}
It leans on the following readily-checked estimate: for $i\in \{1,2\}$,
$$|Q_\psi(x_1,x_2)|\lesssim \inf \bigg( 1,\frac{1}{|x_1|},\frac{1}{|x_2|},\frac{1}{|x_1||x_2|}+\frac{1}{|x_i| |x_1+x_2|} \bigg) \ .$$
Based on this bound, one has, for any small $\ep >0$,
\begin{eqnarray}
\lefteqn{\int_{|x_1|\geq c} dx_1\int_{\R}dx_2 \, \frac{|Q_\psi(x_1,x_2)|^2}{|x_1|^{\la_1-1}|x_2|^{\la_2-1}}}\nonumber\\
&\lesssim & \bigg( \int_{|x_1|\geq c} \frac{dx_1}{|x_1|^{1+\la_1}}\bigg) \cdot \bigg( \int_{|x_2|\leq 1} \frac{dx_2}{|x_2|^{\la_2-1}}\bigg)+\int_{|x_1|\geq c} dx_1\int_{|x_2|\geq 1}dx_2 \, \frac{|Q_\psi(x_1,x_2)|}{|x_1|^{\la_1-\ep}|x_2|^{\la_2-\ep}}\nonumber\\
& &\hspace{3cm}+\int_{|x_1|\geq c} dx_1\int_{|x_2|\geq 1}dx_2 \, \frac{|Q_\psi(x_1,x_2)|}{|x_1|^{\la_1-\ep}|x_2|^{\la_2-1}|x_1+x_2|^{1-\ep}}\ .\label{proof-q}
\end{eqnarray}
To bound the second summand in (\ref{proof-q}), pick $a_1\in [0,1]\cap (1-\la_1,\la_2)$, so that, for any $\ep >0$ small enough,
\begin{multline*}
\int_{|x_1|\geq c} dx_1\int_{|x_2|\geq 1}dx_2 \, \frac{|Q_\psi(x_1,x_2)|}{|x_1|^{\la_1-\ep}|x_2|^{\la_2-\ep}}\\
\lesssim \bigg( \int_{|x_1|\geq c} \frac{dx_1}{|x_1|^{\la_1+a_1-\ep}}\bigg) \cdot \bigg( \int_{|x_2|\geq 1} \frac{dx_2}{|x_2|^{\la_2+1-a_1-\ep}}\bigg)\lesssim c^{1-(\la_1+a_1-\ep)} \ .
\end{multline*}
Then, in order to estimate
$$\cj_c:=\int_{|x_1|\geq c} dx_1\int_{|x_2|\geq 1}dx_2 \, \frac{|Q_\psi(x_1,x_2)|}{|x_1|^{\la_1-\ep}|x_2|^{\la_2-1}|x_1+x_2|^{1-\ep}} \ ,$$
observe first that, without loss of generality, we can assume that $\la_1,\la_2\in (0,1)$. Under this assumption, there exists $a_2\in [0,1] \cap (\la_2-1,1-\la_1)$, and for such a value of $a_2$, one has, for any $\ep >0$ small enough,
\begin{eqnarray*}
\cj_c &\lesssim &\int_{|x_1|\geq c} dx_1\int_{|x_2|\geq 1}dx_2 \, \frac{1}{|x_1|^{\la_1-\ep+a_2}|x_2|^{\la_2-a_2}|x_1+x_2|^{1-\ep}}\\
&\lesssim & c^{-\ep} \cdot \bigg( \int_1^\infty \frac{dr}{r^{\la_1+\la_2-3\ep}}\bigg) \cdot \bigg( \int_0^{2\pi} \frac{d\theta}{\lln \cos \theta\rrn^{\la_1+a_2-2\ep}\lln \sin \theta \rrn^{\la_2-a_2}\lln \cos \theta+\sin \theta \rrn^{1-\ep}}\bigg) \ \lesssim \ c^{-\ep} \ ,
\end{eqnarray*}
which completes the proof of (\ref{estim-q-psi-2}). The estimation of (\ref{estim-q-psi-1}) can clearly be done along the same lines.

\end{proof}

\subsection{Estimation of the renormalization constant}

At this point, we have shown the convergence result of Theorem \ref{main-theo}, point $(ii)$, for the constant $C^n_{H_1,H_2}$ explicitly given by (\ref{cstt}). Let us now complete the proof of the statement with an asymptotic estimate of this constant.

\begin{proposition}
Let $C^n_{H_1,H_2}$ be the sequence defined by (\ref{cstt}). Then, as $n$ tends to infinity, it holds that
$$
C^n_{H_1,H_2}\sim 
\left\lbrace
\begin{array}{ll}
c^1_{H_1,H_2}\cdot  2^{2n(2-2H_1-H_2)}& \quad \text{if} \quad \frac53 < 2H_1+H_2 <2 \ ,\\
c^2_{H_1,H_2}\cdot  n & \quad \text{if} \quad 2H_1+H_2 =2 \ ,
\end{array}
\right.
$$
for some constants $c^1_{H_1,H_2},c^2_{H_1,H_2} >0$.
\end{proposition}

\begin{proof}
Recall the decomposition $K(s,x)=\sum_{k\geq 0} 2^k K_0(2^{2k}s,2^k x)$ given by Lemma \ref{lem:decompo-noyau}, which leads us to
$$\widehat{K}(\xi,\eta)=\sum_{k\geq 0} 2^{-2k} \widehat{K_0}\big(2^{-2k} \xi,2^{-k} \eta\big) \quad ,$$
uniformly over $(\xi,\eta)\in \R^2$. Therefore, it is readily checked that
\begin{eqnarray}
C^n_{H_1,H_2}&=& C_{H_1,H_2}\cdot \sum_{k\geq 0} 2^{-2k} \int_{D_n} d\xi d\eta \, \frac{\widehat{K_0}(2^{-2k}\xi,2^{-k}\eta)}{|\xi|^{2H_1-1}|\eta|^{2H_2-1}}\nonumber\\
&=& C_{H_1,H_2}\cdot 2^{2n(2-2H_1-H_2)} \sum_{k\geq 0} 2^{-2(k-n)} \int_{[-1,1]^2} d\xi d\eta \, \frac{\widehat{K_0}(2^{-2(k-n)}\xi,2^{-(k-n)}\eta)}{|\xi|^{2H_1-1}|\eta|^{2H_2-1}}\nonumber\\
&=& C_{H_1,H_2}\cdot 2^{2n(2-2H_1-H_2)} \sum_{k\geq -n} 2^{-2k} \int_{[-1,1]^2} d\xi d\eta \, \frac{\widehat{K_0}(2^{-2k}\xi,2^{-k}\eta)}{|\xi|^{2H_1-1}|\eta|^{2H_2-1}} \ .\label{rk-cstt-explicit}
\end{eqnarray}
In the case where $\frac53 < 2H_1+H_2<2$, the conclusion can now be derived from the dominated convergence theorem, by using the fact that the sum
$$S_{H_1,H_2}:= \sum_{k\in \Z} 2^{-2k} \int_{[-1,1]^2} d\xi d\eta \, \frac{|\widehat{K_0}(2^{-2k}\xi,2^{-k}\eta)|}{|\xi|^{2H_1-1}|\eta|^{2H_2-1}}$$
is finite. Indeed, since $2H_1+H_2 <2$, we can pick $a,b \geq 0$ such that
$$2H_1-1+a <1 \quad , \quad 2H_2-1+b <1 \quad ,\quad a+\frac{b}{2} >1 \ ,$$
and then
\begin{align*}
&S_{H_1,H_2}\\
 &\lesssim \bigg( \sum_{k\geq 0} 2^{-2k}\bigg) \int_{[-1,1]^2} \frac{d\xi d\eta}{|\xi|^{2H_1-1}|\eta|^{2H_2-1}}+\bigg( \sum_{k\geq 0} 2^{k(2-2a-b)}\bigg) \int_{[-1,1]^2} \frac{d\xi d\eta}{|\xi|^{2H_1-1+a}|\eta|^{2H_2-1+b}} \ < \, \infty \, .
\end{align*}
If $2H_1+H_2=2$, then, as $n$ tends to infinity, one has
\begin{eqnarray*}
\lefteqn{\sum_{k\geq -n} 2^{-2k} \int_{[-1,1]^2} d\xi d\eta \, \frac{\widehat{K_0}(2^{-2k}\xi,2^{-k}\eta)}{|\xi|^{2H_1-1}|\eta|^{2H_2-1}}}\\ &=&\sum_{0\leq k\leq n} 2^{2k} \int_{[-1,1]^2} d\xi d\eta \, \frac{\widehat{K_0}(2^{2k}\xi,2^{k}\eta)}{|\xi|^{2H_1-1}|\eta|^{2H_2-1}}+O(1)\\
&=&\sum_{0\leq k\leq n} \int_{D_k} d\xi d\eta \, \frac{\widehat{K_0}(\xi,\eta)}{|\xi|^{2H_1-1}|\eta|^{2H_2-1}}+O(1)\\
&=&n\cdot \int_{\R^2} d\xi d\eta \, \frac{\widehat{K_0}(\xi,\eta)}{|\xi|^{2H_1-1}|\eta|^{2H_2-1}}-\sum_{0\leq k\leq n} \int_{\R^2 \backslash D_k} d\xi d\eta \, \frac{\widehat{K_0}(\xi,\eta)}{|\xi|^{2H_1-1}|\eta|^{2H_2-1}}+O(1)\ ,
\end{eqnarray*}
and it is readily checked that $\sum_{0\leq k\leq n} \int_{\R^2 \backslash D_k} d\xi d\eta \, \frac{\widehat{K_0}(\xi,\eta)}{|\xi|^{2H_1-1}|\eta|^{2H_2-1}}=O(1)$.

\end{proof}

\begin{remark}
If we look closer at the proof of \cite[Lemma 5.5]{hai-14} for the decomposition (\ref{decompo-g}) of the heat kernel $G$, we see that, with the notations of Lemma \ref{lem:decompo-noyau}, we can actually choose $K_0$ in such a way that for every $(t,x)\in \R^2$,
$$G(t,x)=\sum_{n\in \Z} 2^{-2n} (\cs_{\scal,0}^{2^{-n}}K_0)(t,x) \ ,$$
which immediately entails that for every $(\xi,\eta) \in \R^2$,
$$\widehat{G}(\xi,\eta)=\sum_{n\in \Z} 2^{-2n} \widehat{K_0}(2^{-2n}\xi,2^{-n}\eta) \ .$$
In the case where $\frac53 < 2H_1+H_2 < 2$, and thanks to the estimations of the previous proof, we can then assert that
\begin{eqnarray*}
\lim_{n\to \infty} \sum_{k\geq -n} 2^{-2k} \int_{[-1,1]^2} d\xi d\eta \, \frac{\widehat{K_0}(2^{-2k}\xi,2^{-k}\eta)}{|\xi|^{2H_1-1}|\eta|^{2H_2-1}} &=&\int_{[-1,1]^2} d\xi d\eta \, \frac{\widehat{G}(\xi,\eta)}{|\xi|^{2H_1-1}|\eta|^{2H_2-1}}\\
&=& \sqrt{\frac{2}{\pi}}\int_{[-1,1]^2} d\xi d\eta \, \frac{|\xi|^{1-2H_1} |\eta|^{1-2H_2}}{\eta^2-2\imath \xi}\\
&=& \sqrt{\frac{2}{\pi}}\int_{[-1,1]^2} d\xi d\eta \, \frac{|\xi|^{1-2H_1} |\eta|^{3-2H_2}}{\eta^4+4 \xi^2} \ , 
\end{eqnarray*} 
which, going back to (\ref{rk-cstt-explicit}), provides us with an explicit constant $c_{H_1,H_2}^1$ satisfying Theorem \ref{main-theo}, point $(ii)$.
\end{remark}

\section{Identification of the limit}\label{sec:identif}

Let us now turn to the proof of the last assertion in point $(ii)$ of Theorem \ref{main-theo}. Thus, our aim in this section is to identify, in the situation where $H_1=\frac12$ and $H_2>\frac23$, the limit $Y:=\pmb{\Phi}(\hat{\uxi},\Psi,T)$ exhibited in Section \ref{subsec:proof-point-ii} (and based on the constructions of Section \ref{sec:construction}) with the classical Itô solution of (\ref{ito-equation}). Let us indeed recall that if $H_1=\frac12$, then the noise $\xi=\partial_t\partial_x X^{H_1,H_2}$ under consideration defines a cylindrical Wiener process (that we also denote by $dW^{H_2}$) with spatial covariance described by the formula: for every test-functions $\vp,\psi$ on $\R^2$,
$$\esp\big[\langle dW^{H_2},\vp\rangle \langle dW^{H_2},\psi \rangle  \big]=c_{H_2} \int_{\R^3} dt dx dy \, \vp(t,x) \psi(t,y) \, |x-y|^{2H_2-2} \ .$$

\smallskip

Thanks to the results of \cite{hairer-pardoux} (and more specifically by a straightforward adaptation of the arguments in the proof of \cite[Theorem 6.2]{hairer-pardoux}), one can easily check that the identification of $Y$ with the Itô solution of (\ref{ito-equation}) reduces to an identification at the level of the model. Our last assertion in Theorem \ref{main-theo} is therefore a consequence of the following identity:

\begin{theorem}\label{theo:identif}
Fix $H_1=\frac12$, $H_2\in (\frac23,1)$, and consider the $(\al,K)$-rough path $\hat{\uxi}=(\xi,\hat{\xi}^{\2})$ given by Corollary \ref{cor:converg}. Then, for every $(s,x)\in \R^2$ and every smooth test-function $\psi$ with compact support included in the set $\{(t,y): \ t > 0\}$, one has almost surely
\begin{equation}\label{ito-identity}
\langle \hat{\xi}^{\2}_{(s,x)}, \psi_{(s,x)} \rangle =\int_s^\infty \langle \psi_{(s,x)}(t,.) \, [(K\ast \xi)(t,.)-(K\ast \xi)(s,x)] , dW^{H_2}_t \rangle \ ,
\end{equation}
where $\psi_{(s,x)}(t,y):=\psi(t-s,y-x)$ and the integral in the right-hand side is understood in the Itô sense.
\end{theorem}

The first step of our strategy towards (\ref{ito-identity}) will consist in an identification for the approximated quantity $\langle \hat{\xi}^{\2,n}_{(s,x)}, \psi_{(s,x)} \rangle$, where $\hat{\xi}^{\2,n}_{(s,x)}$ is defined by (\ref{rp-approx}) (with $C^n_{H_1,H_2}$ given by (\ref{cstt})). As $\xi^n$ is obtained through a space-time regularization of the original Wiener process, it is both natural and convenient to frame this study within a general Gaussian setting, namely the one provided by Malliavin calculus.

\smallskip

Given a centered Gaussian field $\{Z(s,x); \ s,x\in \R\}$ defined on a complete probability space $(\Omega, \cf,\mathbb{P})$, we will denote by $\ch_Z$ the Hilbert space associated with $Z$, that is the closure of the linear space generated by the functions $\{\1_{[s_1,s_2]\times [x_1,x_2]}; s_1,s_2,x_1,x_2 \in \R\}$ with respect to the semi-definite positive form
$$\langle \1_{[s_1,s_2]\times [x_1,x_2]},\1_{[t_1,t_2]\times [y_1,y_2]} \rangle_{\ch_Z}=\esp\big[ (\Delta_{(s_2-s_1,x_2-x_1)}Z)(s_1,x_1)\cdot (\Delta_{(t_2-t_1,y_2-y_1)}Z)(t_1,y_1) \big] \ ,$$
where the notation $\Delta$ has been introduced in (\ref{rectangular}). Besides, we denote by $D^Z$ the Malliavin derivative with respect to $Z$, and by $\delta^Z$ the associated divergence operator (or Skorohod integral). We here refer the reader to \cite{nualart} or \cite[Sections 5 and 6]{chouk-tindel} for an exhaustive presentation of these objects, together with their classical properties.

\smallskip

In the sequel, we will also be led to involve the family of operators $\cq^n_{\al_1,\al_2}$ ($\al_1,\al_2\in (0,1)$), resp. $\cq_{\al_1,\al_2}$, defined for every measurable, compactly-supported function $\vp$ and every $\xi,\eta \in \R$ as
$$\cq^n_{\al_1,\al_2}(\vp)(\xi,\eta):=c_{\al_1,\al_2}\1_{\{(\xi,\eta) \in \cd_n\}} \frac{\xi\cdot \eta}{|\xi|^{\al_1+\frac12} |\eta|^{\al_2+\frac12}}\, \widehat{\vp}(-\xi,-\eta) \ ,$$
resp.
$$\cq_{\al_1,\al_2}(\vp)(\xi,\eta):=c_{\al_1,\al_2} \frac{\xi\cdot \eta}{|\xi|^{\al_1+\frac12} |\eta|^{\al_2+\frac12}}\, \widehat{\vp}(-\xi,-\eta) \ ,$$
where $c_{\al_1,\al_2}$ is the same constant as in the representation (\ref{representation-sheet}).

\smallskip

With these considerations in mind, our first intermediate result reads as follows:

\begin{proposition}\label{1st-prop-ident}
Fix $(H_1,H_2)\in (0,1)^2$ such that $\frac53 < 2H_1+H_2<2 $, and consider the distribution $\hat{\xi}^{\2,n}_{(s,x)}$ defined by (\ref{rp-approx}), with $C^n_{H_1,H_2}$ given by (\ref{cstt}). Then for every $(s,x)\in \R^2$ and every smooth test-function $\psi$, it holds that
\begin{equation}\label{identity-approx}
\langle \hat{\xi}^{\2,n}_{(s,x)}, \psi_{(s,x)} \rangle=\delta^{X^n}(V^n_{(s,x)})+r^n_{H_1,H_2}(\psi) \ ,
\end{equation}
where
\begin{equation}\label{def:v-n}
V^n_{(s,x)}(t,y):=\psi_{(s,x)}(t,y) \cdot \{ (K\ast D^{(1,1)}X^n)(t,y)-(K\ast D^{(1,1)}X^n)(s,x)\}
\end{equation}
and
\begin{equation}\label{r-n-identif-part}
r^n_{H_1,H_2}(\psi):=-c_{H_1,H_2}^2\int_{\cd_n} \frac{d\xi d\eta}{|\xi|^{2H_1-1}|\eta|^{2H_2-1}} \widehat{K}(\xi,\eta)\widehat{\psi}(\xi,\eta) \ . 
\end{equation}
\end{proposition}

\begin{proof}
We follow the arguments of the proof of \cite[Proposition 5.8]{chouk-tindel}. Consider a sequence of partitions $\pi_k=(\pi_k^1,\pi_k^2)$ of the support of $\psi_{(s,x)}$, with mesh tending to $0$ as $k$ tends to infinity. For every $(t_i,y_j)\in \pi_k$, write $\square_{ij}=\square_{t_i,y_j}=[t_i,t_{i+1}[ \times [x_j,x_{j+1}[$. Using the basic rules of Malliavin calculus, we can first write
\begin{equation}\label{rule}
\delta^{X_n}(V^n_{(s,x)}(t_i,y_j) \1_{\square_{ij}})=V^n_{(s,x)}(t_i,y_j) (\Delta_{(t_{i+1}-t_i,y_{j+1}-y_j)} X^n)(t_i,y_j)-\langle D^{X^n}(V^n_{(s,x)}(t_i,y_j)),\1_{\square_{ij}}\rangle_{\ch_{X^n}} \ .
\end{equation}
Now, by \cite[Lemma 6.1]{chouk-tindel}, we can rely on the identity
\begin{equation}\label{identit-prod}
\langle D^{X^n}(V^n_{(s,x)}(t_i,y_j)),\1_{\square_{ij}}\rangle_{\ch_{X^n}}  =\langle D^{\widehat{W}}(V^n_{(s,x)}(t_i,y_j)), \cq^n_{H_1,H_2}(\1_{\square_{ij}}) \rangle_{L^2(\R^2)} \ .
\end{equation}
It is easy to check that the Malliavin derivative $D^{\widehat{W}}(V^n_{(s,x)}(t_i,y_j))$ is explicitly given for all $\xi,\eta \in \R$ as
$$D^{\widehat{W}}(V^n_{(s,x)}(t_i,y_j))(\xi,\eta)=c_{H_1,H_2}\psi_{(s,x)}(t_i,y_j) \widehat{K}(\xi,\eta) \1_{\{(\xi,\eta)\in \cd_n\}} \frac{\xi \cdot \eta}{|\xi|^{H_1+\frac12}|\eta|^{H_2+\frac12}}\{e^{\imath t_i \xi}e^{\imath y_j \eta}-e^{\imath s \xi} e^{\imath x\eta}\} \ ,$$
and therefore, by (\ref{identit-prod}), it holds that
\begin{multline*}
\langle D^{X^n}(V^n_{(s,x)}(t_i,y_j)),\1_{\square_{ij}}\rangle_{\ch_{X^n}} \\
=c_{H_1,H_2}^2\int_{\cd_n} \frac{d\xi d\eta}{|\xi|^{2H_1-1}|\eta|^{2H_2-1}} \widehat{K}(\xi,\eta) \bigg[\psi_{(s,x)}(t_i,y_j)\{e^{\imath t_i \xi}e^{\imath y_j \eta}-e^{\imath s \xi} e^{\imath x\eta}\}\int_{t_i}^{t_{i+1}} dt \int_{y_j}^{y_{j+1}} dy \, e^{-\imath t\xi}e^{-\imath y \eta}  \bigg] \ .
\end{multline*}
At this point, observe that for all $\xi,\eta \in \R$, 
\begin{align*}
&\sum_{(t_i,y_j)\in \pi_k}\bigg[\psi_{(s,x)}(t_i,y_j)\{e^{\imath t_i \xi}e^{\imath y_j \eta}-e^{\imath s \xi} e^{\imath x\eta}\}\int_{t_i}^{t_{i+1}} dt \int_{y_j}^{y_{j+1}} dy \, e^{-\imath t\xi}e^{-\imath y \eta}  \bigg]\\
& \stackrel{k\to \infty}{\longrightarrow}\int_{\R^2} dt dy \, \psi_{(s,x)}(t,y) \{1-e^{\imath (s-t)\xi} e^{\imath (x-y)\eta}\} \ ,
\end{align*}
which readily leads us to
$$\sum_{(t_i,y_j)\in \pi_k} \langle D^{X^n}(V^n_{(s,x)}(t_i,y_j)),\1_{\square_{ij}}\rangle_{\ch_{X^n}}
\stackrel{k\to \infty}{\longrightarrow} C^n_{H_1,H_2}\int_{\R^2} dt dy \, \psi(t,y)+r^n_{H_1,H_2}(\psi) \ ,
$$
where $C^n_{H_1,H_2}$ is defined by (\ref{cstt}) and $r^n_{H_1,H_2}(\psi)$ by (\ref{r-n-identif-part}).

\smallskip

\noindent
We can finally take the sum over the points in $\pi_k$ in (\ref{rule}) and let $k$ tend to infinity to derive the conclusion. Indeed, just as in \cite[Proposition 5.7]{chouk-tindel}, it is easy to check that
$$\sum_{(t_i,y_j)\in \pi_k} V^n_{(s,x)}(t_i,y_j) \stackrel{k\to \infty}{\longrightarrow} V^n_{(s,x)} \quad \text{in} \ L^2(\Omega;\ch_{X^n}) \ ,$$
which is sufficient to assert that $\sum_{(t_i,y_j)\in \pi_k} \delta^{X^n}(V^n_{(s,x)}(t_i,y_j))  \stackrel{k\to \infty}{\longrightarrow} \delta^{X^n}(V^n_{(s,x)})$ in $L^2(\Omega)$. As for the fact that $r_n(\psi)$ tends to $0$ as $n$ tends to infinity, it follows from the smoothness of $\psi$, along similar estimates as those in the proof of Lemma \ref{lem:renorm}.

\end{proof}

The whole point now is to pass to the limit in the right-hand side of (\ref{identity-approx}). This is the purpose of our next result, which completes the proof of Theorem \ref{theo:identif}.  

\begin{proposition}
Fix $H_1=\frac12$ and $H_2\in (\frac23,1)$. Then, with the notations of Proposition \ref{1st-prop-ident}, the following assertions hold true:

\smallskip

\noindent
$(i)$ For every smooth test-function $\psi$ with support in the set $\{(t,y): \ t>0\}$, one has 
$$r^n_{H_1,H_2}(\psi)\stackrel{n\to \infty}{\longrightarrow}0 \ .$$

\smallskip

\noindent
$(ii)$ For every $(s,x)\in \R^2$ and every smooth test-function $\psi$, one has
\begin{equation}\label{second-step}
\delta^{X^n}(V^n_{(s,x)}) \stackrel{n\to \infty}{\longrightarrow} \delta^X(V_{(s,x)}) \quad \text{in} \ L^2(\Omega) \ ,
\end{equation}
where $V^n_{(s,x)}$ is defined by (\ref{def:v-n}) and $V_{(s,x)}(t,y):=\psi_{(s,x)}(t,y) \cdot \{ (K\ast \xi)(t,y)-(K\ast \xi)(s,x)\}$.
Besides, $\delta^X(V_{(s,x)})$ coincides with the right-hand side of (\ref{ito-identity}).
\end{proposition}

\begin{proof}
$(i)$ Using the isometry property of the Fourier transform, we easily obtain, for every fixed $\eta \in \R$,
$$\int_{\R} d\xi \, \widehat{K}(\xi,\eta) \widehat{\psi}(\xi,\eta)=-c\int_{\R^2} dx dy \, e^{-\imath x\eta}e^{-\imath y \eta} \int_{\R} ds \, K(s,x) \psi(-s,y) \ ,$$
and the latter integral vanishes for support reasons.

\smallskip

\noindent
$(ii)$ We first use \cite[Lemma 6.1]{chouk-tindel} to write the difference $\delta^{X^n}(V^n_{(s,x)})-\delta^X(V_{(s,x)})$ as
$$\delta^{X^n}(V^n_{(s,x)})-\delta^X(V_{(s,x)})=\delta^{\widehat{W}}(\cq^n_{\frac12,H_2}(V^n_{(s,x)})-\cq_{\frac12,H_2}(V_{(s,x)})) \ ,$$
and from this, the proof of (\ref{second-step}) reduces to showing that
$$\cq^n_{\frac12,H_2}(V^n_{(s,x)})-\cq_{\frac12,H_2}(V_{(s,x)})\stackrel{n\to \infty}{\longrightarrow} 0 \quad \text{in} \ L^2(\Omega;L^2(\R^2)) \ .$$
To do so, write
\begin{align*}
&\| \cq^n_{\frac12,H_2}(V^n_{(s,x)})-\cq_{\frac12,H_2}(V_{(s,x)}) \|_{L^2(\R^2)}\\
& \leq \| \cq_{\frac12,H_2}(V^n_{(s,x)}-V_{(s,x)}) \|_{L^2(\R^2)}+ \| [\cq^n_{\frac12,H_2}-\cq_{\frac12,H_2}](V^n_{(s,x)}) \|_{L^2(\R^2)} \ =: \ I_n+II_n \ .
\end{align*}
In order to bound $I_n$, we first combine the subsequent Lemma \ref{lem:embed-1} with Lemma \ref{lem:convo-c-al} to get that
$$|I_n| \lesssim \|V^n_{(s,x)}-V_{(s,x)}\|_{L^\infty(\R^2)} \lesssim \|\xi^n-\xi\|_{\al;\compac} \ ,$$
for some appropriate compact set $\compac$ and $\al\in (-\frac43,-2+H_2)$. Thanks to (\ref{estim-besov-2}), we can immediately conclude that $I_n \stackrel{n\to \infty}{\longrightarrow} 0$ in $L^2(\Omega)$.

\smallskip

\noindent
As for $II_n$, observe first that for $\ep>0$ small enough, it holds that
$$|II_n|^2 \lesssim 2^{-2n\ep} \{ \| \cq_{\frac12,H_2-\ep}(V^n_{(s,x)})\|_{L^2(\R^2)}^2+\| \cq_{\frac12-\ep,H_2}(V^n_{(s,x)})\|_{L^2(\R^2)}^2 \} \ ,$$
and so, by Lemma \ref{lem:embed-1}, $|II_n| \lesssim 2^{-n\ep} \|V^n_{(s,x)}\|_{\ep,\compac}$, for some appropriate compact set $\compac$. As above, we can now combine Lemma \ref{lem:convo-c-al} with (\ref{estim-besov-2}) to derive that $II_n \stackrel{n\to \infty}{\longrightarrow} 0$ in $L^2(\Omega)$.

\smallskip

Finally, the identification of $\delta^X(V_{(s,x)})$ as an Itô integral with respect to $dW^{H_2}$ is the result of \cite[Lemma 2.10]{grorud-pardoux}. 
\end{proof}

It only remains us to prove the following Sobolev-embedding-type property:

\begin{lemma}\label{lem:embed-1}
Let $\al_1\in (0,\frac12)$, $\al_2 \in (\frac12,1)$ and $\ep >0$ such that $\frac14-\frac{\al_1}{2}+\ep <1$. Then for any $a,b>0$ and every function $\vp\in \cac^{\frac14-\frac{\al_1}{2}+\ep}_c(\R^2)$ with support contained in the rectangle $[-a,a]\times [-b,b]$, it holds that
\begin{equation}
\| \cq_{\al_1,\al_2}(\vp)\|_{L^2(\R^2)} \leq c_{a,b} \, \|\vp\|_{\frac14-\frac{\al_1}{2}+\ep;[-a,a]\times [-b,b]} \ ,
\end{equation}
as well as
\begin{equation}\label{sobo-2}
\| \cq_{\frac12,\al_2}(\vp)\|_{L^2(\R^2)} \leq c_{a,b} \, \|\vp\|_{L^\infty(\R^2)} \ .
\end{equation}
\end{lemma}

\begin{proof}
Set $\ka:=\frac12-\al_1$ and for $s,x,\eta \in \R$, set $\psi_\eta(s):=\int_{\R} dx \, e^{-\imath x \eta} \vp(s,x)$, $\vp_s(x):=\vp(s,x)$. Then, using standard Sobolev inequalities, we successively obtain that
\begin{eqnarray*}
\lefteqn{\| \cq_{\al_1,\al_2}(\vp)\|_{L^2(\R^2)}^2\ \lesssim \  \int_{\R} \frac{d\eta}{|\eta|^{2\al_2-1}} \int_{\R} d\xi \, (1+|\xi|^2)^\ka |\widehat{\psi_\eta}(\xi)|^2}\\
&\lesssim &\int_{\R} \frac{d\eta}{|\eta|^{2\al_2-1}} \bigg[ \int_{\R} ds \, \psi_\eta(s)^2 +\iint_{\R^2} \frac{dsdt}{|s-t|^{1+2\ka}} |\psi_\eta(s)-\psi_\eta(t)|^2 \bigg]\\
&\lesssim& \int_{\R} ds \int_{\R} \frac{d\eta}{|\eta|^{2\al_2-1}} |\widehat{\vp_s}(\eta)|^2+\iint_{\R^2} \frac{dsdt}{|s-t|^{1+2\ka}} \int_{\R} \frac{d\eta}{|\eta|^{2\al_2-1}} |(\widehat{\vp_s-\vp_t})(\eta)|^2\\
&\lesssim& \int_{\R} ds \int_{\R} dx \bigg( \int_{\R} dy \, \frac{\vp(s,y)}{|x-y|^{\frac32-\al_2}} \bigg)^2+\iint_{\R^2} \frac{ds dt}{|s-t|^{1+2\ka}}\int_{\R}dx \bigg( \int_{\R} dy \, \frac{\vp(s,y)-\vp(t,y)}{|x-y|^{\frac32 -\al_2}} \bigg)^2\\
&\leq& c_b \, \sup_{y\in [-b,b]} \bigg[\int_{\R} ds \, |\vp(s,y)|^2+\iint_{\R^2} \frac{ds dt}{|s-t|^{1+2\ka}} \, |\vp(s,y)-\vp(t,y)|^2  \bigg] \\
&\leq&  c_{a,b} \, \|\vp\|_{\frac{\ka}{2}+\ep;[-a,a]\times [-b,b]} \ .
\end{eqnarray*} 
The bound (\ref{sobo-2}) follows from similar estimates.

\end{proof}

\section{The Young case}\label{sec:young}

We conclude the paper with the proof of Theorem \ref{main-theo}, point $(i)$, and see how the condition $2H_1+H_2>2$ allows for a drastic simplification of the modelling procedure exhibited in Section \ref{sec:mod-equation}. With the result of Corollary \ref{cor:converg} in mind, we consider from now on that the noise $\xi$ involved in the equation belongs to $\cac_c^\al(\R^2)$ for some fixed
\begin{equation}
\boxed{\al\in (-1,0)}\ .
\end{equation}
Under this assumption, the argument will importantly rely on the following property.

\begin{lemma}\label{lem:young}
Let $K$ be defined as in Lemma \ref{lem:decompo-noyau}. Then for every $\zeta\in \cac_c^\al(\R^2)$, every compact set $\compac \subset \R^2$ and every $x,y\in \compac$ such that $\|x-y\|_\scal \leq 1$, it holds that
$$|[K\ast \zeta](x)-[K\ast \zeta](y)| \lesssim \|x-y\|_\scal \cdot \|\zeta\|_{\al;\text{rect}(\compac)} \ ,$$
where $\text{rect}(\compac)$ stands for the smallest rectangle $[x_1,x_2]\times [y_1,y_2]$ that contains $\compac$.
\end{lemma}

\begin{proof}
Decompose the increment $[K\ast \zeta](x)-[K\ast \zeta](y)$ as the sum of two terms $A_{xy}$ and $B_{xy}$, with
$$A_{xy}=  [K\ast \zeta](x)-[K\ast \zeta](y)-(x_2-y_2) \cdot [(D^{(0,1)}K)\ast \zeta](y) \ ,$$
$$B_{xy}=(x_2-y_2) \cdot [(D^{(0,1)}K)\ast \zeta](y) \ .$$
As far as $B_{xy}$ is concerned, we immediately get that
$$|B_{xy}|\lesssim \|\xi\|_{\al;\compac} \cdot \|x-y\|_\scal \cdot \sum_{n\geq 0} 2^{-n(\al+1)}\lesssim \|\xi\|_{\al;\compac} \cdot \|x-y\|_\scal \ .$$
In order to estimate $A_{xy}$, pick $i\geq 0$ such that $2^{-(i+1)} \leq \|x-y\|_\scal \leq 2^{-i}$ and denote by $A^n_{xy}$ the expression derived from $A_{xy}$ by replacing each occurence of $K$ with $K_n$. On the one hand, we have
\begin{eqnarray*}
\sum_{n>i}|A^n_{xy}| &\leq & \sum_{n>i} \{|[K_n\ast \zeta](x)|+|[K_n\ast \zeta](y)|+|x_2-y_2| \cdot | [(D^{(0,1)}K_n)\ast \zeta](y)| \}\\
&\lesssim & \|\zeta\|_{\al;\compac} \cdot \sum_{n>i} \{2^{-n(\al+2)}+\|x-y\|_\scal\cdot 2^{-n(\al+1)}\} \ \lesssim \ \|\zeta\|_{\al;\compac}  \cdot \|x-y\|_\scal^{\al+2} \ .
\end{eqnarray*} 
For $n\leq i$, expand $A^n_{xy}$ as
\begin{multline*}
A^n_{xy}=(x_1-y_1) \cdot \int_0^1 dr \, [(D^{(1,0)}K_n)\ast \zeta](y_1+r(x_1-y_1),x_2)\\
+\frac12 (x_2-y_2)^2 \cdot \int_0^1 dr_1\int_0^{r_1} dr_2 \, [(D^{(0,2)}K_n)\ast \zeta](y_1,y_2+r_2(x_2-y_2)) \ .
\end{multline*}
With this decomposition in hand, we readily deduce that
$$\sum_{0\leq n\leq i} |A^n_{xy}| \lesssim \|\zeta\|_{\al;\text{rect}(\compac)} \cdot \|x-y\|_\scal^2 \cdot \sum_{0\leq n\leq i} 2^{-n\al}\lesssim \|\zeta\|_{\al;\text{rect}(\compac)} \cdot \|x-y\|_\scal^{\al+2} \ .$$

\end{proof}

As we announced it, the framework of the lift procedure reduces to a minimum here. Namely, 
$$\struc_{\al}:=\text{Span}(\Xi) \quad , \quad \struc_0:=\text{Span}(\1) \quad , \quad \struc:=\struc_\al \oplus \struc_0 \ ,$$
with commutative product $\star$ in $\struc$ given by
$$\1 \star \1=\1 \quad , \quad \1 \star \Xi=\Xi \quad , \quad \Xi \star \Xi=0 \ .$$
Then define the model $(\Pi^\xi,\gga^\xi)$ for $\struc$ by the formulas:
$$\Pi^\xi(\1)=1 \ , \ \Pi^\xi_x(\Xi)=\xi \ , \ \gga^\xi_{xy}(\1)=\1 \ , \ \gga^\xi_{xy}(\Xi)=\Xi \ .$$
In particular, the "L{\'e}vy area" term no longer comes into the picture here, and one has, along the same lines as in Proposition \ref{prop:model},
$$\|(\Pi^\xi,\gga^\xi)\|_{\al;\compac} \lesssim \|\xi\|_{\al;\compac} \quad , \quad \|(\Pi^\xi,\gga^\xi);(\Pi^{\xi_2},\gga^{\xi_2})\|_{\al;\compac} \lesssim \|\xi-\xi_2\|_{\al;\compac} \ . $$
Next, just as in Section \ref{sec:mod-equation}, we denote by $\crr_{\xi}$ the reconstruction operator associated with this model and characterized by the two relations (\ref{recon-global})-(\ref{recon-local}) (where $\xi$ must be substituted for $\uxi$). Also, we denote by $\cd^{\ga,\eta}_\be(\xi)$ the corresponding spaces of modelled distributions, defined along Formula (\ref{def:norm-ga-eta}).

\smallskip

In fact, in this situation, it turns out that we can fix the background space as $\cd^{\ga,0}_0(\xi)$ with
$$\boxed{\ga\in (-\al,1) } \ .$$
Otherwise stated, any element $\bu\in \cd^{\ga,0}_0(\xi)$ is reduced to a single component $\bu^0\, \1$ in $\struc_0$ which satisfies
$$\|\bu\|_{\ga,0}=\sup_{x\in \R^2} |\bu^0(x)| +\sup_{x,y\in \R^2: \, \|x-y\|_\scal \leq \|x;y\|_P}  \frac{|\bu^0(x)-\bu^0(y)|}{\|x-y\|_\scal^\ga\cdot \|x;y\|_P^{-\ga}} \ < \ \infty \ .$$
Thus, the "lift" operation clearly loses all its relevancy in this setting (dealing with $\bu \in \cd^{\ga,0}_0(\xi)$ or $\bu^0\in \cac^\ga(\R^2)$ is just equivalent, at least away from $P$), and we only keep to this formalism for a direct comparison with the situation described in Section \ref{sec:mod-equation}.

\smallskip

So recall that we want to transpose in $\cd^{\ga,0}_0(\xi)$, along the same steps as before, the localized equation
\begin{equation}\label{eq-loc-young}
u(x)=(G_{x_1} \Psi)(x_2)+\rho_T(x_1) \cdot (G\ast [\rho_+ \cdot F(u) \cdot \xi])(x) \ ,
\end{equation}
where $\rho_+$ and $\rho_T$ stand for the two cut-off functions introduced in Section \ref{subsec:prelim}. 

\smallskip

First, it is readily checked that if $\bu=\bu^0 \, \1 \in \cd^{\ga,0}_0(\xi)$ and $F\in \cac^\infty_{\compac}(\R^2)$, then the element $\rho_+ \cdot \bff(\bu)$ trivially defined by
$$[\rho_+\cdot \bff(\bu)](x)=(\rho_+(x_1)\cdot F(x_2,\bu^0(x)))\, \1 \ ,$$
belongs to $\cd^{\ga,0}_0(\xi)$, and the two bounds (\ref{bound-compo-1})-(\ref{bound-compo-2}) remain valid in this setting. Also, the product
$$\bv(x)=[(\rho_+\cdot \bff(\bu))\star \Xi](x):=(\rho_+(x_1)\cdot F(x_2,\bu^0(x)))\, \Xi \ ,$$
clearly defines an element of $\cd^{\ga+\al,\al}_{\al}(\xi)$ and (\ref{multi-noise-1})-(\ref{multi-noise-2}) still prevail.

\smallskip

At this point, note that, since $\ga+\al>0$, the reconstruction $\crr_\xi\bv$ of $\bv$ does define an element of $\cac_c^\al(\R^2)$ that is compactly supported in $\R_+\times \R$. Thanks to Lemma \ref{lem:convo-regu}, the convolution with $G^\sharp$ can therefore be (locally) lifted as $[\rho_T\cdot \mathcal{G}^\sharp_\xi(\bv)](x)=\rho_T(x_1) [G^\sharp\ast \crr_\xi(\bv)](x) \, \1$, and the two bounds (\ref{convo-smooth-kernel-1})-(\ref{convo-smooth-kernel-2}) remain valid. Using Lemma \ref{lem:young}, the same conclusions actually hold true for the convolution with $K$, lifted as $[\rho_T\cdot \mathcal{K}_\xi(\bv)](x)=\rho_T(x_1) [K\ast \crr_\xi(\bv)](x) \, \1$.

\smallskip

Finally, lifting the initial-condition term $G\Psi$ can be readily done for any $\Psi\in L^\infty(\R)$, since, with the arguments and notation of the proof of Proposition \ref{prop:ini-cond}, 
\begin{multline*}
\big| \int_{\R} dz \, G(x_1,x_2-z) \cdot \Psi(z)-\int_{\R} dz \, G(y_1,y_2-z)\cdot \Psi(z) \big|\\
\lesssim \|\Psi\|_{L^\infty(\R)} \cdot \{|x_1-y_1| \cdot \|x;y\|_P^{-1}+|x_2-y_2| \cdot \|x;y\|_P^{-\frac12}\}\lesssim \|\Psi\|_{L^\infty(\R)} \cdot \|x-y\|_\scal^\ga \cdot \|x;y\|_P^{-\ga} \ .
\end{multline*}

\smallskip

By putting these successive observations together, we find ourselves in the very same situation as in Section \ref{subsec:solving} for the equation
\begin{equation}\label{modelled-equation-young}
\bu=\mathbf{G}\Psi+\rho_T \cdot \big[ \mathcal{K}_{\xi}+\mathcal{G}^\sharp_{\xi}\big] \big( (\rho_+ \cdot \bff(\bu))\star  \Xi\big)  \ ,
\end{equation}
interpreted in $\cd^{\ga,0}_0(\xi)$. It is therefore clear that the arguments of the proofs of Propositions \ref{prop:sol} and \ref{prop:cont-flow} can be transposed at once in this setting, which leads us to the following statement:

\begin{proposition}\label{prop:sol-young}
$(i)$ For every path $\xi \in \cac^\al_c(\R^2)$ and initial condition $\Psi\in L^\infty(\R)$, there exists a time $T_0=T_0(\xi,\Psi) >0$ and a radius $R=R(\xi,\psi)>0$ such that for every $0<T\leq T_0$, Equation (\ref{modelled-equation-young}) admits a unique solution $\mathbf{\Phi}(\xi,\Psi,T)$ within the ball $\cb_{\xi}(R):=\{\bu \in \cd_0^{\ga,0}(\xi): \, \|\bu\|_{\ga,0}\leq R\}$.

\smallskip

\noindent
$(ii)$ Consider a sequence of paths $\xi^n\in \cac^\al_c(\R^2)$ and initial conditions $\Psi^n \in L^\infty(\R)$ such that, for every compact $\compac \subset \R^2$, 
\begin{equation}\label{converg-assump-young}
\|\xi^n-\xi\|_{\al;\compac} \to 0 \quad \text{and} \quad \|\Psi^n-\Psi\|_{L^\infty(\R)} \to 0 \ ,
\end{equation}
for some path $\xi\in \cac_c^\al(\R^2)$ and initial condition $\Psi\in L^\infty(\R)$. Then there exists a time $T^\ast=T^\ast(\xi,\Psi) >0$ such that $\mathbf{\Phi}(\xi^n,\Psi^n,T^\ast)$ is well defined for every $n$ large enough, as well as $\mathbf{\Phi}(\xi,\Psi,T^\ast)$, and
\begin{equation}\label{conve-1-young}
\|\mathbf{\Phi}(\xi^n,\Psi^n,T^\ast);\mathbf{\Phi}(\xi,\Psi,T^\ast)\|_{\ga,0} \to 0 \ .
\end{equation}
In particular, if we set $\Phi(\xi^n,\Psi^n,T^\ast)=\mathcal{R}_{\xi^n}(\mathbf{\Phi}(\xi^n,\Psi^n,T^\ast))$ and $\Phi(\xi,\Psi,T^\ast)=\mathcal{R}_{\xi}(\mathbf{\Phi}(\uxi,\Psi,T^\ast))$, it holds that
\begin{equation}\label{conve-2-young}
\|\Phi(\uxi^n,\Psi^n,T^\ast)-\Phi(\uxi,\Psi,T^\ast)\|_{L^\infty(\R^2)} \to 0 \ ,
\end{equation}
as well as
\begin{equation}\label{conve-3-young}
\|\Phi(\uxi^n,\Psi^n,T^\ast)-\Phi(\uxi,\Psi,T^\ast)\|_{\ga;[s,T^\ast]\times \compac} \to 0
\end{equation}
for every compact set $\compac \subset \R$ and every fixed $s\in (0,T^\ast)$.
\end{proposition}

Let us finally turn to our approximated noise $\xi^n=\partial_t\partial_x X^n$, where $X^n$ is given by (\ref{approx-noise}). A quick survey of the arguments in Sections \ref{subsec:main-statements} and \ref{subsec:first-level} regarding the first-level path $\xi$ shows that condition (\ref{converg-assump-young}) is indeed satisfied in this situation, that is, there exists a $\cac^\al_c(\R^2)$-valued process $\xi$ such that almost surely, for every compact set $\compac \subset \R^2$, $\|\xi^n-\xi\|_{\al;\compac} \to 0$. Endowed with this result, the end of the proof of Theorem \ref{main-theo}, point $(i)$, follows the lines of Section \ref{subsec:proof-point-ii}. The only difference lies in the fact that instead of (\ref{reconstr-renorm}), one has here
\begin{eqnarray*}
\crr_{\uxi^n}((\rho_+ \cdot \bff(\bu^n))\star \Xi)(t,y)&=&  \rho_+(t,y)\cdot F(y,u^n(t,y))\cdot \Pi^{\xi^n}_{(t,y)}(\Xi)(t,y)\\
&=&\rho_+(t,y)\cdot F(y,u^n(t,y))\cdot \xi^n(t,y) \ ,
\end{eqnarray*}
which accounts for the absence of a renormalization term in (\ref{eq-base-young}) (in comparison with (\ref{eq-base})).

\

\section{Appendix: multi-level Schauder estimate}

\begin{lemma}\label{lem:schauder}
With the notation of Proposition \ref{prop:convo-sing-kernel}, it holds that
\begin{equation}\label{bound-appendix}
\|\mathcal{K}_{\uxi} (\mathbf{v})\|_{\ga,0;[-12T,12T] \times \compac}\leq Q_{\uxi}\cdot  T^{\frac12 (\al+2)}\cdot \|\mathbf{v} \|_{\ga+\al,\al} \ ,
\end{equation}
for some parameter $\ka >0$.
\end{lemma}

\smallskip

The proof of (\ref{bound-appendix}) relies on a very subtle juggling between the "global" and "local" properties of $\crr_{\uxi}$, that is between the respective bounds (\ref{recon-global}) and (\ref{recon-local}), together with suitable Taylor expansions of the components $K_n$ introduced in Lemma \ref{lem:decompo-noyau}. 


\smallskip

Throughout the proof, we use the notation $Q_{\bv,\uxi}:=\|v\|_{\ga+\al,\al}\cdot Q_{\uxi}$, where $Q_{\uxi}$ represents a generic polynomial expression in $\|\uxi\|_{\al;\compac_0}$, for some suitable compact set $\compac_0\subset \R^2$.

\subsection{Supremum norms} Let us start with the consideration of the three "supremum norms" associated with $\mathcal{K}_{\uxi} (\mathbf{v})$ along (\ref{condition-sup}), which, by the way, will ensure that the components in (\ref{lift-conv}) are indeed well-defined functions. It turns out that each of these three estimates relies on a distinct argument.

\smallskip

We first have to deal with $\big|[K\ast \crr_{\uxi}\bv](x)\big|$ for $x=(x_1,x_2)\in [-12T,12T]\times \compac$. To do so, we appeal to the very same "non-anticipativity" argument as in the proof of Proposition \ref{prop:convo-smooth-kernel}, which here allows us to write
$$(K\ast \crr_{\uxi}\bv)(x)=(K\ast \crr_{\uxi}\bv)(x)-(K\ast \crr_{\uxi}\bv)(0,x_2) \ .$$
Now, with the decomposition $K=\sum_{n\geq 0} K_n$ of Lemma \ref{lem:decompo-noyau} in mind, pick $i\geq 0$ such that $2^{-(i+1)}\leq \|x\|_P^{1/2}\leq 2^{-i}$. On the one hand, using the representation (\ref{decompo-k}) and the "global" regularity (\ref{recon-global}), we get that
\begin{eqnarray}
\big| \sum_{n> i} \big[ (K_n \ast \crr_{\xi} \bv)(x)-(K_n \ast \crr_{\uxi}\bv)(0,x_2)\big] \big|
 &\leq & \sum_{n> i}\big\{ \big| (K_n \ast \crr_{\xi} \bv)(x)\big|+\big| ( K_n \ast \crr_{\uxi}\bv)(0,x_2)\big|\big\}\nonumber\\
&\leq & Q_{\bv,\uxi} \cdot \sum_{n> i} 2^{-n(\al+2)}\nonumber\\
&\leq & Q_{\bv,\uxi}\cdot \|x\|_P^{\frac12 (\al+2)}\ \leq \ Q_{\bv,\uxi}\cdot T^{\frac12 (\al+2)}\ .\label{schau-sup-1}
\end{eqnarray}
On the other hand, we can of course write
$$(K_n \ast \crr_{\xi} \bv)(x)-(K_n \ast \crr_{\uxi}\bv)(0,x_2)=x_1 \cdot \int_0^1 dr \, [(D^{(1,0)}K_n) \ast \crr_{\uxi }\bv ](rx_1,x_2) \ ,$$
and so, still with the help of (\ref{decompo-k}) and (\ref{recon-global}),
\begin{eqnarray}\label{schau-sup-2}
\big| \sum_{0\leq n\leq i} \big[ (K_n \ast \crr_{\xi} \bv)(x)-(K_n \ast \crr_{\uxi}\bv)(0,x_2)\big] \big| &\leq &  Q_{\bv,\uxi} \cdot |x_1| \cdot \sum_{0\leq n\leq i} 2^{-n\al} \nonumber\\
&\leq & Q_{\bv,\uxi} \cdot |x_1| \cdot \|x\|_P^{\frac{\al}{2}} \ \leq \ Q_{\bv,\uxi} \cdot T^{\frac12(\al+2)} \ .
\end{eqnarray}
Combining (\ref{schau-sup-1}) and (\ref{schau-sup-2}) gives the desired bound.

\smallskip

The bound for the second term of the supremum in (\ref{condition-sup}) follows immediately from the definition of the space $\cd^{\ga+\al,\al}(\uxi)$. Indeed, for any $x\in [-12T,12T]\times \compac$,
$$|\bv^0(x)| \cdot \|x\|_P^{\al+2}=|\struc_\al(\bv(x))| \cdot \|x\|_P^{\al+2} \lesssim \|\bv\|_{\ga+\al,\al} \cdot T^{\al+2} \ .$$

\smallskip

Finally, regarding the third term of the supremum in (\ref{condition-sup}), pick $i\geq 0$ such that $2^{-(i+1)}\leq \frac14 \|x\|_P \leq 2^{-i}$ (for some fixed $x\in [-12T,12T]\times \compac$) and decompose, along the same pattern as above,
\begin{eqnarray}
\lefteqn{[(D^{(0,1)} K)\ast \{\crr_{\uxi}\bv -\bv^0(x)\cdot \xi\}](x)}\nonumber\\
&=&\sum_{0\leq n\leq i} [(D^{(0,1)} K_n)\ast \{\crr_{\uxi}\bv -\bv^0(x)\cdot \xi\}](x)+\sum_{n>i} [(D^{(0,1)} K_n)\ast \{\crr_{\uxi}\bv -\bv^0(x)\cdot \xi\}](x)\nonumber\\
&=:&I_x+II_x \ .\label{thir-ter-sup}
\end{eqnarray}
For $n\leq i$, we can invoke (\ref{recon-global}) and use the regularity of $\xi$ to derive that
\begin{eqnarray*}
\|x\|_P \cdot |I_x| &\leq & \|x\|_P \cdot \sum_{0\leq n\leq i} \big\{ |[(D^{(0,1)}K_n) \ast \crr_{\uxi} \bv](x)| +|\bv^0(x)| \cdot |[(D^{(0,1)}K_n) \ast \xi](x)|\big\}\\
&\leq & Q_{\bv,\uxi} \cdot \|x\|_P \cdot \sum_{0\leq n\leq i} 2^{-n(1+\al)}\ \leq \  Q_{\bv,\uxi} \cdot \|x\|_P^{1+ (1+\al)} \ \leq \  Q_{\bv,\uxi} \cdot T^{\al+2} \ .
\end{eqnarray*}
For $n>i$, the estimate appeals this time to the "local" property (\ref{recon-local}) of $\crr_{\uxi}$. Write first $II_x$ as
\begin{multline*}
II_x=\\
\sum_{n>i}[(D^{(0,1)} K_n)\ast \{\crr_{\uxi}\bv -\Pi^{\uxi}_x(\bv(x))\}](x)+\sum_{n>i}[(D^{(0,1)} K_n)\ast \{\Pi^{\uxi}_x(\bv(x)) -\bv^0(x)\cdot \xi\}](x)=:II^1_x+II^2_x \ .
\end{multline*}
Then, due to (\ref{recon-local}), it holds that
$$
\|x\|_P \cdot |II^1_x| \leq  Q_{\bv,\uxi} \cdot \|x\|_P \cdot \|x\|_P^{-\ga} \cdot \sum_{n>i} 2^{-n(1+\al+\ga)} \leq Q_{\bv,\uxi} \cdot \|x\|_P^{1-\ga+(1+\al+\ga)} \leq Q_{\bv,\uxi} \cdot T^{\al+2} \ . 
$$
As for $II^2_x$, observe that the difference $\Pi^{\uxi}_x(\bv(x)) -\bv^0(x)\cdot \xi$ reduces to
\begin{equation}\label{decompo:ii-2}
\Pi^{\uxi}_x(\bv(x)) -\bv^0(x)\cdot \xi=\bv^1(x) \cdot \xi^{\2}_x+\bv^2(x) \cdot \Pi_x^{\uxi}(\Xi X_2) \ .
\end{equation}
Then for instance, since $\xi^{\2}\in \pmb{\cac}^{2\al+2}_c(\R^2)$, we get that
\begin{align*}
&\|x\|_P \cdot |\bv^1(x)| \cdot \sum_{n>i}|[(D^{(0,1)}K_n)\ast \xi^{\2}_x](x)|\\
&\leq Q_{\bv,\uxi} \cdot \|x\|_P^{1-(\al+2)} \cdot \sum_{n>i} 2^{-n(2\al+3)}\leq Q_{\bv,\uxi} \cdot \|x\|_P^{1-(\al+2)+ (2\al+3)}\leq Q_{\bv,\uxi}\cdot T^{\al+2} \ .
\end{align*}
The second term derived from (\ref{decompo:ii-2}) can be readily treated along the same lines, which finally completes the estimation of the supremum norms associated with $\mathcal{K}_{\uxi} (\mathbf{v})$.

\subsection{Projection in $\struc_0$} Let us now turn to the bound for the projection 
\begin{equation}\label{projec}
\struc_0\big( (\mathcal{K}_{\uxi} \mathbf{v})(x)-\gga^{\uxi}_{xy}((\mathcal{K}_{\uxi} \mathbf{v})(y)) \big) \ ,
\end{equation}
where, from now on, $x$ and $y$ are fixed elements in $[-12T,12T]\times \compac$ such that $\|x-y\|_\scal\leq \|x;y\|_P$. In fact, easy computations based on the sole definition of the model shows that the quantity (\ref{projec}) can be expanded as
$$\int_{\R^2} [K(x-z)-K(y-z)-(x_2-y_2) \cdot (D^{(0,1)}K)(y-z)] \cdot [\crr_{\uxi}\bv-\bv^0(y)\cdot \xi](dz)$$
In turn, it will appear convenient to decompose the latter integral as the sum of three terms $A,B,C$ of growing complexity:
\begin{equation}
A_{xy}=\int_{\R^2} [(x_1-y_1) \cdot (D^{(1,0)}K)(y-z)+\frac12 (x_2-y_2)^2  \cdot (D^{(0,2)}K)(y-z)] \cdot [\crr_{\uxi}\bv-\Pi^{\uxi}_y(\bv(y))](dz) \ ,
\end{equation}
\begin{equation}
B_{xy}=\int_{\R^2} [K(x-z)-K(y-z)-(x_2-y_2) \cdot (D^{(0,1)}K)(y-z)] \cdot [\Pi^{\uxi}_y(\bv(y))-\bv^0(y)\cdot \xi](dz) \ ,
\end{equation}
and
\begin{multline}\label{def:c}
C_{xy}=\int_{\R^2} [K(x-z)-K(y-z)-(x_2-y_2) \cdot (D^{(0,1)}K)(y-z)-(x_1-y_1) \cdot (D^{(1,0)}K)(y-z)\\
-\frac12 (x_2-y_2)^2  \cdot (D^{(0,2)}K)(y-z)]  \cdot [\crr_{\uxi}\bv-\Pi^{\uxi}_y(\bv(y))](dz) \ .
\end{multline}

\

\noindent
$(i)$ \underline{Estimation of $A_{xy}$} . It is relatively straightforward. Observe that
$$|x_1-y_1|+|x_2-y_2|^2 \lesssim \|x-y\|_{\scal}^2 \lesssim \|x-y\|_{\scal}^\ga \cdot \|x;y\|_P^{2-\ga}\lesssim (\|x-y\|_{\scal}^\ga \cdot \|x;y\|_P^{-\ga}) \cdot \|y\|^{2}_P \ .$$
Then the estimation of
$$\|y\|_P^{2} \cdot\big|\int_{\R^2} (D^{k}K)(y-z) \cdot [\crr_{\uxi}\bv-\Pi^{\uxi}_y(\bv(y))](dz)\big| \ ,$$
for $k\in \{(1,0),(0,2)\}$, can be done with the very same arguments as those we used in order to deal with (\ref{thir-ter-sup}). The procedure leads us to the conclusion that
$$|A_{xy}| \leq  Q_{\bv,\uxi} \cdot (\|x-y\|_{\scal}^\ga \cdot \|x;y\|_P^{-\ga})\cdot T^{\al+2}\ .$$

\
 
\noindent
$(ii)$ \underline{Estimation of $B_{xy}$} . As above, we rely on the expansion
\begin{equation}\label{decompo:b-xy}
\Pi^{\uxi}_y(\bv(y)) -\bv^0(y)\cdot \xi=\bv^1(y) \cdot \xi^{\2}_y+\bv^2(y) \cdot \Pi_y^{\uxi}(\Xi X_2) \ .
\end{equation}
\indent Let us focus on the term
$$\bar{B}_{xy}=\bv^1(y)\cdot \int_{\R^2} [K(x-z)-K(y-z)-(x_2-y_2) \cdot (D^{(0,1)}K)(y-z)] \cdot \xi^{\2}_y(dz) \ ,$$
keeping in mind that the subsequent arguments could similarly apply to the second summand in decomposition (\ref{decompo:b-xy}). Also, denote by $\bar{B}^n_{xy}$ the expression derived from $\bar{B}_{xy}$ by replacing each occurrence of $K$ with $K_n$. Pick now $j\geq 0$ such that $2^{-(j+1)}\leq\frac14  \|x-y\|_\scal \leq 2^{-j}$. 

\smallskip

\underline{$(ii$-$a)$: $n>j$}. We simply bound $\bar{B}^n_{xy}$ with a triangle inequality
\begin{equation}\label{boun-b}
|\bar{B}^n_{xy}| \leq |\bv^1(y)| \cdot \big\{ |[\xi^{\2}_y \ast K_n](x)|+|[\xi^{\2}_y \ast K_n](y)|+|x_2-y_2| \cdot | [\xi^{\2}_y \ast (D^{(0,1)}K_n)](y)|\big\} \ .
\end{equation}
The estimation of the last two terms follows immediately from the regularity of $\xi^{\2}$ and the representation (\ref{decompo-k}). As for $|[\xi^{\2}_y \ast K_n](x)|$, we must let the $K$-Chen relation come into the picture. Write indeed
\begin{eqnarray*}
| [\xi^{\2}_y \ast K_n](x)|&=& | [\xi^{\2}_x \ast K_n](x)+[(K\ast \xi)(x)-(K\ast \xi)(y)] \cdot [\xi\ast K_n](x)| \\
&\leq &| [\xi^{\2}_x \ast K_n](x)|+|[(K\ast \xi)(y)-(K\ast \xi)(x)]| \cdot | [\xi\ast K_n](x)| \ ,
\end{eqnarray*}
so that, going back to (\ref{boun-b}),
\begin{eqnarray*}
\sum_{n>j} |\bar{B}^n_{xy}| &\leq & Q_{\bv,\uxi} \cdot \|y\|_P^{-(\al+2)} \cdot \sum_{n>j} \{2^{-n(2\al+4)}+\|x-y\|_\scal^{\al+2} \cdot 2^{-n(\al+2)}+\|x-y\|_\scal \cdot 2^{-n(2\al+3)} \}\\
&\leq & Q_{\bv,\uxi}\cdot \|y\|_P^{-(\al+2)} \cdot \|x-y\|_\scal^{\ga} \ \leq \  Q_{\bv,\uxi}\cdot ( \|x-y\|_\scal^\ga\cdot \|x;y\|_P^{-\ga}) \cdot T^{\al+2} \ .
\end{eqnarray*}

\smallskip

\underline{$(ii$-$b)$: $n\leq j$}. We expand $\bar{B}^n_{xy}$ as $\bar{B}^n_{xy}=\bar{B}^{n,1}_{xy}+\bar{B}^{n,2}_{xy}$, with
$$\bar{B}^{n,1}_{xy}=\bv^1(y) \cdot (x_1-y_1) \cdot \int_0^1 dr \, [\xi^{\2}_y \ast (D^{(1,0)}K_n)](y_1+(x_1-y_1),x_2) \ ,$$
$$\bar{B}^{n,2}_{xy}=\frac12 \bv^1(y) \cdot (x_2-y_2)^2 \cdot \int_0^1 dr \, [\xi^{\2}_y \ast (D^{(0,2)}K_n)](y_1,y_2+r(x_2-y_2)) \ .$$
Let us focus on $\bar{B}^{n,1}_{xy}$. Fix $w=(y_1+r(x_1-y_1),x_2)$ ($r\in [0,1]$) and note that $\|w-y\|_\scal \leq \|x-y\|_\scal$. Then, using the $K$-Chen relation, we get that
\begin{eqnarray*}
\lefteqn{| [\xi^{\2}_y \ast (D^{(1,0)}K_n)](w)\big|}\\
 &\leq &| [\xi^{\2}_w \ast (D^{(1,0)}K_n)](w)|+|[(K\ast \xi)(w)-(K\ast \xi)(y)]| \cdot | [\xi\ast (D^{(1,0)}K_n)](w)| \\
&\leq& Q_{\uxi} \cdot \{2^{-n(2\al+2)}+\|x-y\|_\scal \cdot 2^{-n\al}\} \ ,
\end{eqnarray*}
and hence
\begin{eqnarray*}
\sum_{0\leq n\leq j} |\bar{B}^{n,1}_{xy} | &\leq & Q_{\bv,\uxi} \cdot \|y\|_P^{-(\al+2)} \cdot \|x-y\|_\scal^2 \cdot \sum_{n\leq j}  \{2^{-n(2\al+2)}+\|x-y\|_\scal \cdot 2^{-n\al}\}  \\
&\leq & Q_{\bv,\uxi}\cdot ( \|x-y\|_\scal^\ga\cdot \|x;y\|_P^{-\ga}) \cdot T^{\al+2} \ .
\end{eqnarray*}
The very same arguments clearly hold for $\bar{B}^{n,2}_{xy}$, which completes the estimation of $\bar{B}_{xy}$.

\

\noindent
$(iii)$ \underline{Estimation of $C_{xy}$} . It is even more tricky, and we are led to introduce two integers $i\leq j$ such that $2^{-(i+1)} \leq \frac14 \|x;y\|_P\leq 2^{-i}$ and $2^{-(j+1)}\leq \frac14 \|x-y\|_\scal \leq 2^{-j}$. As before, denote by $C^n_{xy}$ the expression derived from $C_{xy}$ by replacing each occurrence of $K$ with $K_n$.

\smallskip

\underline{$(iii$-$a)$: $n>j$}. We first estimate $C^n_{xy}$ with a basic triangle inequality
\begin{multline}\label{c-n-gd}
|C^n_{xy}|\leq 
\big| \int_{\R^2} K_n(x-z)  \cdot [\crr_{\uxi}\bv-\Pi^{\uxi}_y(\bv(y))](dz)\big|\\+\big| \int_{\R^2} K_n(y-z)  \cdot [\crr_{\uxi}\bv-\Pi^{\uxi}_y(\bv(y))](dz)\big|+|x_2-y_2| \cdot \big| \int_{\R^2} (D^{(0,1)}K_n(y-z)  \cdot [\crr_{\uxi}\bv-\Pi^{\uxi}_y(\bv(y))](dz)\big|+\ldots
\end{multline}
Using only the local property (\ref{recon-global}) and the representation (\ref{decompo-k}), we immediately deduce a sharp bound on each of these terms except on the first one, for which a slight refinement is needed. Write indeed
\begin{align*}
&\int_{\R^2} K_n(x-z)  \cdot [\crr_{\uxi}\bv-\Pi^{\uxi}_y(\bv(y))](dz)\\
&=\int_{\R^2} K_n(x-z)  \cdot [\crr_{\uxi}\bv-\Pi^{\uxi}_x(\bv(x))](dz)+\int_{\R^2} K_n(x-z)  \cdot [\Pi^{\uxi}_x(\bv(x))-\Pi^{\uxi}_y(\bv(y))](dz)
\end{align*}
and then invoke the relation (\ref{key-relation}) to get that
\begin{eqnarray}
\lefteqn{\big| \int_{\R^2} K_n(x-z)  \cdot [\Pi^{\uxi}_x(\bv(x))-\Pi^{\uxi}_y(\bv(y))](dz) \big|}\nonumber\\
& = & \big| \int_{\R^2} K_n(x-z)  \cdot \big[\Pi^{\uxi}_x\big( \bv(x)-\gga_{xy}^{\uxi}(\bv(y))\big)\big](dz) \big|\nonumber\\
 &\leq & Q_{\uxi} \cdot \sum_{\la \in \{\al,2\al+2,\al+1\}} 2^{-n(\la+2)} \cdot \big|\struc_\la \big( \bv(x)-\gga^{\uxi}_{xy}(\bv(y))\big) \big|\nonumber\\
&\leq & Q_{\uxi,\bv} \cdot \|x;y\|_P^{-\ga} \cdot \big\{ 2^{-n(\al+2)} \|x-y\|_{\scal}^\ga + 2^{-n\ga} \|x-y\|_{\scal}^{\ga-\al-2} + 2^{-n(\al+3)}  \|x-y\|_{\scal}^{\ga-1} \big\}\ .\label{c-n-rule}
\end{eqnarray}
Going back to (\ref{c-n-gd}), we can now assert that
\begin{eqnarray*}
\sum_{n>j} |C^n_{xy}|& \leq & Q_{\bv,\uxi} \cdot \|x;y\|_P^{-\ga} \cdot \sum_{n>j}\Big\{2^{-n(\ga+\al+2)}+ 2^{-n(\al+2)}  \|x-y\|_{\scal}^\ga+ 2^{-n\ga}  \|x-y\|_{\scal}^{\ga-\al-2} \\
& & + 2^{-n(\al+3)}  \|x-y\|_{\scal}^{\ga-1}+2^{-n(\ga+\al+1)} \|x-y\|_\scal +2^{-n(\ga+\al)}\|x-y\|_\scal^2\Big\}\\
&\leq &Q_{\bv,\uxi} \cdot \|x;y\|_P^{-\ga} \cdot \|x-y\|_\scal^{\ga+\al+2}\ \leq \ Q_{\bv,\uxi} \cdot ( \|x-y\|_\scal^{\ga}\cdot \|x;y\|_P^{-\ga} ) \cdot T^{\al+2} \ .
\end{eqnarray*}

\

\underline{$(iii$-$b)$: $n\leq j $}. We need to turn to a sharper control on the expansion of $K_n$ involved in $C^n_{xy}$, that is on the first bracket in (\ref{def:c}) (with $K_n$ instead of $K$). In fact, using basic differential calculus, this expansion can be easily written as the sum of the following three expressions:
\begin{equation}\label{decom-c-n}
(x_1-y_1)\cdot (x_2-y_2)\cdot \int_0^1 dr_1 \, (D^{(1,1)}K_n)(y_1-z_1,y_2+r_1(x_2-y_2)-z_2)\ ,
\end{equation}
\begin{equation}\label{decom-c-n-2}
\frac12\cdot (x_1-y_1)^2 \cdot \int_0^1 dr_1 \int_0^{r_1} dr_2 \, (D^{(2,0)}K_n)(y_1+r_2(x_1-y_1)-z_1,x_2-z_2) \ ,
\end{equation}
\begin{equation}\label{decom-c-n-3}
\frac16\cdot (x_2-y_2)^3\cdot \int_0^1 dr_1\int_0^{r_1} dr_2 \int_0^{r_2}dr_3 \, (D^{(0,3)}K_n)(y_1-z_1,y_2+r_3(x_2-y_2)-z_2)\ .
\end{equation}
We will only focus on the term inherited from (\ref{decom-c-n}), that is on
\begin{equation}\label{c-bar-n}
\bar{C}^{n}_{xy}:=(x_1-y_1)\cdot (x_2-y_2)\cdot \int_0^1 dr \int_{\R^2} \,  (D^{(1,1)}K_n)(y_1-z_1,y_2+r(x_2-y_2)-z_2)\cdot [\crr_{\uxi}\bv-\Pi^{\uxi}_y(\bv(y))](dz) \ ,
\end{equation}
but it should be clear that the subsequent arguments also hold for the terms derived from (\ref{decom-c-n-2})-(\ref{decom-c-n-3}).

\smallskip

So, with the above expression of $\bar{C}^n_{xy}$ in mind, set $w=(y_1,y_2+r(x_2-y_2))$ (with $r\in [0,1]$) and note that $\|w-y\|_\scal \leq \|x-y\|_\scal$, while $\|w\|_P=\|w;y\|_P \geq \|x;y\|_P$. Then write
\begin{multline}\label{c-n-2}
\int_{\R^2} \,  (D^{(1,1)}K_n)(w-z)\cdot [\crr_{\uxi}\bv-\Pi^{\uxi}_y(\bv(y))](dz)\\
=\int_{\R^2} \,  (D^{(1,1)}K_n)(w-z)\cdot [\crr_{\uxi}\bv-\Pi^{\uxi}_w(\bv(w))](dz)+\int_{\R^2} \,  (D^{(1,1)}K_n)(w-z)\cdot [\Pi^{\uxi}_w(\bv(w))-\Pi^{\uxi}_y(\bv(y))](dz)\ ,
\end{multline}
which, injected into (\ref{c-bar-n}), gives a decomposition of $\bar{C}^n_{xy}$ as a sum of two terms $\bar{C}^{n,1}_{xy}$ and $\bar{C}^{n,2}_{xy}$.

\smallskip

With the same argument as in (\ref{c-n-rule}), we obtain that
\begin{multline*}
\big| \int_{\R^2} \,  (D^{(1,1)}K_n)(w-z)\cdot [\Pi^{\uxi}_w(\bv(w))-\Pi^{\uxi}_y(\bv(y))](dz)\big|\\
\leq Q_{\bv,\uxi} \cdot \|x;y\|_P^{-\ga}\cdot \big\{ 2^{-n(\al-1)} \|x-y\|_\scal^\ga+2^{-n(2\al+1)} \|x-y\|_\scal^{\ga-\al-2}+2^{-n\al} \|x-y\|_\scal^{\ga-1}\big\}\ ,
\end{multline*}
which leads us to the conclusion that
\begin{equation*}
\sum_{0\leq n\leq j} |\bar{C}^{n,2}_{xy}|\leq Q_{\bv,\uxi} \cdot \|x-y\|_\scal^3 \cdot \|x;y\|_P^{-\ga}\cdot \|x-y\|_\scal^{\al+\ga-1} \leq  Q_{\bv,\uxi}  \cdot (\|x;y\|_P^{-\ga}\cdot \|x-y\|_\scal^{\ga}) \cdot T^{\al+2} \ .   
\end{equation*}

\smallskip

As far as $\bar{C}^{n,1}_{xy}$ is concerned, consider first the subcase where $i<n\leq j$. Then by (\ref{recon-local}),
$$\big| \int_{\R^2} \,  (D^{(1,1)}K_n)(w-z)\cdot [\crr_{\uxi}\bv-\Pi^{\uxi}_w(\bv(w))](dz) \big|\leq  Q_{\bv,\uxi} \cdot \|x;y\|_P^{-\ga}\cdot 2^{-n(\al+\ga-1)} \ ,$$
and so
\begin{equation*}
\sum_{i< n\leq j} |\bar{C}^{n,1}_{xy}|\leq Q_{\bv,\uxi} \cdot \|x-y\|_\scal^3 \cdot \|x;y\|_P^{-\ga}\cdot \|x-y\|_\scal^{\al+\ga-1} \leq  Q_{\bv,\uxi}  \cdot (\|x;y\|_P^{-\ga}\cdot \|x-y\|_\scal^{\ga}) \cdot T^{\al+2} \ .   
\end{equation*}
For $n\leq i$, using the global property (\ref{recon-global}) and the regularity of $\uxi$, we get that
\begin{align*}
&\big| \int_{\R^2} \,  (D^{(1,1)}K_n)(w-z)\cdot [\crr_{\uxi}\bv-\Pi^{\uxi}_w(\bv(w))](dz) \big|\\
&\leq \big| \int_{\R^2} \,  (D^{(1,1)}K_n)(w-z)\cdot [\crr_{\uxi}\bv](dz) \big|+\big| \int_{\R^2} \,  (D^{(1,1)}K_n)(w-z)\cdot [\Pi^{\uxi}_w(\bv(w))](dz) \big| \ \\
&\leq Q_{\bv,\uxi} \cdot \{2^{-n(\al-1)}+2^{-n(2\al+1)} \|w\|_P^{-(\al+2)}+2^{-n\al}\|w\|_P^{-1}\} \ ,
\end{align*}
and hence
$$\sum_{0\leq n\leq i} |\bar{C}^{n,1}_{xy}|\leq Q_{\bv,\uxi} \cdot \|x-y\|_\scal^3 \cdot \|x;y\|_P^{\al-1}\leq  Q_{\bv,\uxi}  \cdot (\|x;y\|_P^{-\ga}\cdot \|x-y\|_\scal^{\ga}) \cdot T^{\al+2} \ .$$

\subsection{Projection in $\struc_{\al+2}$} The increment reduces to
$$\struc_{\al+2}\big( (\mathcal{K}_{\uxi} \mathbf{v})(x)-\gga^{\uxi}_{xy}((\mathcal{K}_{\uxi} \mathbf{v})(y)) \big)=\bv^0(x)-\bv^0(y) \ ,$$
and we know that by definition of the space $\cd^{\ga+\al,\al}(\uxi)$
\begin{eqnarray}
\lefteqn{|\bv^0(x)-\bv^0(y)|}\nonumber\\
 &\leq & |\bv^0(x)-\bv^0(y)-\bv^1(y) \cdot [(K\ast \xi)(x)-(K\ast \xi)(y)]-\bv^2(y) \cdot (x_2-y_2)|\nonumber\\
& &+|\bv^1(y) \cdot [(K\ast \xi)(x)-(K\ast \xi)(y)]|+|\bv^2(y) \cdot (x_2-y_2)|\nonumber\\
&\leq & Q_{\bv,\uxi} \cdot \{ \|x-y\|_\scal^\ga \cdot \|x;y\|_P^{-\ga}+\|y\|_P^{-(\al+2)}\cdot \|x-y\|_\scal^{\al+2}+\|y\|_P^{-1}\cdot \|x-y\|_\scal \}\label{holder-v-0}\\
&\leq & Q_{\bv,\uxi}\cdot (\|x-y\|_\scal^{\al+2}\cdot \|x;y\|_P^{-\ga}) \cdot T^{\al+2} \ .\nonumber
\end{eqnarray}

\subsection{Projection in $\struc_1$} We decompose the increment as a sum of two terms:
$$\struc_{1}\big( (\mathcal{K}_{\uxi} \mathbf{v})(x)-\gga^{\uxi}_{xy}((\mathcal{K}_{\uxi} \mathbf{v})(y)) \big)=D_{xy}+E_{xy} \ ,$$
with
$$D_{xy}=(x_2-y_2) \cdot \int_{\R^2} (D^{(0,2)}K)(y-z) \cdot \{ \crr_{\uxi}\bv-\Pi^{\uxi}_y(\bv(y))\}(dz) \ ,$$
\begin{multline*}
E_{xy}=\\
\int_{\R^2} (D^{(0,1)}K)(x-z) \cdot \{ \crr_{\uxi}\bv-\bv^0(x) \cdot \xi\}(dz)-\int_{\R^2} (D^{(0,1)}K)(y-z) \cdot \{ \crr_{\uxi}\bv-\bv^0(y) \cdot \xi \}(dz)\\-(x_2-y_2) \cdot \int_{\R^2} (D^{(0,2)}K)(y-z) \cdot \{ \crr_{\uxi}\bv-\Pi^{\uxi}_y(\bv(y))\}(dz) \ .
\end{multline*}

$(iv)$ \underline{Estimation of $D_{xy}$}. Note that
$$|x_2-y_2| \leq \|x-y\|_\scal \leq (\|x-y\|_\scal^{\ga-1} \cdot \|x;y\|_P^{-\ga}) \cdot \|x;y\|_P^{2} \ .$$
Thus, we are exactly in the same position as with the above term $A_{xy}$, which allows us to conclude that
$$|D_{xy}| \leq Q_{\bv,\uxi} \cdot (\|x-y\|_\scal^{\ga-1} \cdot \|x;y\|_P^{-\ga}) \cdot T^{\al+2} \ .$$

$(v)$ \underline{Estimation of $E_{xy}$}. Let us introduce $j\geq 0$ such that $2^{-(j+1)} \leq \frac14 \|x-y\|_\scal \leq 2^{-j}$, and define the notation $E^n_{xy}$ along the same pattern as before.

\smallskip

\underline{$(v$-$a)$: $n>j$}. We start with the basic inequality
\begin{multline*}
|E^n_{xy}|\leq 
\big| \int_{\R^2} (D^{(0,1)}K_n)(x-z) \cdot \{ \crr_{\uxi}\bv-\bv^0(x) \cdot \xi\}(dz)\big|\\ + \big| \int_{\R^2} (D^{(0,1)}K_n)(y-z) \cdot \{ \crr_{\uxi}\bv-\bv^0(y) \cdot \xi \}(dz)\big|\\
+|x_2-y_2| \cdot \big| \int_{\R^2} (D^{(0,2)}K_n)(y-z) \cdot \{ \crr_{\uxi}\bv-\Pi^{\uxi}_y(\bv(y))\}(dz)\big| \ .
\end{multline*}
Then, on the one hand, we have
\begin{eqnarray*}
\lefteqn{\big| \int_{\R^2} (D^{(0,1)}K_n)(x-z) \cdot \{ \crr_{\uxi}\bv-\bv^0(x) \cdot \xi\}(dz)\big|}\\
&\leq & \big| \int_{\R^2} (D^{(0,1)}K_n)(x-z) \cdot \{ \crr_{\uxi}\bv-\Pi^{\uxi}_x(\bv(x))\}(dz)\big|\\
& &+\big| \int_{\R^2} (D^{(0,1)}K_n)(x-z) \cdot \{ \Pi^{\uxi}_x(\bv(x))-\bv^0(x) \cdot \xi\}(dz)\big|\\
&\leq & Q_{\bv,\uxi} \cdot \{2^{-n(\al+\ga+1)}\|x\|_P^{-\ga}+2^{-n(2\al+3)} \|x\|_P^{-(\al+2)}+2^{-n(\al+2)}\|x\|_P^{-1}\} \ ,
\end{eqnarray*}
and on the other hand,
$$|x_2-y_2| \cdot \big| \int_{\R^2} (D^{(0,2)}K_n)(y-z) \cdot \{ \crr_{\uxi}\bv-\Pi^{\uxi}_y(\bv(y))\}(dz)\big| \leq Q_{\bv,\uxi} \cdot \|x-y\|_\scal \cdot 2^{-n(\al+\ga)} \|y\|_P^{-\ga} \ .$$
As a result,
\begin{eqnarray*}
\lefteqn{\sum_{n>j} |E^n_{xy}|}\\
 &\leq & Q_{\bv,\uxi} \cdot \{\|x-y\|_\scal^{\al+\ga+1} \cdot \|x;y\|_P^{-\ga}+\|x-y\|_\scal^{2\al+3} \cdot \|x;y\|_P^{-(\al+2)}+\|x-y\|_\scal^{\al+2} \cdot \|x;y\|_P^{-1} \}\\
&\leq & Q_{\bv,\uxi} \cdot (\|x-y\|_\scal^{\ga-1} \cdot \|x;y\|_P^{-\ga}) \cdot T^{\al+2} \ .
\end{eqnarray*}

\smallskip

\underline{$(v$-$b)$: $n\leq j$}. We decompose $E^n_{xy}$ as the sum of the following three terms:
\begin{multline*}
E^{n,1}_{xy}=\\
\int_{\R^2} [(D^{(0,1)}K_n)(x-z)-(D^{(0,1)}K_n)(y-z)-(x_2-y_2)\cdot (D^{(0,2)}K_n)(y-z)]\cdot [\crr_{\uxi}\bv-\Pi^{\uxi}_y(\bv(y))](dz)\ ,
\end{multline*}
$$E^{n,2}_{xy}=\int_{\R^2} [(D^{(0,1)}K_n)(x-z)-(D^{(0,1)}K_n)(y-z)]\cdot [\Pi^{\uxi}_y(\bv(y))-\bv^0(y) \cdot \xi](dz)\ ,$$
$$E^{n,3}_{xy}=[\bv^0(x)-\bv^0(y)] \cdot \int_{\R^2} (D^{(0,1)}K_n)(x-z) \cdot \xi(dz) \ .$$
In order to estimate $E^{n,1}$, it suffices to follow the lines of the above-described situation $(iii$-$b)$, while for $E^{n,2}$, we just have to copy the arguments of the case $(ii$-$b)$. Putting these strategies together, we easily deduce that
$$\sum_{0\leq n\leq j} \{ |E^{n,1}_{xy}|+|E^{n,2}_{xy}|\} \leq Q_{\bv,\uxi} \cdot (\|x-y\|_\scal^{\ga-1} \cdot \|x;y\|_P^{-\ga}) \cdot  T^{\al+2} \ .$$
As far as $E^{n,3}_{xy}$, we can invoke (\ref{holder-v-0}) to get that
$$|E^{n,3}_{xy}| \leq Q_{\bv,\uxi} \cdot 2^{-n(\al+1)} \cdot \{\|x-y\|_\scal^\ga \cdot \|x;y\|_P^{-\ga}+\|y\|_P^{-(\al+2)}\cdot \|x-y\|_\scal^{\al+2}+\|y\|_P^{-1}\cdot \|x-y\|_\scal\} \ ,$$
and accordingly
$$\sum_{0\leq n\leq j} |E^{n,3}_{xy}| \leq Q_{\bv,\uxi} \cdot (\|x-y\|_\scal^{\ga-1} \cdot \|x;y\|_P^{-\ga}) \cdot T^{\al+2} \ .$$

\end{document}